\documentclass[12pt]{article}

\usepackage[english]{babel}
\usepackage[utf8x]{inputenc}
\usepackage[T1]{fontenc}

\usepackage[a4paper,top=3cm,bottom=2cm,left=3cm,right=3cm,marginparwidth=1.75cm]{geometry}

\usepackage{multirow}
\usepackage{mathrsfs}
\usepackage{amsmath, amsthm, amssymb, amsfonts, bm}
\usepackage{dsfont}
\usepackage{graphicx}
\usepackage{epstopdf}
\usepackage{subfig}
\usepackage{booktabs}
\usepackage{enumitem}
\usepackage{diagbox}
\usepackage{array, boldline, makecell, booktabs}

\usepackage[colorinlistoftodos]{todonotes}
\usepackage{amsthm}
\usepackage{float}
\usepackage{float}
\newtheorem{theorem}{Theorem}[section]

\newtheorem{lemma}{Lemma}[section]
\newtheorem{corollary}{Corollary}[section]
\newtheorem{definition}{Definition}[section]
\newtheorem{remark}[theorem]{Remark}

\numberwithin{equation}{section}
\numberwithin{lemma}{section}

\begin{document}

\title{Deterministic-Statistical Approach for an Inverse Acoustic Source Problem using Multiple Frequency Limited Aperture Data}
\author{Yanfang Liu \thanks{Department of Mathematics, The George Washington University, Washington, DC 20052, U.S.A. ({\it yliu11@gwu.edu})}
\and Zhizhang Wu \thanks{Department of Mathematics, The University of Hong Kong, Pokfulam Road, Hong Kong SAR, China ({\it wuzz@hku.hk}).}
\and Jiguang Sun \thanks{Department of Mathematical Sciences, Michigan Technological University, Houghton, MI 49931, U.S.A. ({\it  jiguangs@mtu.edu}).}
\and Zhiwen Zhang \thanks{Department of Mathematics, The University of Hong Kong, Pokfulam Road, Hong Kong SAR, China ({\it zhangzw@hku.hk}).}}
\maketitle
\begin{abstract}
We propose a deterministic-statistical method for an inverse source problem using multiple frequency limited aperture far field data. The direct sampling method is used to obtain a disc such that it contains the compact support of the source. The Dirichlet eigenfunctions of the disc are used to expand the source function. Then the inverse problem is recast as a statistical inference problem for the expansion coefficients and the Bayesian inversion is employed to reconstruct the coefficients. The stability of the statistical inverse problem with respect to the measured data is justified in the sense of Hellinger distance. A preconditioned Crank-Nicolson (pCN) Metropolis-Hastings (MH) algorithm is implemented to explore the posterior density function of the unknowns. Numerical examples show that the proposed method is effective for both smooth and non-smooth sources given limited-aperture data.
\end{abstract}

{\bf Key words:} inverse source problem, direct sampling method, Bayesian inversion, eigenfunction expansion, limited-aperture data.

\bigskip

{\bf AMS subject classification:} 35R30, 62F15, 65N21

\section{Introduction}
In recent years, the inverse problem of determining an unknown source function has attracted significant attention due to its practical importance in many applications such as the biomedical imaging and the identification of pollution sources \cite{Isakov1990,Arridge1999, ELBadia2002, Davaney2007, Anastasio2007, Eller2009}. The reconstruction of the acoustic source using single frequency data is challenging. Inverse source problems at a fixed frequency do not possess a unique solution due to the existence of non-radiating sources \cite{Bleistein1977, Devaney1982}. For multiple-frequency data, the uniqueness of the inverse source problem is derived in \cite{Eller2009} for a chosen unbounded set of the Dirichlet eigenvalues of the Laplacian using near field data (see also \cite{bao2010multi} for the uniqueness with the measurements taken on a bounded band of frequency). The use of multiple frequency data improves the stability of the inverse source problem \cite{LiYuan2017IPI}. Accordingly, many researchers consider the reconstruction of an extended acoustic source problem using multiple frequency data. Various methods have been proposed  in the last decade including the continuation methods \cite{Bao2011CR, Bao2015SIAMNA}, eigenfunction expansion methods \cite{Eller2009, Zhang2015}, and sampling type methods \cite{Davaney2007, Griesmaier2017SIAMIS,Alzaalig2020IP,Liu2017IP}.

Bayesian statistics is a classical approach for inverse problems \cite{Fitzpatrick1991IP}. Due to the increase of the computational power, Bayesian inversion has been becoming more popular \cite{kaipio2006statistical, Stuart2010, BuiThanh2014, WangMaZheng2015, Yang2020IP}. Recently, focusing on partial data, we combined the deterministic methods and Bayesian inversion to successfully treat several inverse problems including an inverse scattering problem, an inverse source problem, and the reconstruction of moving point sources using limited-aperture data \cite{li2021quality, LiDengSun2020, LiuGuoSun2021}. In particular, we use certain deterministic method to obtain qualitative information of the unknowns. Such information is built into the priors for the Bayesian inversion, which is then used to compute more details of the unknowns. Both the deterministic method and the Bayesian inversion use the same measured data. Numerical results show that such a combination can provide better reconstructions.

In this paper, we propose a deterministic-statistical approach for an inverse source problem using multiple frequency limited aperture data. The direct sampling method is used to find the support of the source. A disc is identified such that the support of the source is contained in the disc. Using the Dirichlet eigenfunctions of the disc (Bessel's functions) as the basis, we expand the source function. These coefficients are the unknowns for the Bayesian inverse problem, whose posterior density function is explored using an M-H (Metropolis-Hastings) MCMC (Markov chain Monte Carlo) algorithm. The conditional mean (CM) is used to represent the solution. Numerical examples show that the proposed approach is effective for limited-aperture data.

The rest of the paper is organized as follows. In Section 2, we introduce the inverse acoustic source problem of interest. Section 3 presents the direct sampling method to reconstruct a disc that contains the support of the source. In Section 4, we first expand the unknown source using the Dirichlet eigenfunctions of the disc and propose a Bayesian approach to reconstruct the expansion coefficients. The proposed method is validated by various numerical examples in Section 5. Finally, we discuss the method and make some conclusions in Section 6.  

\section{The inverse source problem}
Let $\Omega$ be a bounded domain in $\mathbb{R}^2$ with a Lipschitz boundary $\partial \Omega$. We assume that $\mathbb{R}^2 \setminus {\overline{\Omega}}$ is connected. Let $u$ be the outgoing solution to the inhomogeneous Helmholtz equation in $\mathbb{R}^2$:
\begin{equation}\label{Helmholtz1}
\begin{split}
       &\Delta u(x,k)+k^2 u(x,k) =f(x), ~ x=(x_{1},x_{2}) \in \mathbb{R}^2, \\
       &\lim_{r\rightarrow \infty} \sqrt{r} \left(\frac{\partial u}{\partial r}-iku\right)=0, \qquad r=|x|,
\end{split}
\end{equation}
where $k \in K$, $K=[k_a, k_b]$, $0 < k_a < k_b$, is the wavenumber and $f(x)\in L^2(\Omega)$ with $\text{supp}f\subset \Omega$. Note that $k$ is proportional to the frequency.

There exists a unique solution $u$ to \eqref{Helmholtz1} given by (see \cite{colton2013inverse})
\begin{equation}\label{near_field}
u({x},k)=\int_{\Omega}\Phi_k(x,y) f(y)dy,
\end{equation}
where $\Phi_k(x,y)=-\frac{i}{4}H_{0}^{(1)}(k|x-y|)$ is the fundamental solution to the Helmholtz equation and $H_{0}^{(1)}$ denotes the zeroth-order Hankel function of the first kind. Furthermore, $u(x,k)$ has the asymptotic behavior \cite{colton2013inverse}
\begin{equation*}
u(x,k)=\frac{e^{i\frac{\pi}{4}}}{\sqrt{8k\pi}}
\frac{e^{ikr}}{\sqrt{r}}
\left\{u^{\infty}(\hat{x},k)+\mathcal{O}\left(\frac{1}{r}\right)\right\}
\quad \text{as}\; r\rightarrow\infty,
\end{equation*} where $\hat{x}={x}/{|x|}\in\mathbb{S}$, $\mathbb{S}:=\{|\hat{x}| = 1 : \hat{x} \in \mathbb R^2\}$.  The far field pattern $u^{\infty}(\hat{x},k)$ of $u(x,k)$ is given by
\begin{equation}\label{far_field}
u^{\infty}(\hat{x},k)=\int_{\Omega}\Phi^{\infty}_k(\hat{x},y) f(y)dy,
\end{equation}
where
\begin{equation}
\Phi^{\infty}_k(\hat{x},y)=\exp{(-ik\hat{x}\cdot y)}
\end{equation}
is the far field pattern of the fundamental solution $\Phi_k(x,y)$.

We are interested in the inverse source problem of determining the unknown source $f(x)$ from the partial measurement of the far field pattern $u^{\infty}(\hat{x},k)$ prescribed on the unit circle $\mathbb{S}$ for multiple $k$'s, i.e. reconstruct $f(x)$ from $U:=\{u^{\infty}(\hat{x},k)|\hat{x}\in \Gamma,k\in K \}$, where $\Gamma \subset \mathbb{S}$. In practice, the measurement data is usually discrete $u^{\infty}(\hat{x}_i,k_j)$ for
$\hat{x}_i \in \Gamma, i=1, 2, \ldots, I$ and $k_j \in K, j=1, 2, \ldots, J$.

We propose a deterministic-statistical approach to reconstruct the source function in two steps. Firstly, the direct sampling method (DSM) is applied to obtain a disc $\hat{B}$ which contains the compact support of the source function $f(x)$. Secondly, we expand $f(x)$ in terms of the Dirichlet eigenfunctions of $\hat{B}$ and employ the Bayesian statistics to recover the expansion coefficients. Note that, ideally, the disc $\hat{B}$ should be such that $\text{supp}f\subset \hat{B}$ and $\hat{B}\setminus \text{supp}f$ is not too large.

\section{Direct Sampling Method}
The direct sampling method was proposed in \cite{ito2012direct} to reconstruct small scattering objects. It is simple and effective to reconstruct the support of the unknown target (obstacle, inhomogeneous medium, source) and can process limited aperture data. Following \cite{li2021quality,Alzaalig2020IP}, for multiple frequency far field pattern, we employ the direct sampling method to determine a disc such that it contains the compact support of the source. It turns out that the DSM is effective to obtain a disc, which is important for the success of the Bayesian inversion.

Assume that a domain $D$ is known such that $\Omega \subset D$, i.e., the source function $f(x)$ lies inside $D$. Usually, $D$ is the region of interest and is quite large. Let $D$ be covered by a set of uniformly distributed sampling points $S$. For each point $x_{p} \in S$, we define an indicator function
\begin{equation}
    I(x_{p})=\frac{ | \sum _{k_j} \langle u^{\infty}(\hat{x},k_j),\Phi_{k_j}^{\infty}(\hat{x},x_{p}) \rangle_{L^2(\Gamma)} | }{  \sum _{k_j} \| u^{\infty}(\hat{x},k_j) \|_{L^2(\Gamma)}  \|\Phi_{k_j}^{\infty}(\hat{x},x_{p}) \|_{L^2(\Gamma)}  },
\end{equation}
where the inner product $ \langle \cdot,\cdot \rangle_{L^2(\Gamma)}$ is defined as
\begin{equation*}
    \langle u^{\infty}(\hat{x},k_j),\Phi_{k_j}^{\infty}(\hat{x},x_{p}) \rangle_{L^2(\Gamma)}=\int_{L^2(\Gamma)} u^{\infty}(\hat{x},k_j) \bar{\Phi} _{k_j}^{\infty}(\hat{x},x_{p}) ds(\hat{x})
\end{equation*}
and $\bar{\Phi} _{k_j}^{\infty}(\hat{x},x_{p})$ is the conjugate of ${\Phi} _{k_j}^{\infty}(\hat{x},x_{p})$. In the case of discrete data $u^{\infty}(\hat{x}_i,k_j), i=1, \ldots, I, j=1, \ldots, J$, the indicator function becomes
\begin{equation}\label{Ixp}
    I(x_{p})=\frac{ \sum _{j=1}^J |\sum_{i=1}^I u^{\infty}(\hat{x}_i,k_j) \cdot \overline{\Phi_{k_j}^{\infty}(\hat{x}_i,x_{p})}| }{  \sum _{j=1}^J  \sqrt{\sum_{i=1}^I |u^{\infty}(\hat{x}_i,k_j)|^2}  \sqrt{\sum_{i=1}^I|\Phi_{k_j}^{\infty}(\hat{x}_i,x_{p})|^2}}.
\end{equation}

The DSM uses the indicator function to obtain the support of $f(x)$ approximately. It is clear that $I(x_{p}) \in [0,1]$. If $ I(x_{p})$ is small (close to $0$), then the point $x_{p}$ is likely to lie outside the source. On the other hand, if $ I(x_{p})$ is large (close to $1$), $x_{p}$ is likely to lie inside the source. We refer to \cite{Liu2017IP} for some theoretical justification of the indicator function.

Based on the value of the indicator function, we are able to find a subdomain $\hat{B} \subset D$ containing the support of the source such that $I(x_{p})$ is larger than a cutoff value $\gamma$ for $x_p \in \hat{B}$. In particular, we will take $\hat{B}$ as a disc with radius $R$. The radius $R$ is given by
\begin{equation}\label{radius}
    R=\max_{ x_{p}\in D, I(x_{p}) \geq \gamma } \|x_{p}\|,
\end{equation}
The motivation to use a disc $\hat{B}$ is two folds. Firstly, a disc can easily cover the compact support of the source. Secondly, the Dirichlet eigenfunctions for a disc are known. Note that a square/rectangle domain also works.

The algorithm for multiple frequency limited aperture inverse source problems (MFLAISP) is as follows.
\vskip 2mm
{\bf DSM for MFLAISP}
\begin{itemize}
 \item[1.] Collect the data $u^{\infty}(\hat{x}_i,k_j), i=1, \ldots, I, j=1, \ldots, J$ for $x_i \in \Gamma$ and $k_j\in K$.
 \item[2.] Generate sampling points set $S$ for $D$.
 \item[3.] For each $x_p \in S$, compute ${I}(x_p)$ using \eqref{Ixp}.
 \item[4.] Identify a disc $\hat{B}$ using $I(x_p)$ with radius $R$ given by \eqref{radius}.
\end{itemize}
We remark that other deterministic methods such as the orthogonality sampling method and extended sampling method \cite{Griesmaier2011IP, LiuSun2018} can also be used as long as such a method can provide a good prediction of a disc (or a square) that contains the compact support of the source.

\section{Bayesian Inversion}
We expand the source using the Dirichlet eigenfunctions of $\hat{B}$ obtained by the DSM, and use Bayesian inversion to explore the posterior density function of the expansion coefficients. In particular, we shall construct an approximation $f_{BE}$ for the source $f$ in a finite-dimensional subspace spanned by the Dirichlet eigenfunctions of $\hat{B}$.

Let $\|\cdot\|$ be the usual $L^2$-norm.
The Dirichlet eigenvalue problem (see, e.g., \cite{SunZhou2016}) is to find $\lambda$ and nontrivial $w$ such that
\begin{equation} \label{Dirichleteignfunc}
\begin{split}
    -\Delta w &= \lambda w \quad {\text{in}}~ \hat{B}, \\
    w&=0 \quad  \text{on} ~\partial \hat{B}.
\end{split}
\end{equation}
We call $\lambda$ the Dirichlet eigenvalue and $w$ the eigenfunction corresponding to $\lambda$. All the eigenvalues are positive and have no finite point of accumulation. Since the Dirichlet eigenfunctions are associated to an elliptic self-adjoint compact operator on $L^2(\hat{B})$, $\{w_{n}\}_{n=1}^\infty$ forms a complete orthonormal set \cite{Eller2009}.  Consequently, one can expand $f$ as
\begin{equation}
    f(x)= \sum_{n=1}^\infty A_{n}w_{n},
\end{equation}
where the Fourier coefficients are given by \[
A_{n}=\int _{\hat{B}} f(x)w_{n}(x)~ dx.
\]

Using polar coordinate $x=(r \cos(\theta), r\sin(\theta))$, the Dirichlet eigenfunctions of a disc $\hat{B}$ centered at the origin with radius $R$ are given by \cite{Eller2009}
\begin{equation*}
\begin{split}
  & Q_{mn}^{1}(x)=\frac{1}{\sqrt{\pi}RJ_{n+1}(q_{mn})}J_{n}\left(\frac{q_{mn}r}{R}\right)\cos{(n\theta)},\quad m=1,2,3,\cdots, n=0,1,2,\cdots, \\
  & Q_{mn}^{2}(x)=\frac{1}{\sqrt{\pi}RJ_{n+1}(q_{mn})}J_{n}\left(\frac{q_{mn}r}{R}\right)\sin{(n\theta)},\quad m=1,2,3,\cdots, n=1,2,\cdots,
\end{split}
\end{equation*}
where $J_{n}$ is the Bessel function of order $n$ and $q_{mn}$ is the $m$th zero of $J_{n}$. These eigenfunctions satisfy
\begin{equation*}
    \Delta Q_{mn}^{j} + k_{mn}^2 Q_{mn}^{j} =0, \quad j=1,2,
\end{equation*}
with wavenumber $k_{mn}=q_{mn}/R $.

An approximation $f_{BE}$ of $f$ on the disc $\hat{B}$ is given by
\begin{equation}\label{eigenfuncexpansion}
    f_{BE}(x)= \sum_{m=1} ^M \left( \sum_{n=0}^N A_{mn}^{1} Q_{mn}^{1}(x)+\sum_{n=1}^N A_{mn}^{2} Q_{mn}^{2}(x) \right),
\end{equation}
where
\[
A_{mn}^1 =\int_{\hat{B}} f(x) Q_{mn}^1(x) dx , \quad A_{mn}^2 = \int_{\hat{B}} f(x) Q_{mn}^1(x) dx.
\]

 Denote $ H^{s}(\hat{B})$ the Sobolev space of order $s>0$ equipped with the standard norm $\| \cdot \|_s$. Moreover, $ H^{s}_{0}(\hat{B})$ is defined as the closure of $C_{0}^\infty(\hat{B})$ with respect to the norm in $ H^{s}(\hat{B})$.  The property of the Dirichlet eigenfunction expansion is stated in the following lemma \cite{Eller2009}.
\begin{lemma}
Let $f\in H^{s}_{0}(\hat{B})$ with $s>1$. Furthermore, let $\{w_{n}\}_{n=1}^\infty$ be the set of normalized Dirichlet eigenfunctions of $\hat{B}$. There exists a constant $C$ depending only on  $\hat{B}$ such that
\begin{equation}
    \| f- f_{BE}\| \leq C \|f \|_{s} N^{(1-s)/2}.
\end{equation}
\end{lemma}

\begin{remark}
Here we choose a disc $\hat{B}$ since the Dirichlet eigenfunctions are known analytically. One can also use rectangular domains containing the support of the source. If a general domain, e.g., a polygon, is used, the Dirichlet eigenfunctions can be computed using numerical methods such as the finite element methods (see, e.g., \cite{SunZhou2016}).
\end{remark}

Let $A$ be the vector $\{A_{m0}^{1},A_{mn}^{1},A_{mn}^{2}\}_{n=1,m=1}^{N,M }$ and $X$ be the vector space $\mathbb{R}^{(2N+1)M}$. The inverse problem of the reconstruction of the source function becomes the determination of the coefficients $A \in X$ given the measurement $U$.

Based on the eigenfunction expansion \eqref{eigenfuncexpansion}, we employ the Bayesian inversion to reconstruct $A$ for the source $f(x)$ from the measurement data \cite{kaipio2006statistical, Stuart2010}. The statistical model of the inverse source problem can be written as
\begin{equation} \label{statistical model}
   U=\mathcal{F}(A)+\eta,
\end{equation}
where $\mathcal{F}(A)=\int_{\hat{B}}\Phi^{\infty}_k(\hat{x},y) f_{BE}(y)dy$ and $\eta \sim \mathcal{N}(0,\sigma^2 \mathbb{I})$ is the Gaussian noise.

Using Bayes' formula \cite{kaipio2006statistical,Stuart2010}, the posterior density of the random variable $A$ satisfies
\begin{equation} \label{Bayes rule}
        \pi(A|U) \propto \pi(U|A)\pi(A),\\
\end{equation}
where $\propto$ means ``proportional to'', $\pi(A)$ represents the prior density of the unknown $A$, the conditional distribution $\pi(U|A)=\mathcal{N}(U-\mathcal{F}(A),\sigma^2 \mathbb{I})$ is the likelihood function, and the posterior distribution $\pi(A|U)$ is solution to the Bayesian inverse problem. To represent the statistical information of the unknown $A$, point estimators are often used, e.g., the conditional mean (CM)
\[
A_{\text{CM}}=\mathbb{E}(\pi(A|U)).
\]

We now analyze the stability of the Bayesian inverse problem. Define
\[
G(A;U)=\frac{1}{2\sigma^2}\|U-\mathcal{F}(A)\|^2_{L^2(\Gamma)}.
\]
The relationship \eqref{Bayes rule} in terms of measures $\mu^{U}$ and $\mu_{0}$ corresponding to posterior and prior densities can be written as
 \begin{equation}
     \frac{d \mu^{U}}{ d \mu_{0}}(A)=\frac{1}{L(U)}  \exp \left( -G(A;U)  \right),
 \end{equation}
 where $L(U)=\int_{X} \exp \left(-G(A;U)  \right) \mathrm{d} \mu_{0}(A)$ is the normalization constant.

\begin{lemma}
For integer values of $n$, the Bessel function of the first kind $ J_{n}(y)$ can be defined by the Hansen-Bessel Formula \cite{Ito1993}
\begin{equation}\label{Hansen-Bessel Formula}
    J_{n}(y)=\frac{1}{\pi} \int_{0}^{\pi} \cos(y \sin t -nt) dt.
\end{equation}
\end{lemma}

We prove a property of the operator $\mathcal{F}$ following \cite{Stuart2010}.
\begin{lemma} \label{Lemma1}
There exists a constant $C$ such that, for all $A\in X$,
\begin{equation*}\label{F_property 1}
\|\mathcal{F}(A)\|_{L^2(\Gamma)}\leqslant C\|A\|_{1}.
\end{equation*}
\end{lemma}
\begin{proof}
From the Fourier-Bessel expansion \eqref{eigenfuncexpansion} and the definition of $\mathcal{F}(A)$, we have
\begin{equation}
    \begin{split}
        |\mathcal{F}(A)| &= \left| \int_{\hat{B}}\Phi^{\infty}_k(\hat{x},y)  \sum_{n,m} A_{nm} Q_{nm}(y) ~dy  \right|\\
        &\leq \sum_{n,m} |A_{nm}| \left| \int_{\hat{B}}\exp(-ik\hat{x}\cdot y)  Q_{nm}(y) ~dy  \right|.\\
    \end{split}
\end{equation}
For simplicity, we consider $\hat{B}=B(0,R)$, the disc centered at the origin with radius $R$ for the proof. The case for a general $\hat{B}$ is similar.  It is clear that
\begin{equation}\label{F inequality}
        |\mathcal{F}(A)|\leq \sum_{n,m}  \pi R^2  |A_{nm}|  |Q_{nm}(y)|.\\
\end{equation}

According to \eqref{Hansen-Bessel Formula}, for $y \in B(0,R)$, we have that
\begin{equation}
       |J_{n}(y)|= \left|\frac{1}{\pi} \int_{0}^{\pi} \cos(y\sin t -nt) dt\right|\leq 1,
\end{equation}
which implies that
\begin{equation} \label{Dirichlet eigenfunc bound}
   | Q_{nm}(y)| \leq \frac{1}{\sqrt{\pi}RJ_{n+1}(q_{mn})}.
\end{equation}
Combining \eqref{F inequality} and \eqref{Dirichlet eigenfunc bound}, we obtain
\begin{equation}
\begin{split}
     |\mathcal{F}(A)| & \leq\frac{ \sqrt{\pi} R}{J_{n+1}(q_{mn})}  \sum_{n,m}    |A_{nm}|.\\
\end{split}
\end{equation}
Consequently, we have that
\begin{equation}
     \left\|\mathcal{F}(A)\right\|_{L^2(\Gamma)} \leq \frac{ \sqrt{2}\pi R}{J_{n+1}(q_{mn})}  \sum_{n,m}    |A_{nm} |=C \|A\|_{1},
\end{equation}
where $C=\frac{ \sqrt{2}\pi R}{J_{n+1}(q_{mn})}$.
\end{proof}
\begin{corollary}
For all $A_{1},A_{2}\in X$, there exists a constant $C$, such that
\begin{equation*}\label{F_property 2}
\|\mathcal{F}(A_1)-\mathcal{F}(A_2)\|_{L^2(\Gamma)} \leqslant
C \|A_1-A_2\|_1.
\end{equation*}
\end{corollary}

\begin{definition}
The Hellinger distance between two probability measures $\mu_1$ and $\mu_2$ with common reference measure $\nu$ is defined as
\begin{equation*}
    d_{\rm Hell}(\mu_1, \mu_2)=\left( \int \left(\sqrt{\mathrm{d} \mu_{1} / \mathrm{d} \nu}-\sqrt{\mathrm{d} \mu_{2} / \mathrm{d} \nu }\right)^2 ~ \mathrm{d} \nu \right)^{1/2}.
\end{equation*}
\end{definition}

The following theorem states the well-posedness of the Bayesian inverse problem under investigation.
\begin{theorem}
Let $\mu_{0}$ be a Gaussian measure such that $\mu_0(X)=1$ and $\mu^{U} \ll \mu_{0}$. For $U_1$ and $U_2$ with $\max\{\|U_1\|_{L^2(\Gamma)} , \|U_2\|_{L^2(\Gamma)} \}\leq r$, there exists $M=M(r)>0$ such that
\begin{equation*}
   d_{\rm Hell}(\mu_{U_1}, \mu_{U_2}) \leq M \|U_1-U_2\|_{L^2(\Gamma)} .
\end{equation*}
\end{theorem}

\begin{proof}
From
\begin{equation*}
    L(U)=\int_{X} \exp \left( -\frac{1}{2\sigma^2} \|U-\mathcal{F}(A)\|^2_{L^2(\Gamma)} \right) \mathrm{d} \mu_{0}(A),
\end{equation*}
we have that
\begin{equation}\label{Hproof1}
    0\leq L(U)\leq 1.
\end{equation}
Using Lemma~\ref{Lemma1}, we obtain that
\begin{equation}\label{Hproof2}
\begin{split}
   L(U) & \geq \int_{{X}} \exp \left(- \frac{1}{2\sigma^2}~\| U\|^2_{L^2(\Gamma)} -\frac{1}{2\sigma^2}~\| \mathcal{F}(A)\|^2_{L^2(\Gamma)}  \right)  \mathrm{d} \mu_{0}(A) \\
   & \geq \int_{\|A\|_{1}\leq 1} \exp \left(- \frac{1}{2\sigma^2}~\| U\|^2_{L^2(\Gamma)}-\frac{C}{2\sigma^2}~\|A\|_{1}\right) \mathrm{d} \mu_{0}(A) \\
   & = \exp(-M)\mu_{0}\{\|A\|_{1}\leq 1 \} \\
   & > 0
\end{split}
\end{equation}
since $\mu_{0}$ is a Gaussian measure.

Using the mean value theorem and Lemma~\ref{Lemma1}, for $\mu_{0}$, it holds that
\begin{equation} \label{Hproof3}
 \begin{split}
& \,| L(U_1)-L(U_2)| \\
\leq&\, \int_{X} \left|\exp \left( -G(A;U_1)  \right)- \exp \left( -G(A;U_2)  \right) \right|  \mathrm{d} \mu_{0}(A)  \\ \leq& \int_{X} \left|  -G(A;U_1) - ( -G(A;U_2) ) \right|  \mathrm{d} \mu_{0}(A) \\
 = &\,\int_{X} \left|-\frac{1}{2\sigma^2} \|U_1-\mathcal{F}(A)\|^2_{L^2(\Gamma)}+\frac{1}{2\sigma^2} \|U_2-\mathcal{F}(A)\|^2_{L^2(\Gamma)}) \right|  \mathrm{d}\mu_{0}(A) \\
 \leq& \, \int_{X}  \frac{1}{2\sigma^2}\left(\left| \| U_1\|_{L^2(\Gamma)} ^2-\| U_2\|_{L^2(\Gamma)} ^2\right| + 2 \|\mathcal{F}(A)\|_{L^2(\Gamma)}  ~ \| U_1- U_2\|_{L^2(\Gamma)}   \right) \mathrm{d} \mu_{0}(A)\\
 \leq &\, \int_{X}  \frac{1}{2\sigma^2}\left( \| U_1\|_{L^2(\Gamma)} +\| U_2\|_{L^2(\Gamma)}  + 2C ||A||_{1}  \right) \mathrm{d} \mu_{0}(A) \| U_1- U_2\| _{L^2(\Gamma)} \\
      \leq &\, M \|U_1-U_2\|_{L^2(\Gamma)}.
       \end{split}
\end{equation}
From the definition of the Hellinger distance, we have that
\begin{equation} \label{Hproof4}
  \begin{split}
    & \,d_{\rm Hell}^2(\mu_{U_1}, \mu_{U_2}) \\ = & \frac{1}{2} \int_{X} \left\{ \left(\frac{\exp (-G(A;U_{1}))}{L(U_{1})} \right)^{1/2} -\left(\frac{\exp (-G(A;U_{2}))}{L(U_{2})} \right)^{1/2}\right\}^2 \mathrm{d} \mu_{0}(A)  \\
    = &\, \frac{1}{2} \int_{X} \left\{ \left(\frac{\exp (-G(A;U_{1}))}{L(U_{1})} \right)^{1/2} -\left(\frac{\exp (-G(A;U_{2}))}{L(U_{1})} \right)^{1/2} \right. \\
    &\, \left. {} +\left(\frac{\exp (-G(A;U_{2}))}{L(U_{1})} \right)^{1/2} -\left(\frac{\exp (-G(A;U_{2}))}{L(U_2)} \right)^{1/2} \right\}^2 \mathrm{d} \mu_{0}(A) \\
    \leq &\, L(U_{1})^{-1} \int_{X} \left\{ {\exp \left(-\frac{1}{2}G(A;U_{1}) \right)} -{\exp \left(-\frac{1}{2}G(A;U_{2}) \right)} \right\}^2 \mathrm{d} \mu_{0}(A) \\
    & \,+ \left| L(U_{1})^{-1/2}- L(U_{2})^{-1/2} \right|^2 \int_{X}  {\exp (-G(A;U_{2}))}  \mathrm{d} \mu_{0}(A).
  \end{split}
\end{equation}
With the mean value theorem and Lemma~\ref{Lemma1}, it holds that
\begin{equation}\label{Hproof5}
    \begin{split}
       & \,\int_{X} \left\{ {\exp \Big(-\frac{1}{2}G(A;U_{1}) \Big)} -{\exp \Big(-\frac{1}{2}G(A;U_{2}) \Big)} \right\}^2 \mathrm{d} \mu_{0}(A) \\
       \leqslant & \, \int_{X} \Big|\frac{1}{2}G(A;U_{1}) -\frac{1}{2}G(A;U_{2}) \Big|^2 \mathrm{d} \mu_{0}(A) \\
       \leqslant &\, \frac{ 1}{16\sigma^4} \int_{X}  \Big|\|U_{1}-\mathcal{F}(A)\|_{L^2(\Gamma)} ^2 -\|U_{2}-\mathcal{F}(A)\|_{L^2(\Gamma)} ^2\Big|^2 \mathrm{d} \mu_{0}(A) \\
       \leqslant &\, M \|U_{1}-U_{2}\|_{L^2(\Gamma)} ^2.
    \end{split}
\end{equation}
Using the bounds on $L(U_1)$ and $L(U_2)$, we have that
\begin{equation} \label{Hproof6}
    \begin{split}
    \left | L(U_1)^{-1/2}- L(U_2)^{-1/2} \right |^2 &\leqslant M \max \Big(L(U_1)^{-3}, L(U_2)^{-3} \Big) | L(U_{1})- L(U_{2}) |^2\\
     &\leqslant M \|U_{1}-U_{2}\|_{L^2(\Gamma)} ^2.
        \end{split}
\end{equation}
Combining \eqref{Hproof1}-\eqref{Hproof6}, we conclude that
\begin{equation*}
   d_{\rm Hell}(\mu_{U_1}, \mu_{U_2}) \leqslant M\|U_1-U_2\|_{L^2(\Gamma)}.
\end{equation*}
\end{proof}

To explore the posterior probability distribution of the unknown $A$, we employ the preconditioned Crank-Nicolson (pCN) Metropolis-Hastings (MH) algorithm for the Markov chain Monte Carlo (MCMC) method  \cite{Cotter2013}.

\vspace{1.5ex}
{\bf pCN-MH:}
\begin{itemize}
  \item[1.] Set $j \leftarrow 0$ and choose an initial value ${A}^{(0)}$.
  \item[2.] Propose a move according to
  \begin{equation*}
   \tilde{A}^{(j)}=\left(1-{\beta}^2\right)^{1/2}{A}^{(j)}+\beta W_n,\quad W_n\sim \mathcal{N}(0, \mathbb{I}).
  \end{equation*}
  \item[3.] Compute
  \begin{equation*}
 \alpha({A}^{(j)}, \tilde{A}^{(j)})=\min \left\{1, {\exp \left(-G(\tilde{A}^{(j)};U)+G({A}^{(j)};U)\right)}\right\}.
  \end{equation*}
  \item[4.] Draw $\tilde{\alpha}\sim \mathcal{U}(0,1)$. If $\alpha({A}^{(j)}, \tilde{A}^{(j)})\geq \tilde{\alpha}$,
      set ${A}^{(j+1)}=\tilde{A}^{(j)}$.  Else, ${A}^{(j+1)}={A}^{(j)}$.
  \item[5.] When $j=\text{MaxIt}$, the maximum sample size, stop.
  	Otherwise, set $j \leftarrow j+1$ and go to Step 2.
\end{itemize}

\section{Numerical Examples}
In this section, we present some numerical experiments to demonstrate the effectiveness of the proposed deterministic-statistical method.

In all examples, the synthetic far field data is generated by decomposing $\Omega$ into a triangular mesh ${\mathcal T}$ and approximating \eqref{far_field} by
\begin{equation}
    u^{\infty}(\hat{x},k) \approx \sum_{T\in {\mathcal T}} \Phi^{\infty}_k(\hat{x},y_{T}) f(y_{T}) |T|, \quad \hat{x}=(\cos\theta,\sin\theta),
\end{equation}
 where $T \in {\mathcal T}$ is a triangle, $y_T$ is the center of $T$, and $|T|$ denotes the area of $T$. The observation directions $\theta$'s are chosen from the following three apertures:
\begin{equation*}
       \Gamma_1 = 0: \frac{\pi}{26}:2\pi- \frac{\pi}{26}, \qquad
        \Gamma_2 = 0: \frac{\pi}{26}:\pi- \frac{\pi}{26}, \qquad
         \Gamma_3 = 0: \frac{\pi}{26}:\frac{\pi}{2}- \frac{\pi}{26},
\end{equation*}
i.e. $\Gamma_1$ is the full aperture, $\Gamma_2$ is a half of the full aperture and $\Gamma_3$ is a quarter of the full aperture. To ensure the accuracy of the far field data, we use fine meshes with the mesh size $h \approx 0.01$. The perturbed far field measurement is given by
 \begin{equation*}
u^{m}(\hat{x},k):=u^{\infty}(\hat{x},k)+0.03 ( \max_{\hat{x}} \Re ( u^{\infty}(\hat{x},k))+ i  \max_{\hat{x}} \Im ( u^{\infty}(\hat{x},k))),
\end{equation*}
where $\Re $ and $\Im$ represent the real and imaginary part, respectively.

For the DSM, the measurement data is the far field pattern correspond to wavenumbers $K_{1}=1:1:3$. The domain $D$ is the square $[-4,4]^2$, which is uniformly covered by $81 \times 81$ sampling points. The cutoff values for the indicator function of three scenarios $\Gamma_1$, $\Gamma_2$ and $\Gamma_3$ are $\gamma=0.41, 0.64$ and $0.70$, correspondingly. These $\gamma$'s are obtained by trial and error.  In the contour plots of the indicator $I_{x_{p}}$ over the sampling domain $D$, the red dashed line represents the exact boundary of source $f(x)$ and the estimation of the support is the black circle. In fact, we will see that the DSM uses a smaller set of the far field data than the Bayesian inversion does. In general, a satisfactory reconstruction of the disc $\hat{B}$ can be obtained using the data for a few smaller $k$'s.

Once we obtain the approximate disc $\hat{B}$, we choose $N=2,M=5$ in \eqref{eigenfuncexpansion} for the approximation $f_{BE}$ ($25$ terms in total). The measurement for the Bayesian method is corresponding to the wavenumbers $K=1:1:20$. In the MCMC we take $\pi(A)=\mathcal{N}(0,0.01)$ and $\sigma=0.04$ in the likelihood. To compute the posterior distribution of $A$, we apply {\bf pCN-MH} with $\beta=0.001$. A Markov chain of sample size $120,000$ is drawn in the Bayesian inversion, of which the first $20,000$ samples are discarded. The CM is then used as a point estimate for $A$. To evaluate the performance of the reconstruction, we compute both the absolute error (AE)  $\|f-f_{BE}\|_2$ and the relative error (RE) $\frac{\|f_{BE}-f\|_2}{\|f\|_2}$.

\vskip 0.2cm
\textbf{Example 1:} Let
\begin{equation}\label{f1}
f(x)=3Q_{11}(x), \quad x\in B(0,0.9),
\end{equation}
i.e., the source function is a constant multiple of an eigenfunction $Q_{11}$ for $B(0,0.9)$. We first show the performance of the Bayesian inversion when the compact support of $f(x)$ is known exactly, namely, $\hat{B} = B(0,0.9)$. Due to \eqref{f1}, we expect that the CM of the coefficient for $Q_{11}$ is $3$ and the CM's of the other coefficients are zeros in \eqref{eigenfuncexpansion}. Using the Bayesian method for three apertures, $\Gamma_1$, $\Gamma_2$ and $\Gamma_3$, the reconstructions $f_{BE}$ are shown in Fig. \ref{fig:example1fig2}. It can be seen that the samples for the coefficient of $Q_{11}(x)$ accumulate around $3$ and the rest accumulate around $0$ for all three apertures.
\begin{figure}[h!]
\begin{center}
\begin{tabular}{lll}
\resizebox{0.30\textwidth}{!}{\includegraphics{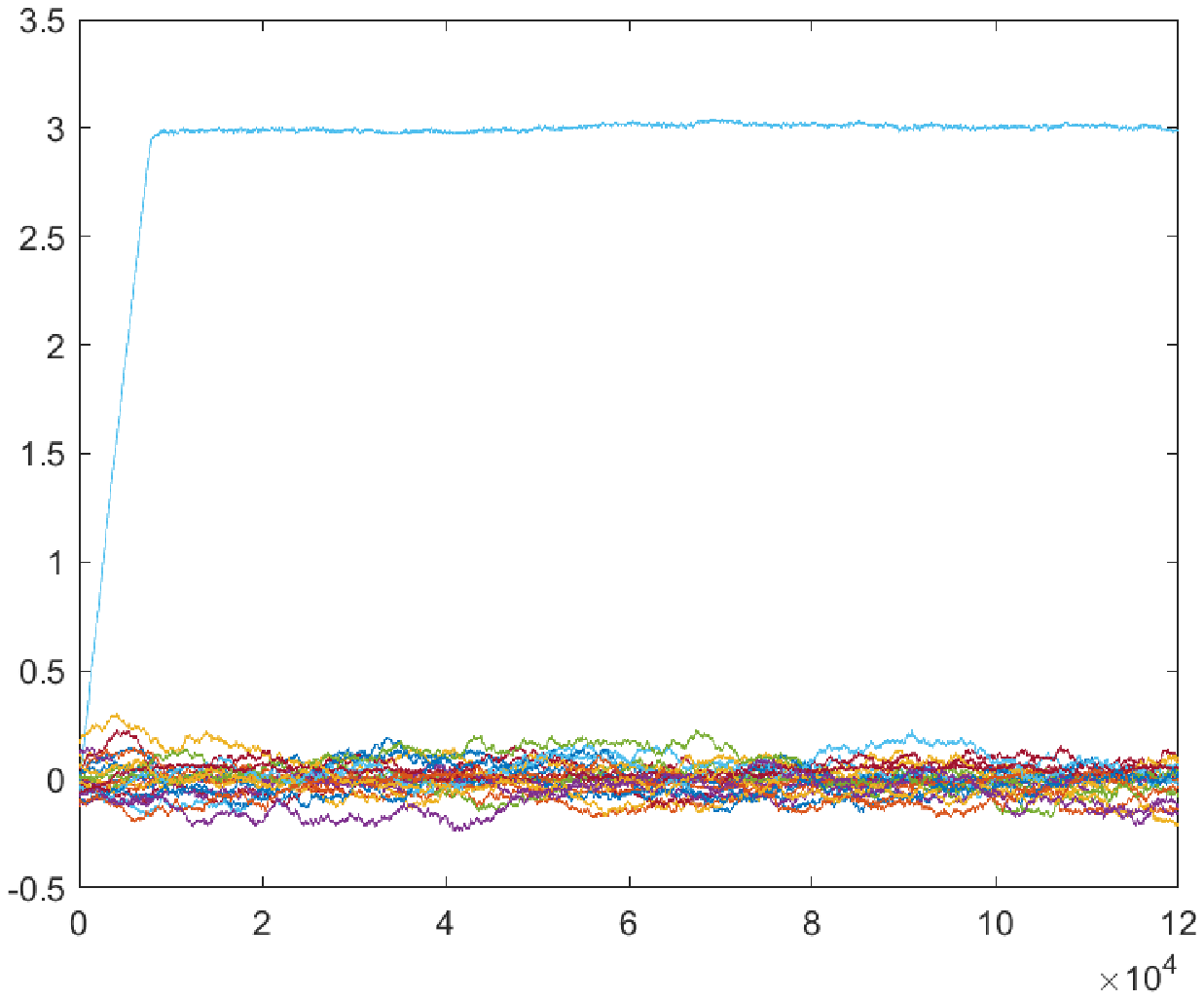}}&
\resizebox{0.30\textwidth}{!}{\includegraphics{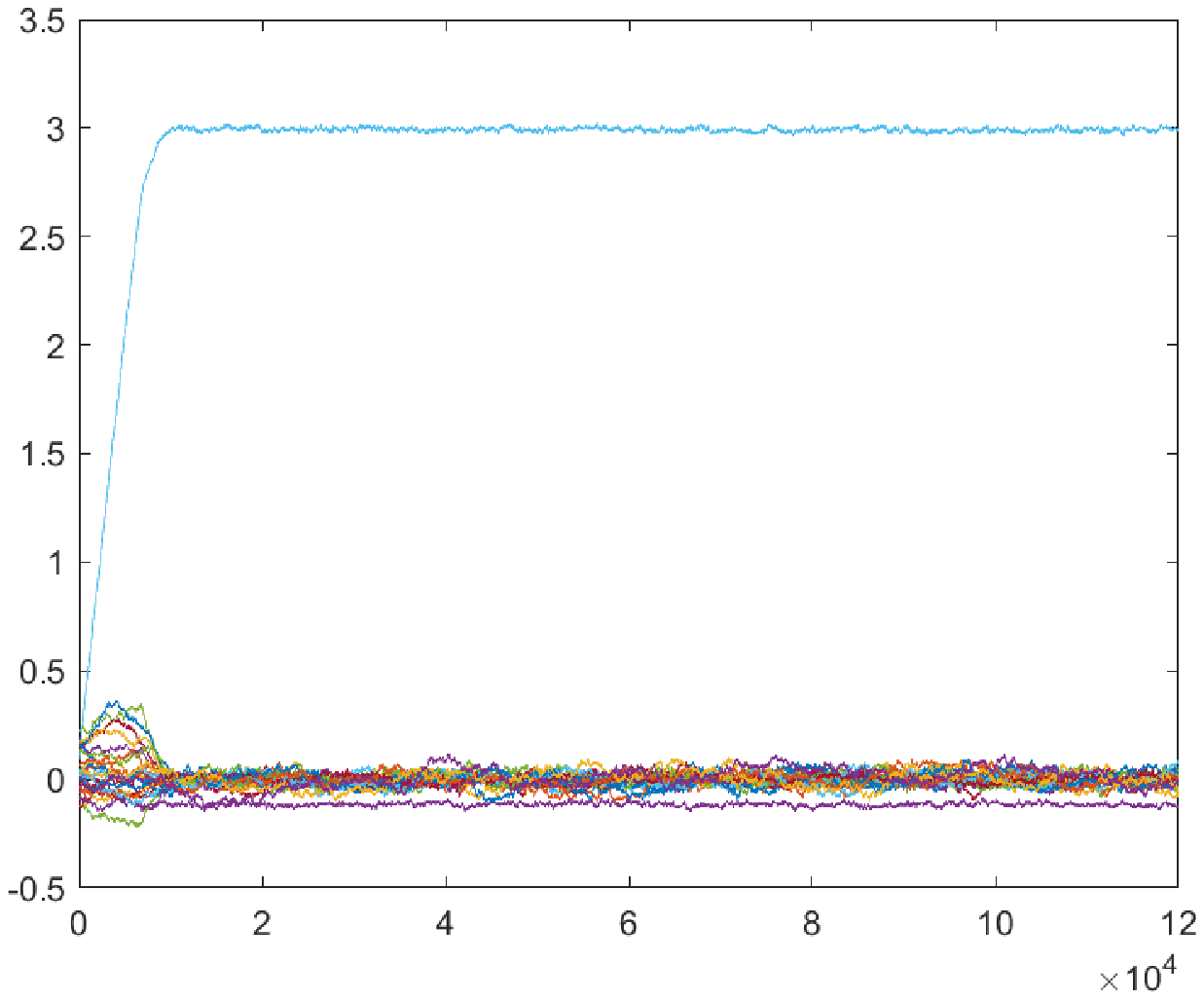}}&
\resizebox{0.30\textwidth}{!}{\includegraphics{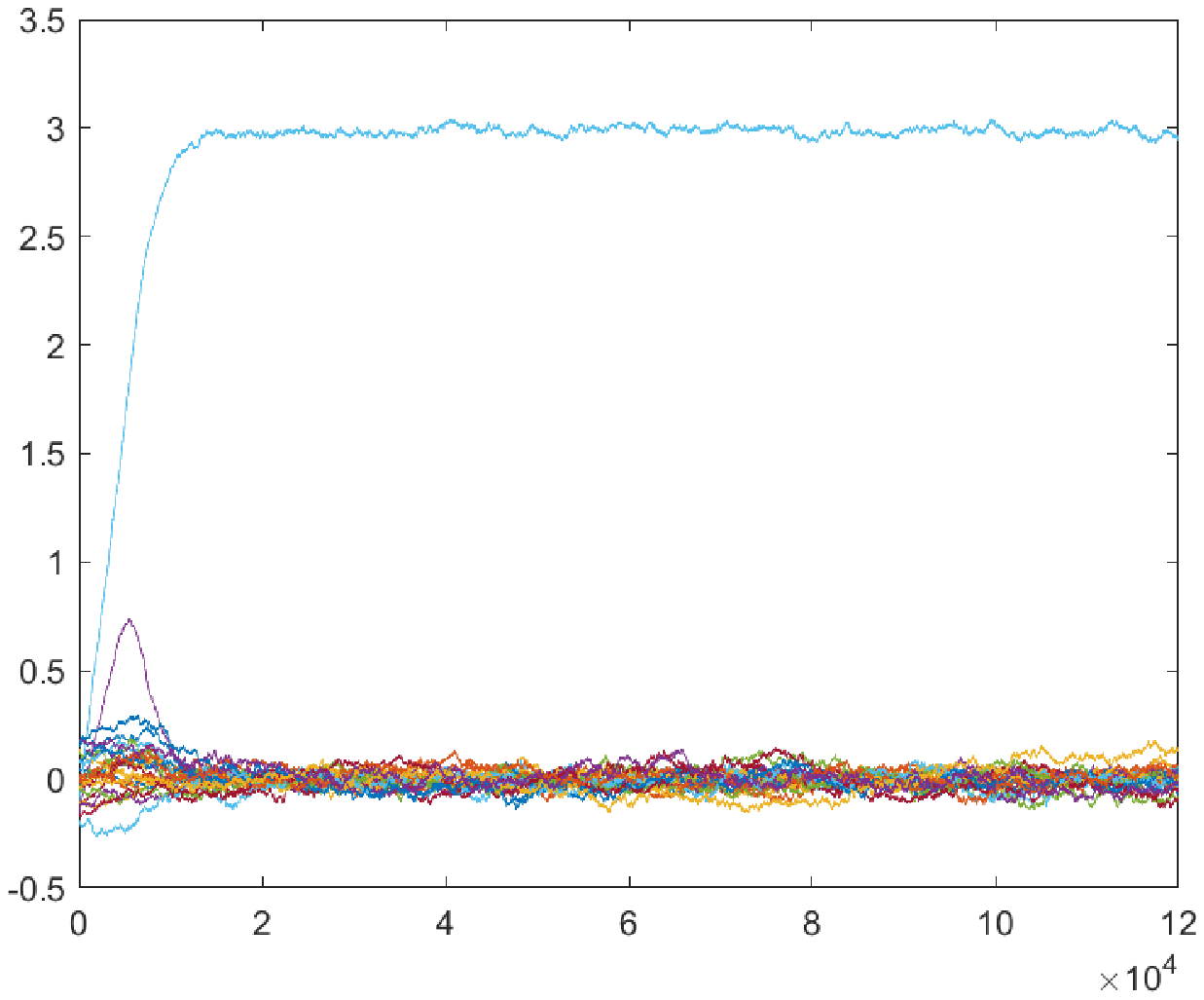}} \\
\resizebox{0.30\textwidth}{!}{\includegraphics{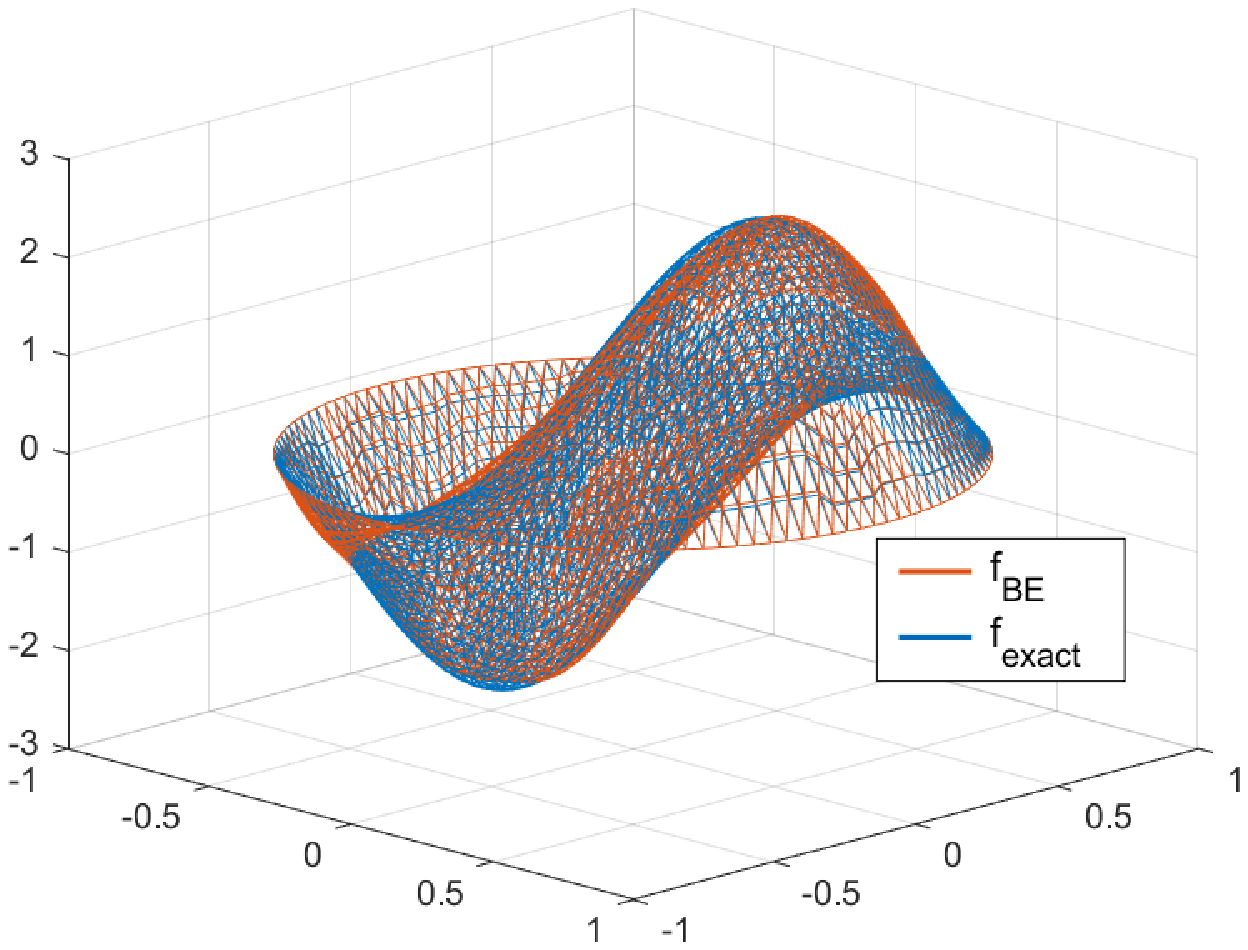}}&
\resizebox{0.30\textwidth}{!}{\includegraphics{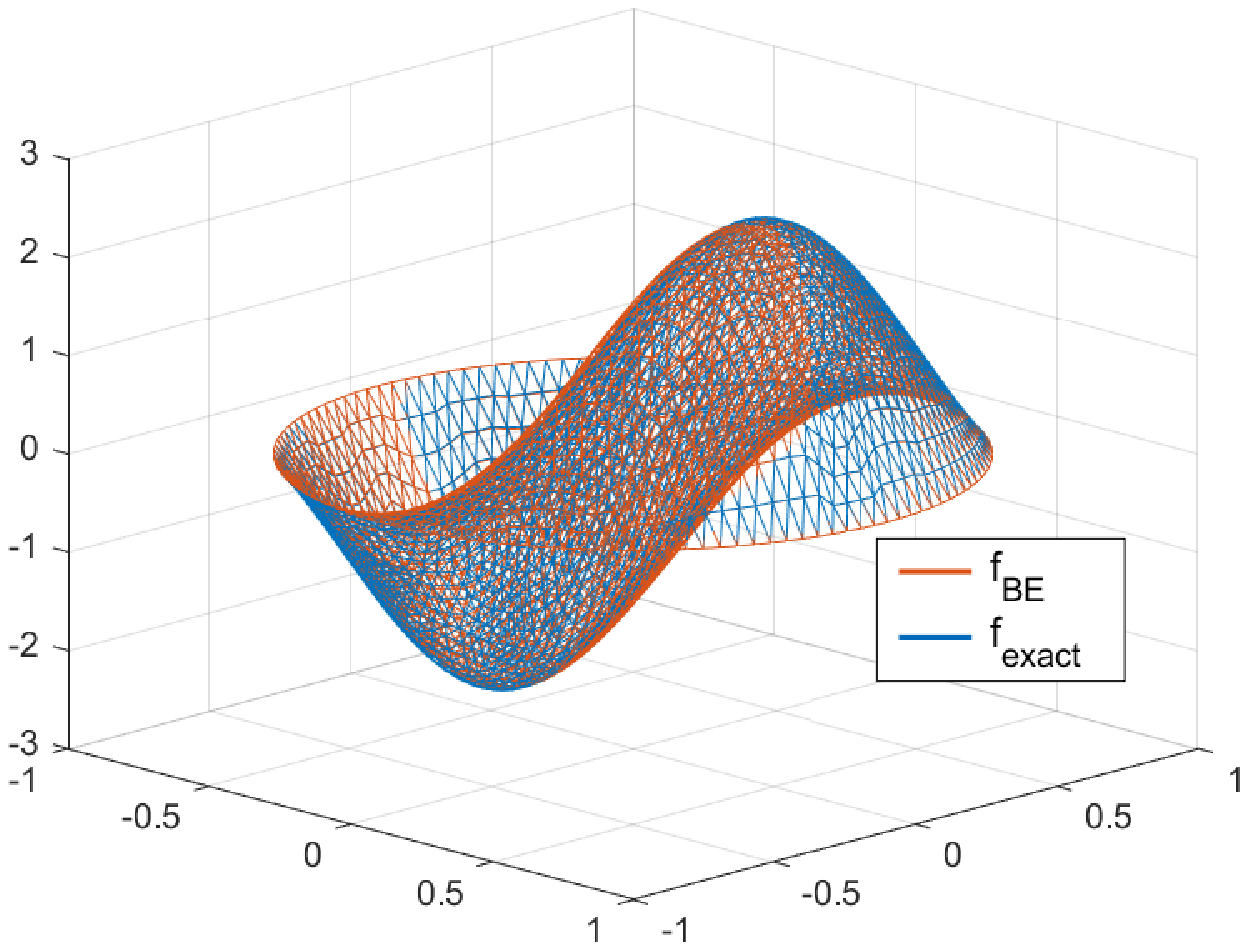}}&
\resizebox{0.30\textwidth}{!}{\includegraphics{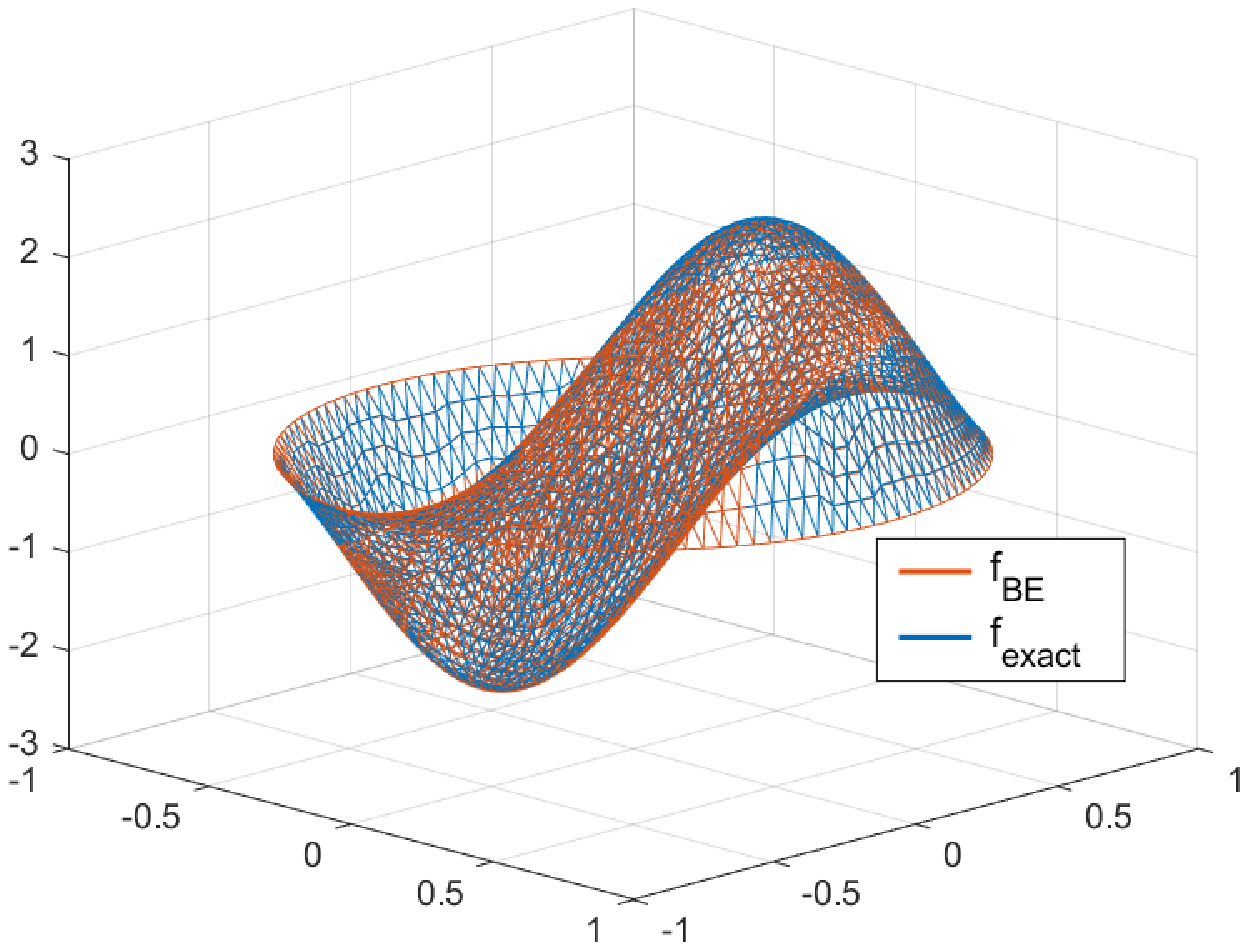}} \\
\end{tabular}
\end{center}
\caption{Example 1 (exact support known). Top row: the histograms of the coefficients for $f_{BE}$ when the support is known exactly. Bottom row: the reconstructed $f_{BE}$ and exact $f$. Left column: $\Gamma_1$. Middle column: $\Gamma_2$. Right column: $\Gamma_3$.}
\label{fig:example1fig2}
\end{figure}

 Next we use the proposed deterministic-statistical method to reconstruct $f(x)$ without the knowledge of its support. The DSM is first used to find a disc $\hat{B}$ containing the support of $f(x)$. For all three apertures $\Gamma_1$, $\Gamma_2$ and $\Gamma_3$, the indicator functions $I(x_{p})$'s and the discs $\hat{B}$'s obtained are shown in the top row of Fig.\ref{fig:example1fig1}.  The associated approximate radii of $\hat{B}$'s are $ 1.3601$, $1.4213 $ and $1.0817$ (see Table~\ref{table1}). All $\hat{B}$'s are close to the exact support, which indicates the effectiveness of the DSM.

In the Bayesian inversion stage, based on the reconstructed $\hat{B}$, we explore the statistical information of the coefficients for $f_{BE}$ using {\bf pCN-MH}. The second row of Fig.~\ref{fig:example1fig1} shows the histograms of the coefficients, which tend to converge. Note that the eigenfunctions of $\hat{B}$ are used and the coefficients for $f_{BE}$ are not zero in general. The exact source function $f$ and the reconstructions $f_{BE}$ are shown in the third row of Fig.~\ref{fig:example1fig1}. The absolute and the relative errors of the reconstructions using the CM's are listed in Table~\ref{table2} (first four columns). It can be seen that all the approximate source functions $f_{BE}$'s are quite close to the exact sources. For all three apertures, the absolute errors are small and  the relative errors are less than 7\%.
\begin{figure}[h!]
\begin{center}
\begin{tabular}{lll}
\resizebox{0.3\textwidth}{!}{\includegraphics{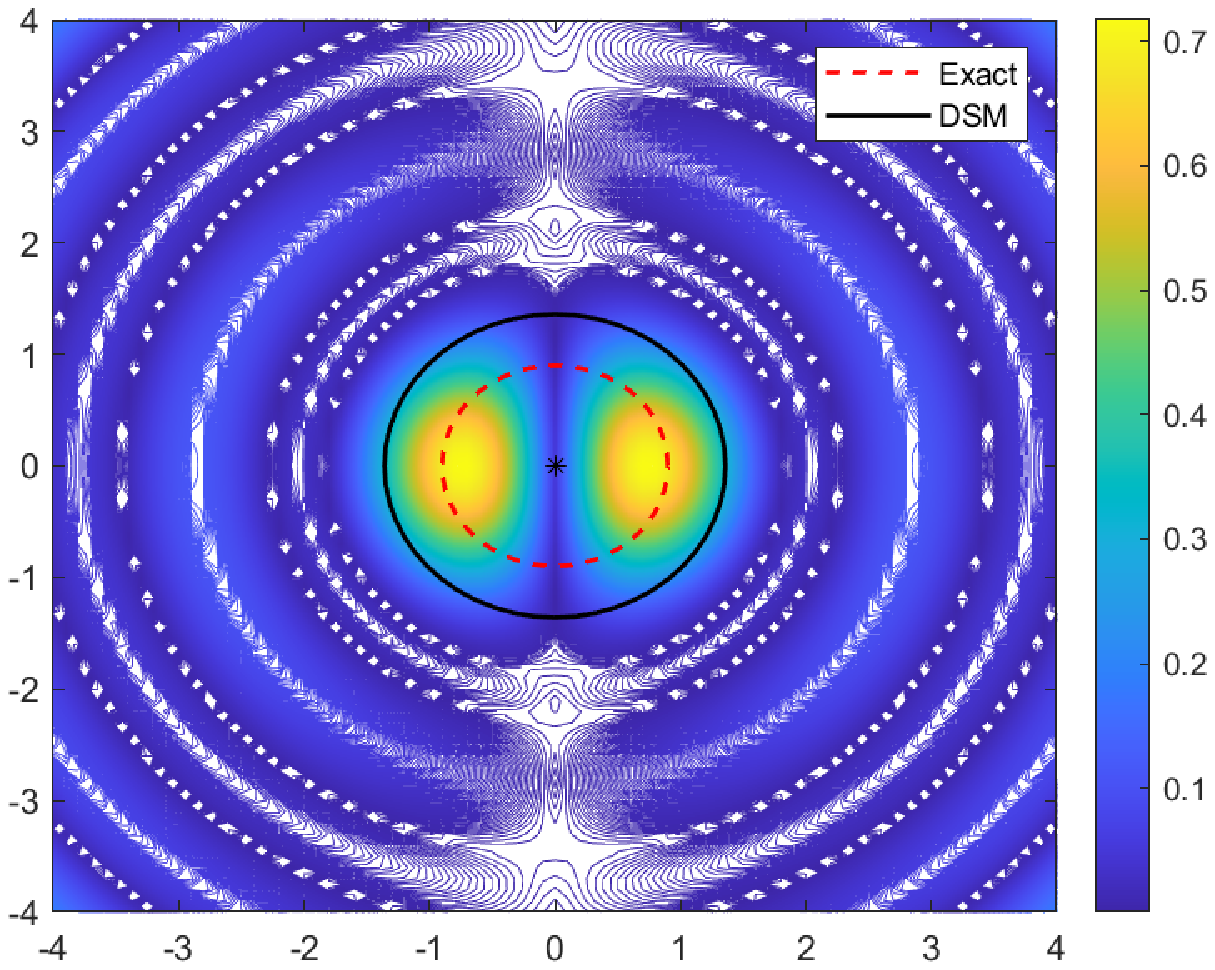}}&
\resizebox{0.3\textwidth}{!}{\includegraphics{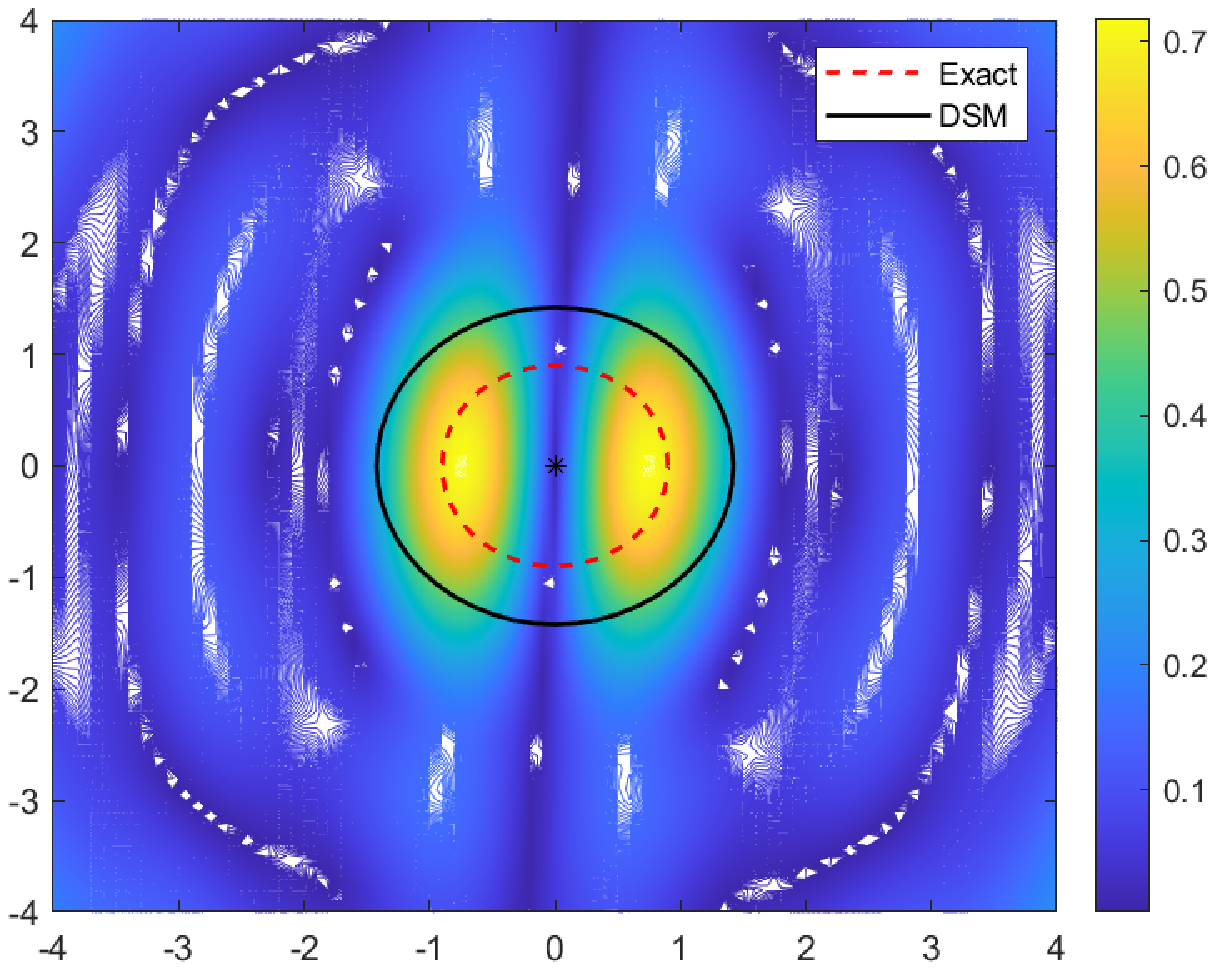}}&
\resizebox{0.3\textwidth}{!}{\includegraphics{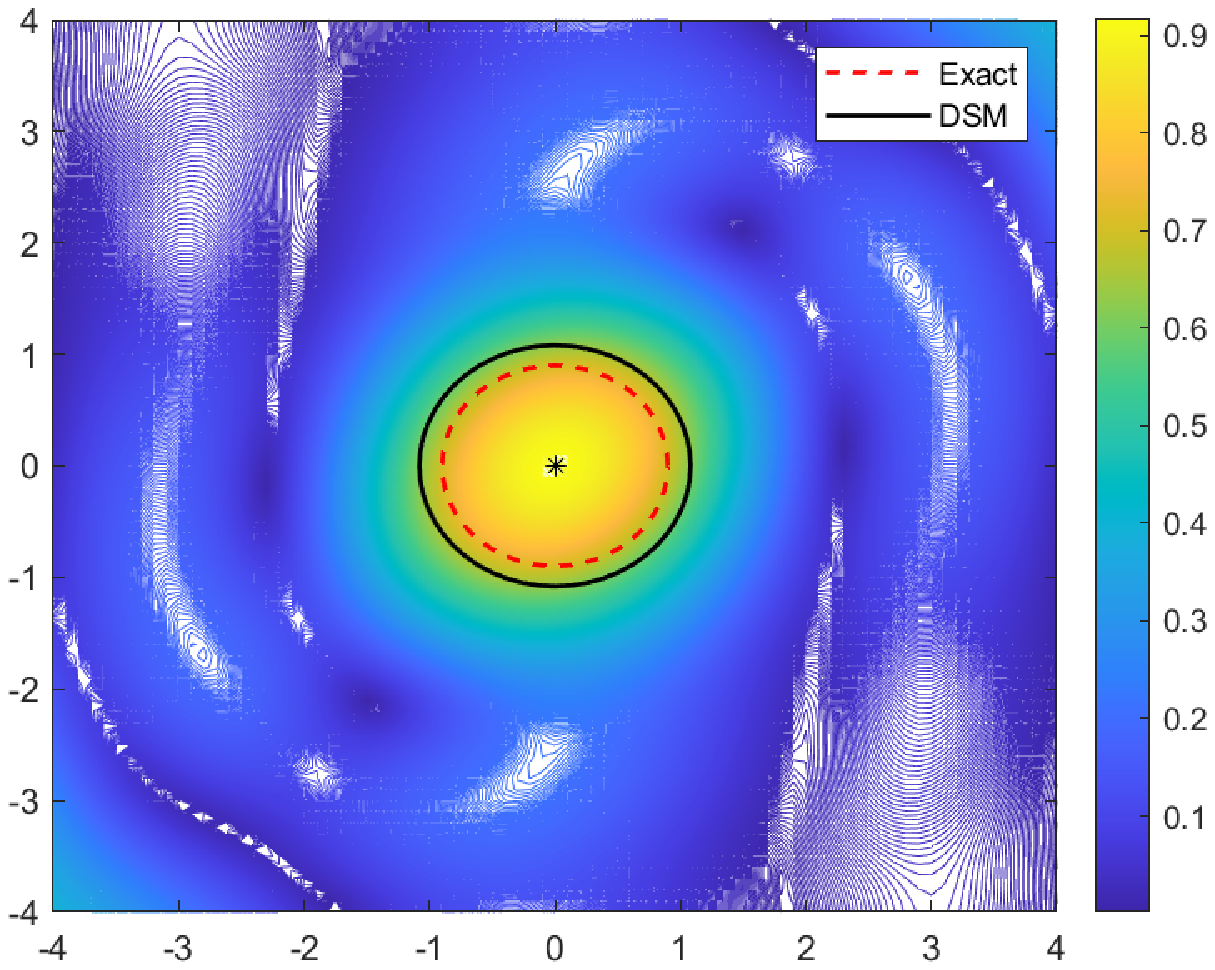}} \\
\resizebox{0.3\textwidth}{!}{\includegraphics{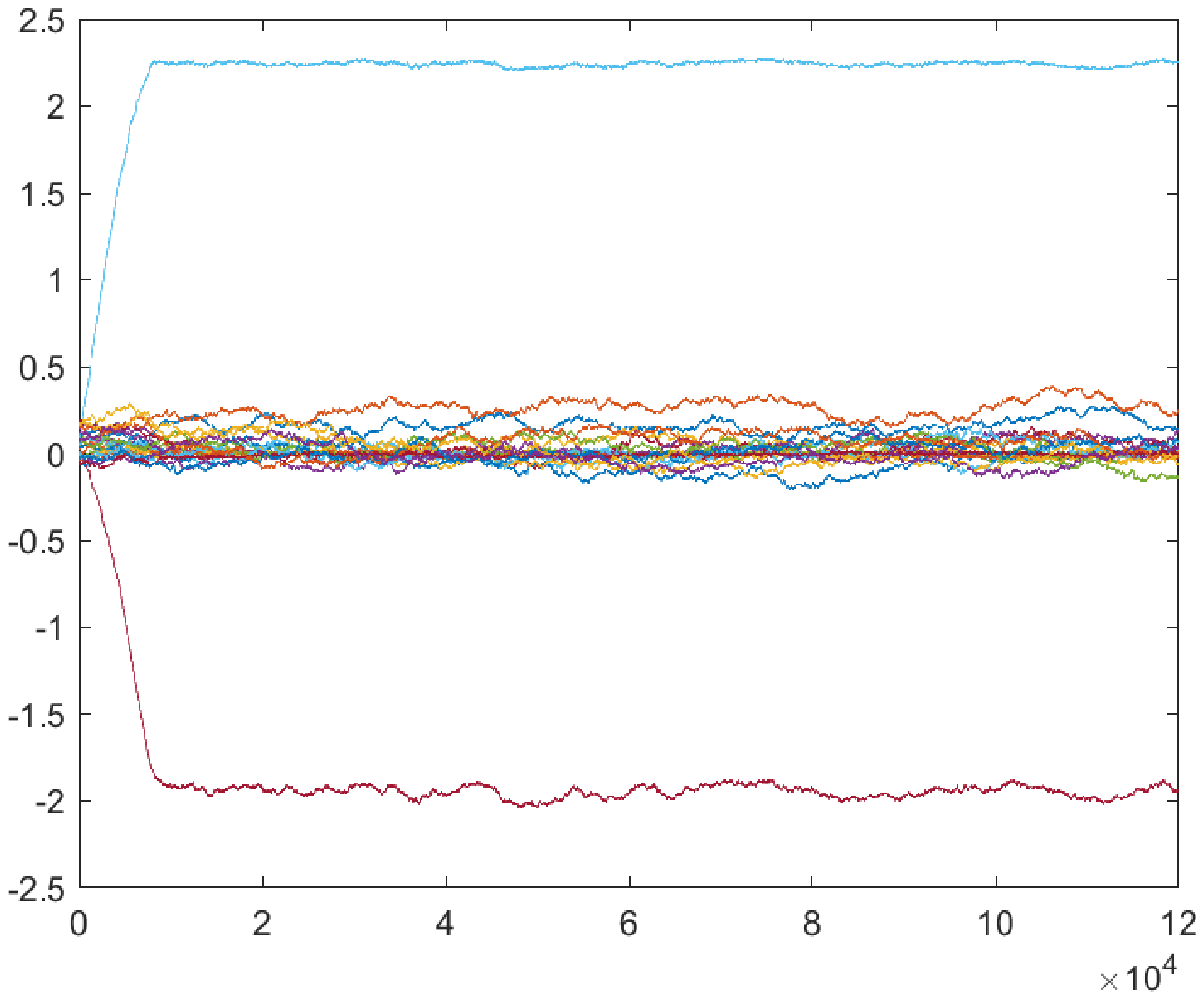}}&
\resizebox{0.3\textwidth}{!}{\includegraphics{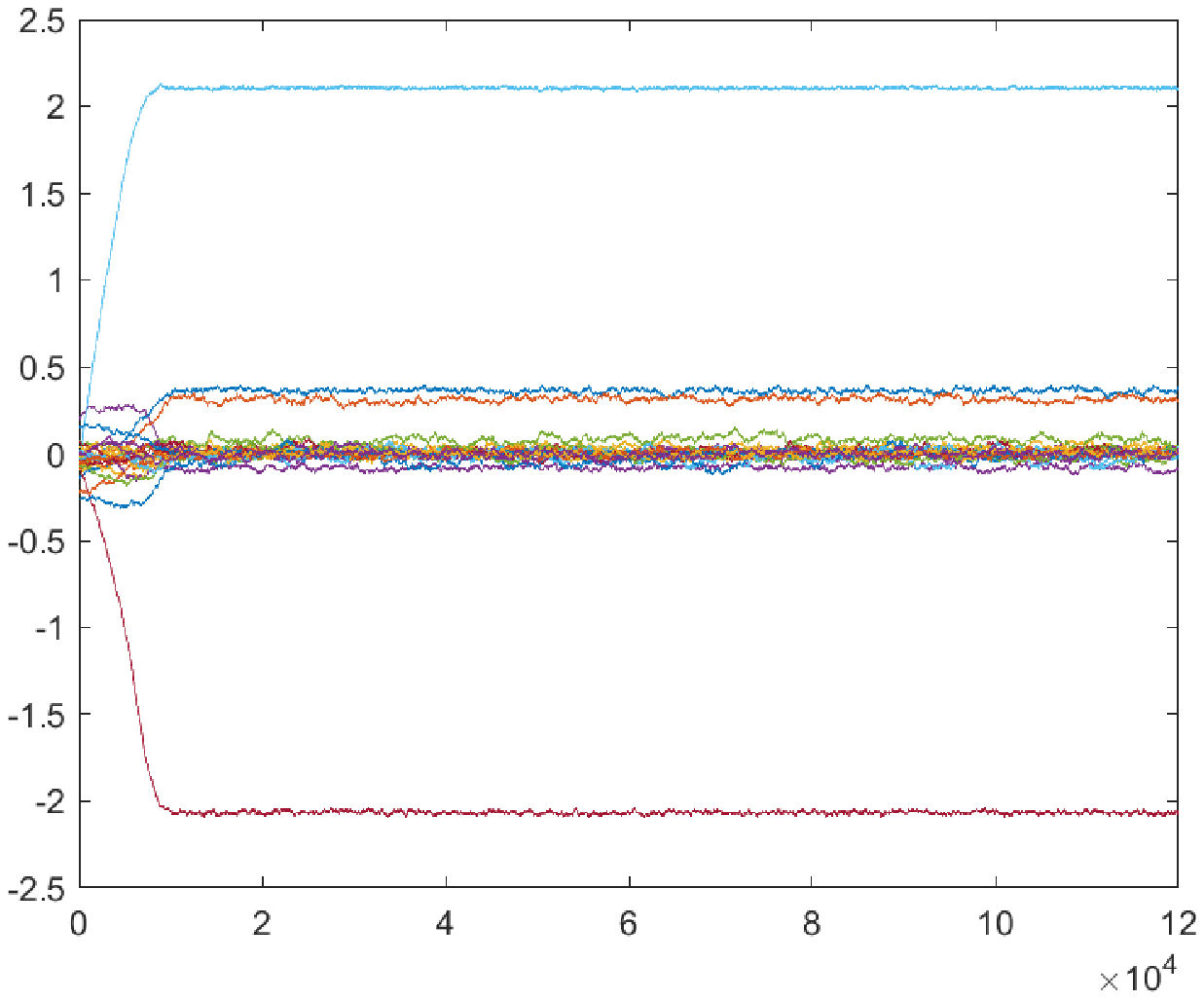}}&
\resizebox{0.3\textwidth}{!}{\includegraphics{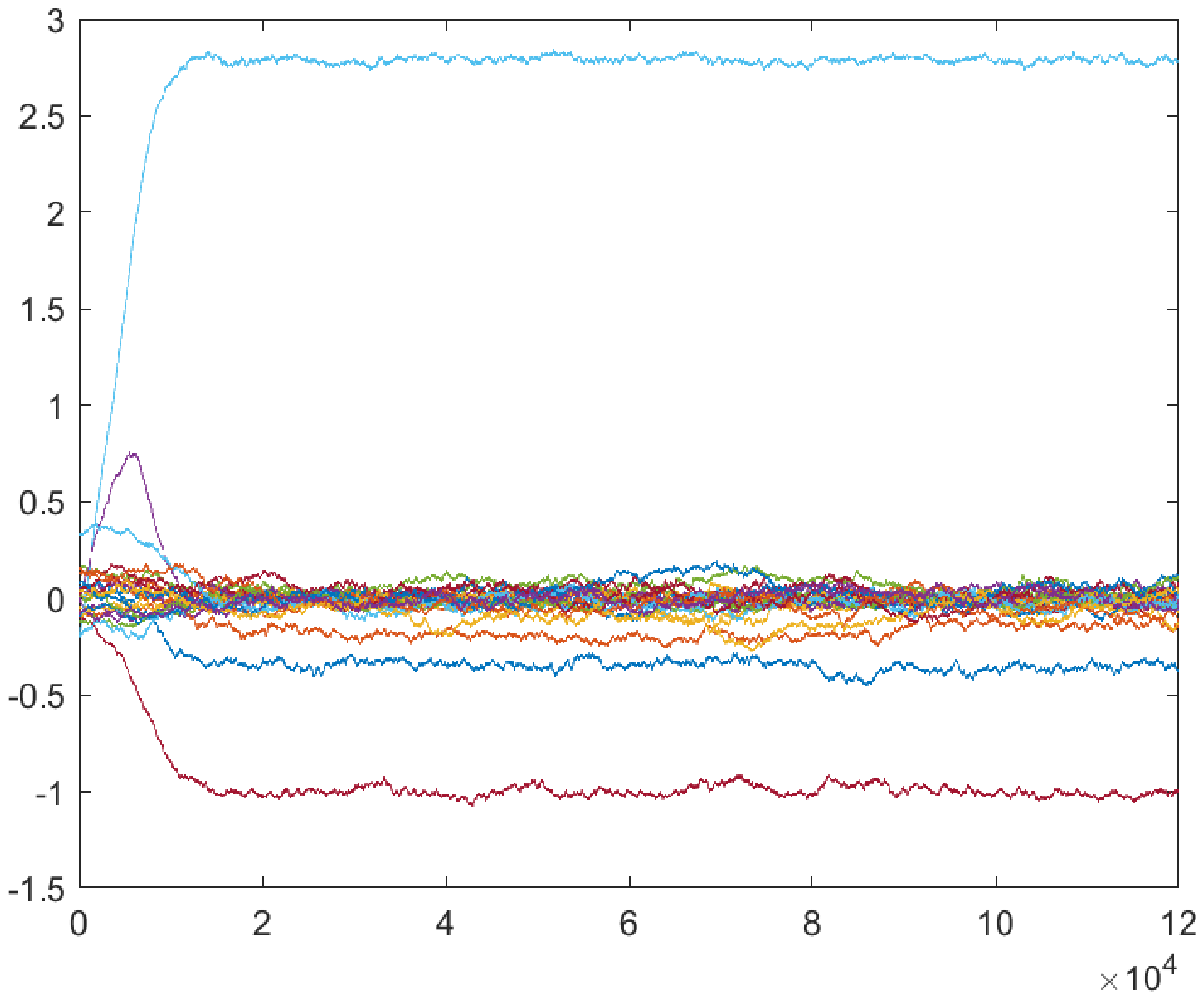}} \\
\resizebox{0.3\textwidth}{!}{\includegraphics{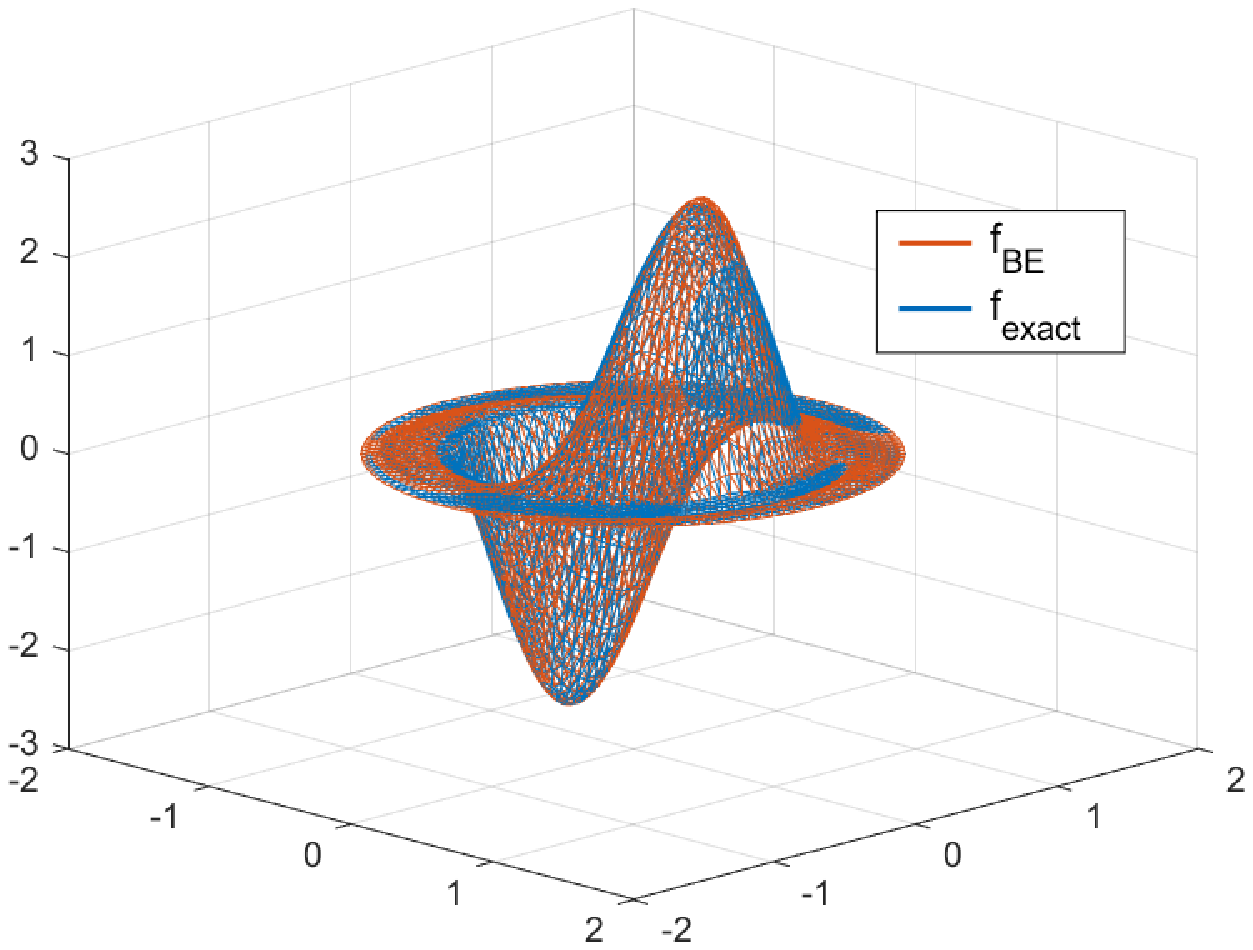}}&
\resizebox{0.3\textwidth}{!}{\includegraphics{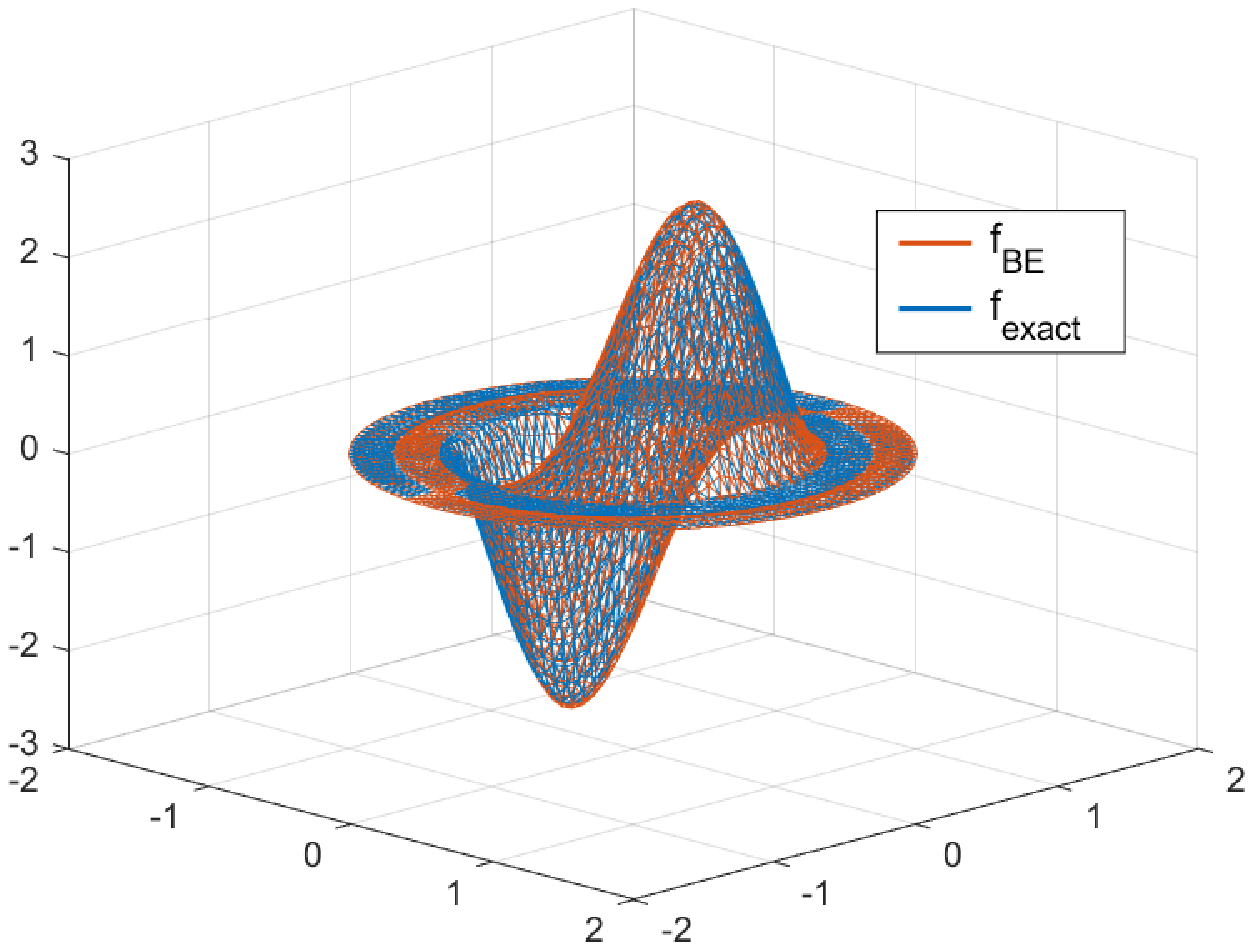}}&
\resizebox{0.3\textwidth}{!}{\includegraphics{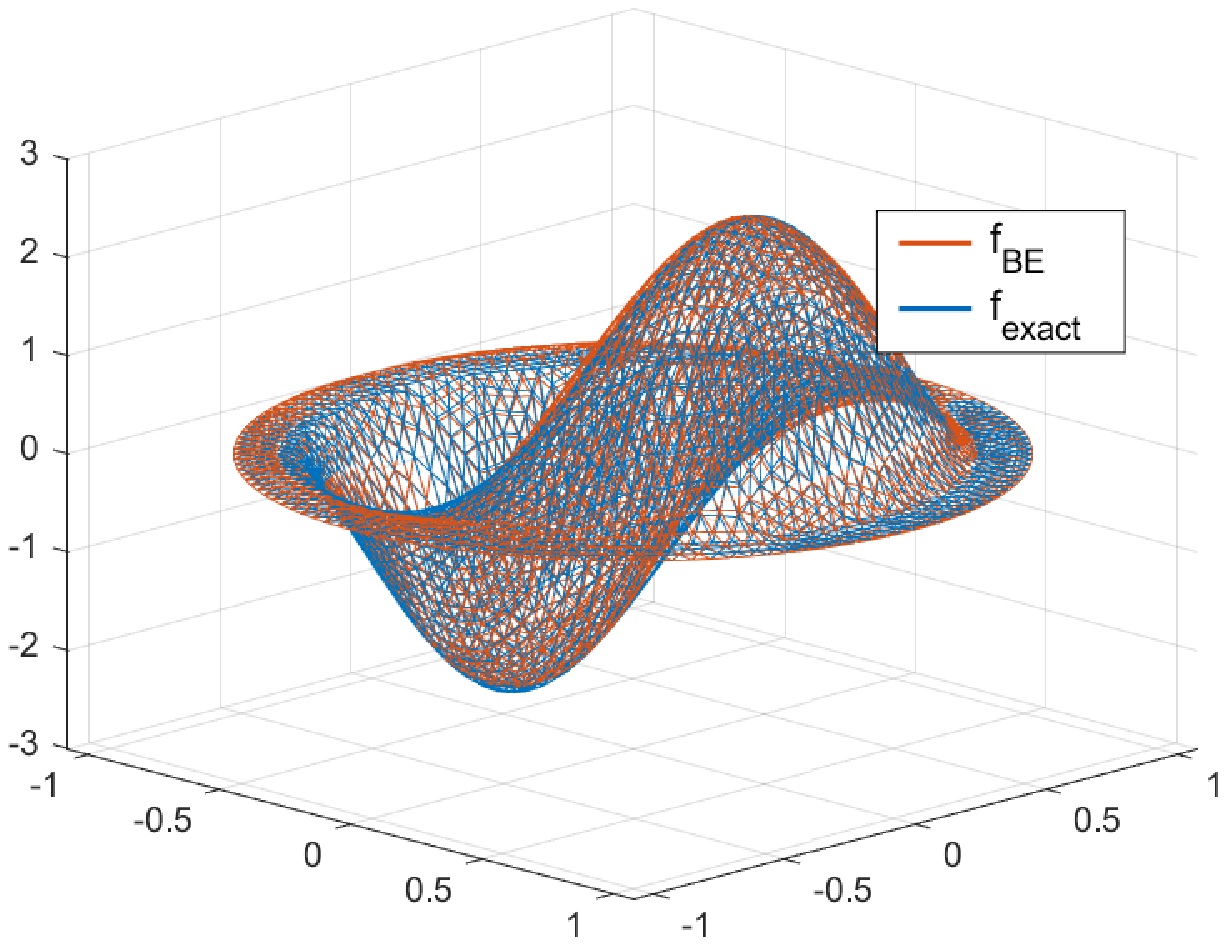}} \\
\end{tabular}
\end{center}
\caption{Example 1 (reconstructed support). First row: contour plots of the indicators for the DSM. Second row: the histograms of the coefficients for $f_{BE}$. Third row: the reconstructed $f_{BE}$ and exact $f$. Left column: $\Gamma_1$. Middle column: $\Gamma_2$. Right column: $\Gamma_3$.}
\label{fig:example1fig1}
\end{figure}

\vskip 0.2cm
\textbf{Example 2:} Let
\[
f(x)=2(0.81-(x_{1}^2+x_{2}^2))\chi_{x_{1}^2+x_{2}^2 \le 0.81},
\]
where $\chi$ is the characteristic function. The exact support of $f(x)$ is $B(0,0.9)$. The contour plots of the indicator functions by the DSM are shown in the first row of Fig.~\ref{fig:example2fig1} for $\Gamma_1$, $\Gamma_2$ and $\Gamma_3$. The radii of the discs $\hat{B}$'s are $0.9055$, $1.1180$ and $1.0817$, which are listed in Table~\ref{table1}. The histograms of the coefficients are shown in the second row of Fig.~\ref{fig:example2fig1}. The reconstructed $f_{BE}$'s and the exact $f$ are shown in the third row of Fig.~\ref{fig:example2fig1}. The errors are listed in Table~\ref{table2}. It can be seen that as the measurement aperture becomes less, the errors increase.

\begin{table}[ht]
\caption{Exact support of $f(x)$ and the radii of the discs by the DSM.}
    \centering
    \resizebox{\textwidth}{!}{
    \begin{tabular}{ c | c  c  c  c  c  }
 \Xhline{1pt}
\multirow{2}{*}{Exact support}  & Example 1 & Example 2 & Example 3 & Example 4 & Example 5 \\
&  B(0,0.9) &   B(0,0.9) & $ B$(0,0.7471)$^*$ &  $a,b=0.9,1.08$   &  B(0,0.9) \\\midrule
 $\Gamma_1$ & 1.3601  & 0.9055   &  0.8246  & 1.7205  & 0.9849  \\
 $\Gamma_2$  & 1.4213   &  1.1180 &    1.0198   &   1.5000 &    1.2166  \\
 $\Gamma_3$ &    1.0817&    1.0817 &    1.0630 &     1.2806&    1.1705 \\
 \Xhline{1pt}
\end{tabular} }
\label{table1}
\end{table}

\begin{table}[ht]
\caption{Absolute error (AE)  $\|f-f_{BE}\|_2$ and the relative error (RE) $\frac{\|f_{BE}-f\|_2}{\|f\|_2}$.}
    \centering
    \resizebox{\textwidth}{!}{
    \begin{tabular}{ c | c c c c| c c | c c| c  c | c  c   }
 \Xhline{1pt}
 & \multicolumn{4}{c|}{Example 1} & \multicolumn{2}{c|}{Example 2} & \multicolumn{2}{c|}{Example 3} & \multicolumn{2}{c|}{Example 4} & \multicolumn{2}{c}{Example 5} \\
 & AE$^e$ & RE$^e$ &       AE& RE &      AE& RE       &AE& RE&          AE& RE         &AE& RE \\ \midrule
 $\Gamma_1$ &  0.1274   & 6.02\% &   0.1184   & 5.61\% &    0.0455     &  3.06\% &   0.0735&  7.17\%  &  0.3414&  25.97\%  &   0.2127  &  13.43\%  \\
 $\Gamma_2$  &    0.0833  &  3.94\% &    0.1309  &  6.20\% &   0.0606    &  4.07\% &   0.1691 &  16.62\% &    0.3406    &  25.81\%   &   0.2697 &  17.13\%\\
 $\Gamma_3$ &   0.0346 &  1.63\% &  0.1224 &  5.79\% &      0.0667&  4.48\% &    0.2752 &  26.99\% & 0.4083&  30.88\% &        0.3021 &   19.14\% \\
 \Xhline{1pt}
\end{tabular} }
\label{table2}
\end{table}

\begin{figure}[h!]
\begin{center}
\begin{tabular}{lll}
\resizebox{0.3\textwidth}{!}{\includegraphics{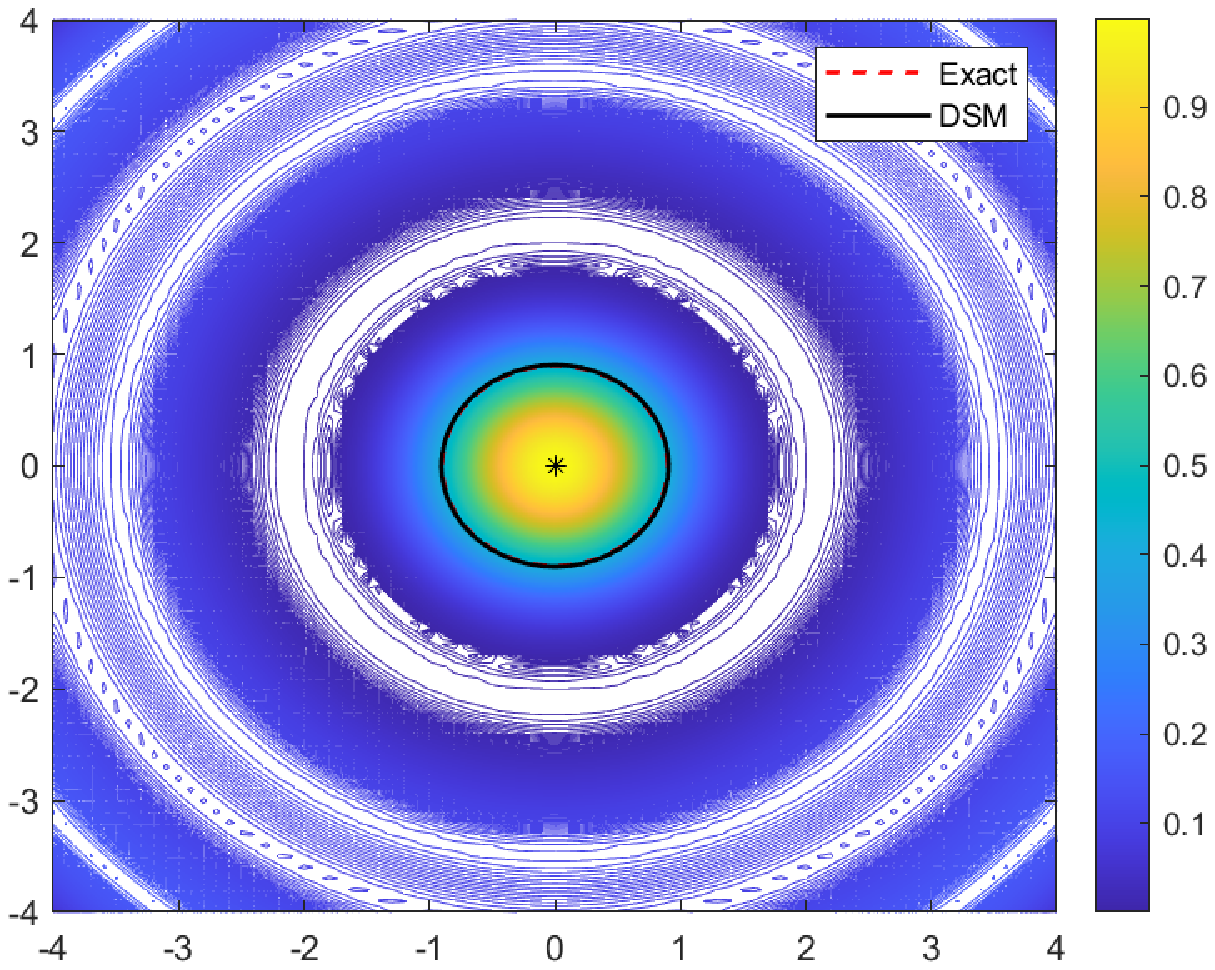}}&
\resizebox{0.3\textwidth}{!}{\includegraphics{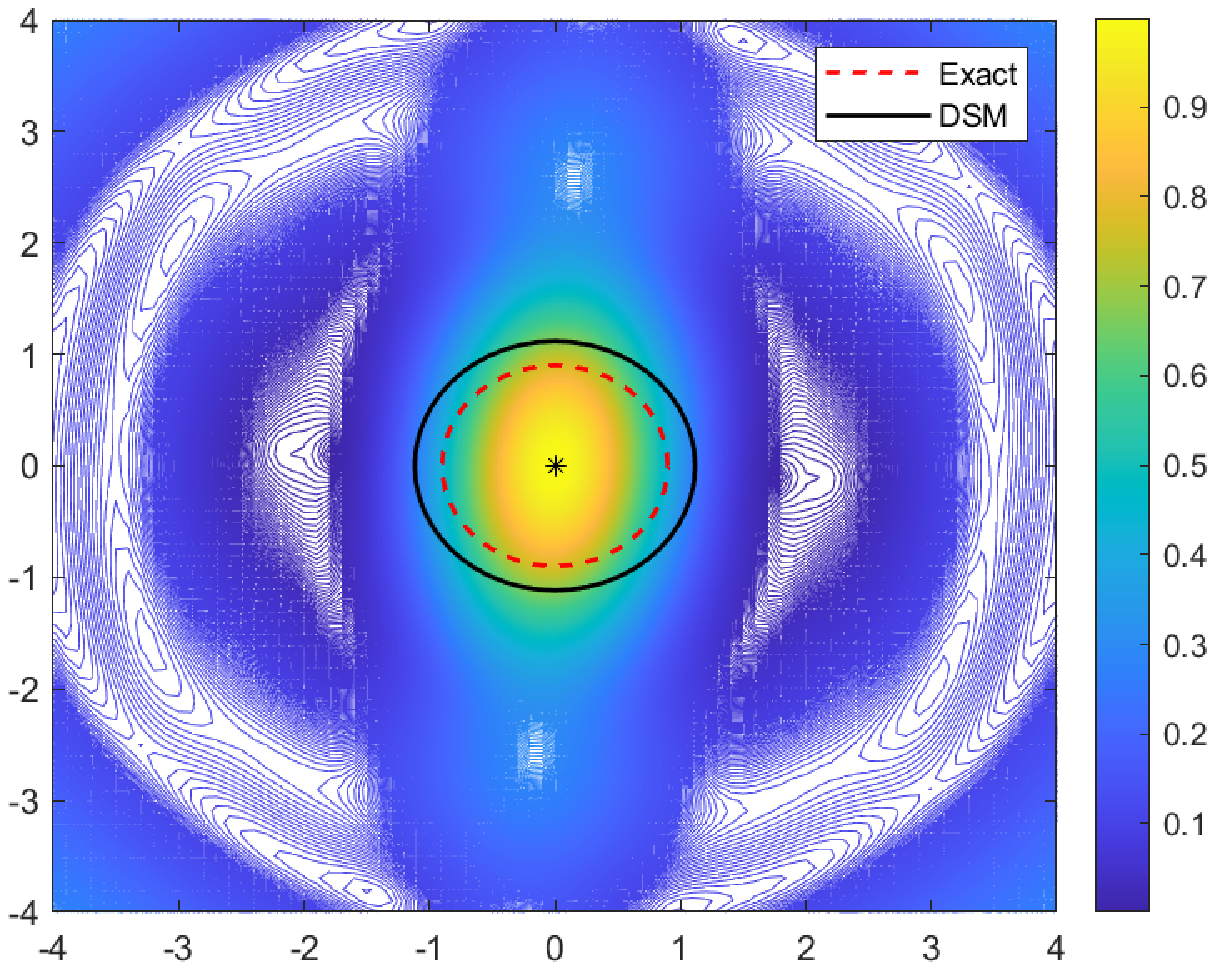}}&
\resizebox{0.3\textwidth}{!}{\includegraphics{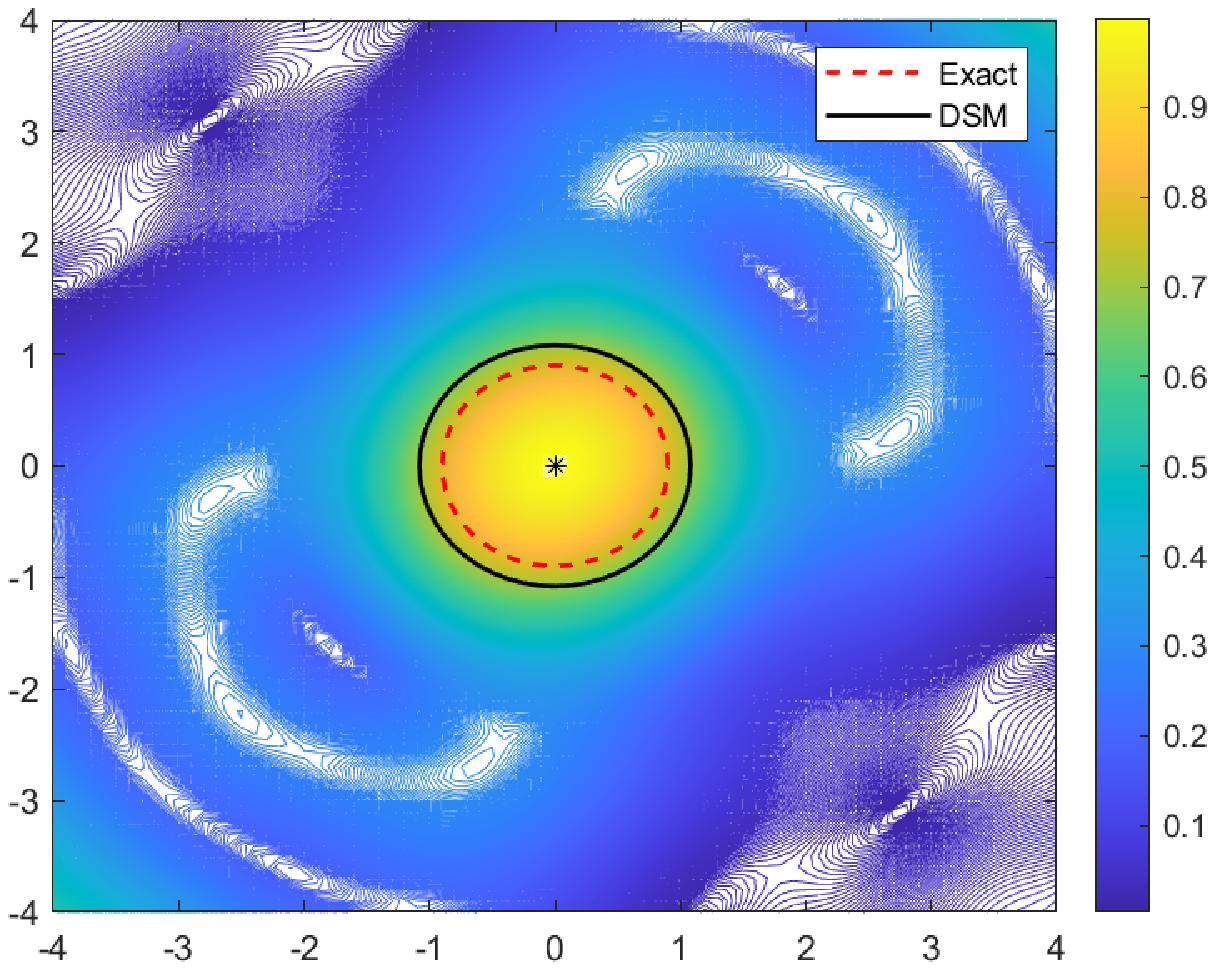}} \\
\resizebox{0.3\textwidth}{!}{\includegraphics{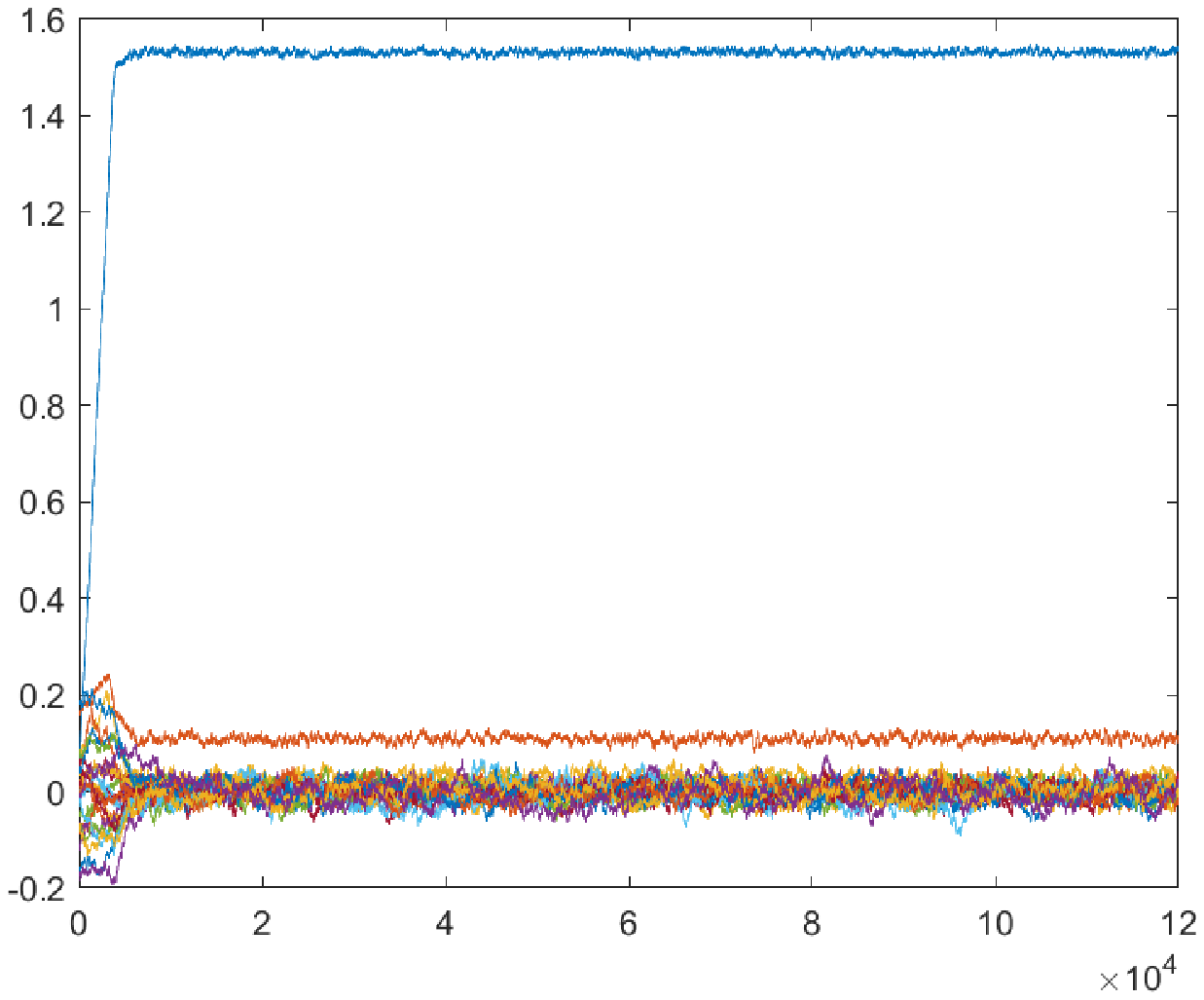}}&
\resizebox{0.3\textwidth}{!}{\includegraphics{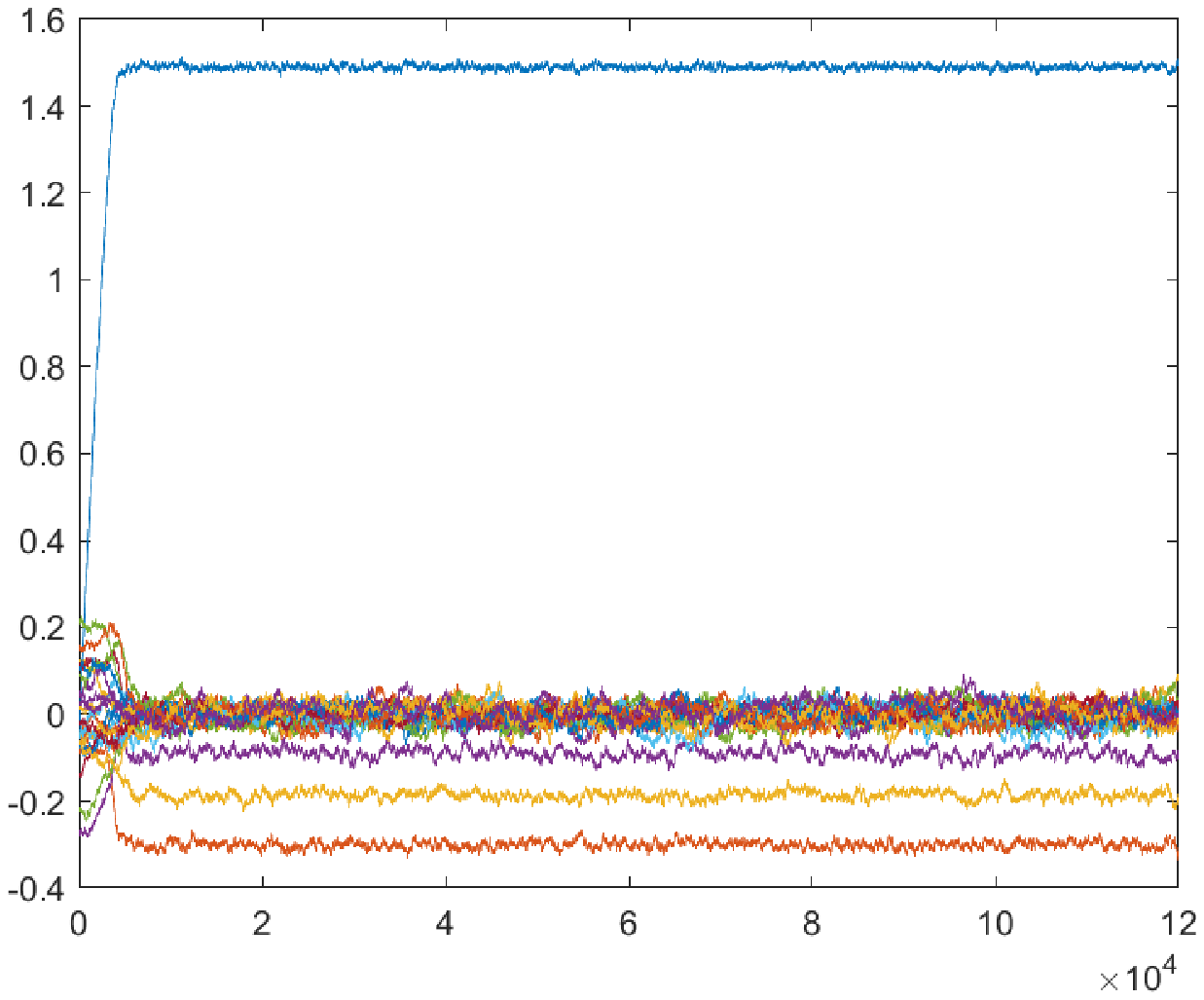}}&
\resizebox{0.3\textwidth}{!}{\includegraphics{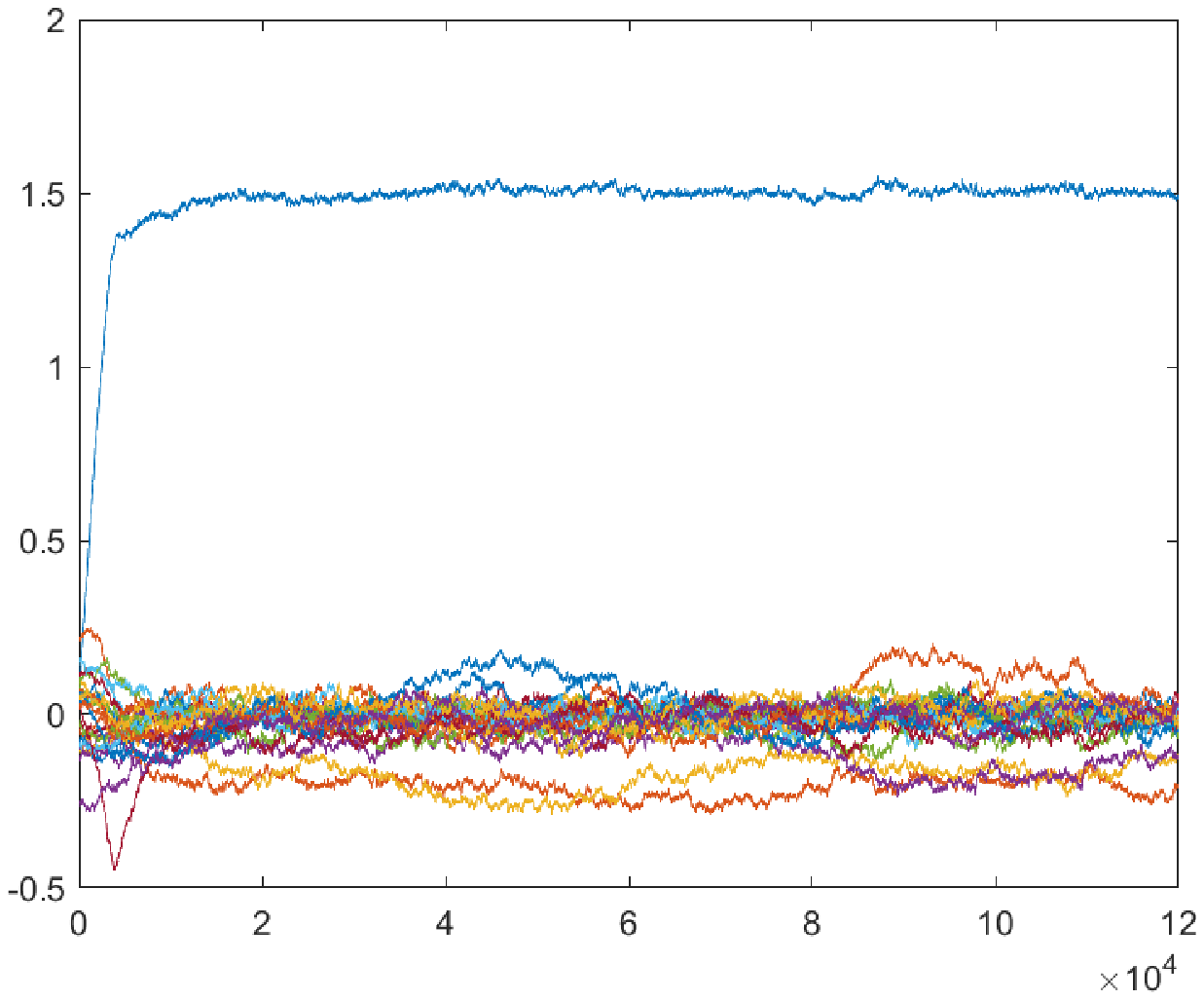}} \\
\resizebox{0.3\textwidth}{!}{\includegraphics{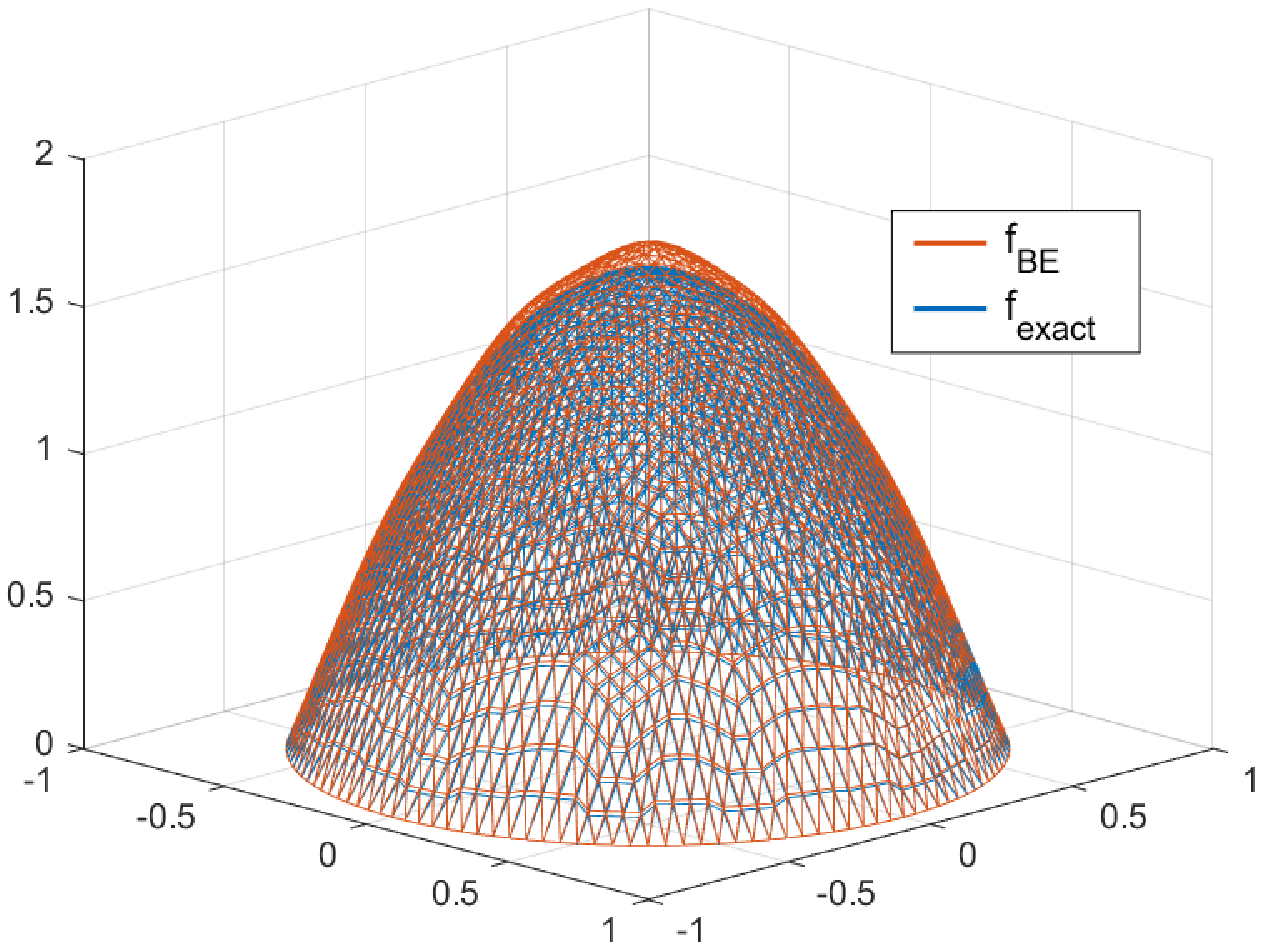}}&
\resizebox{0.3\textwidth}{!}{\includegraphics{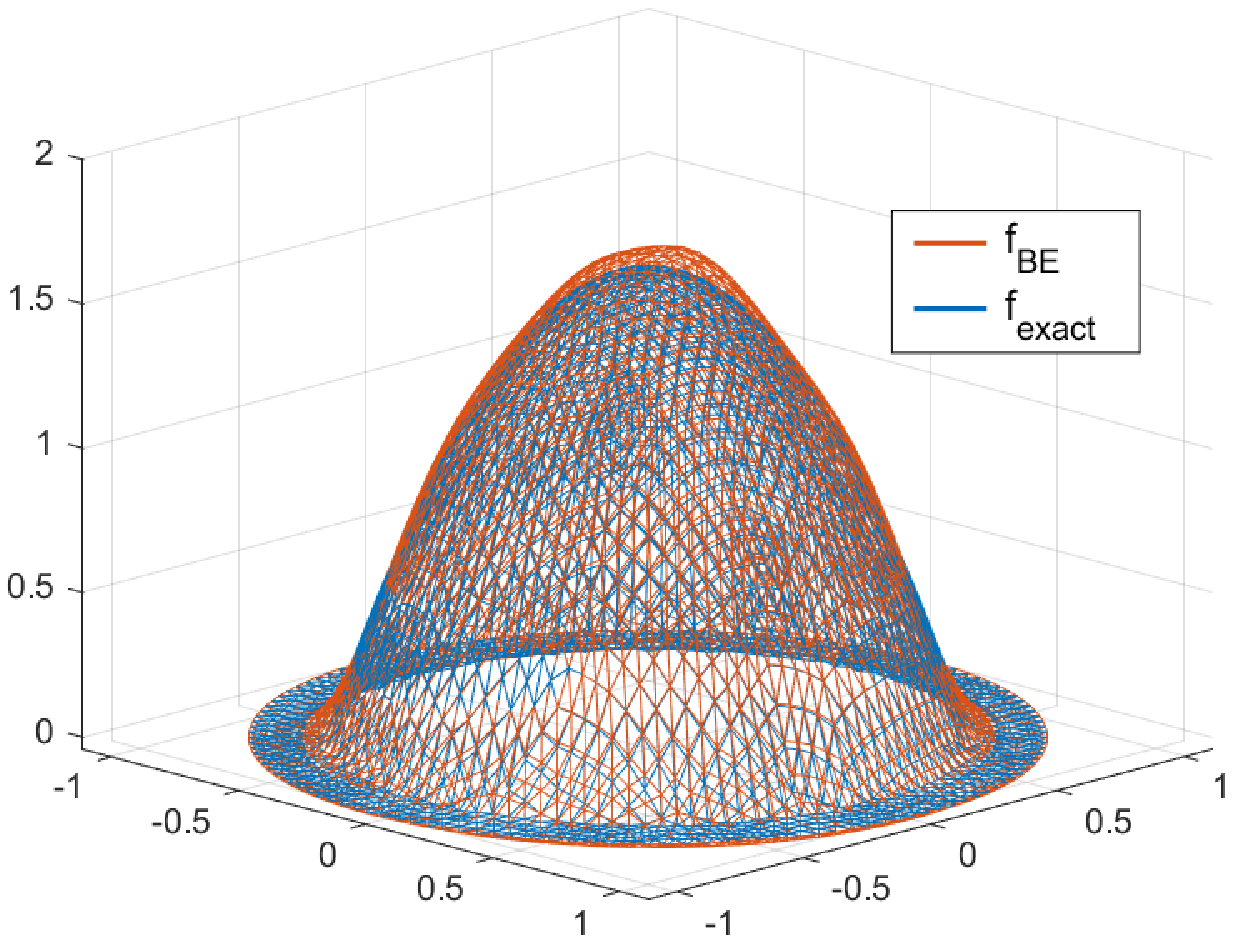}}&
\resizebox{0.3\textwidth}{!}{\includegraphics{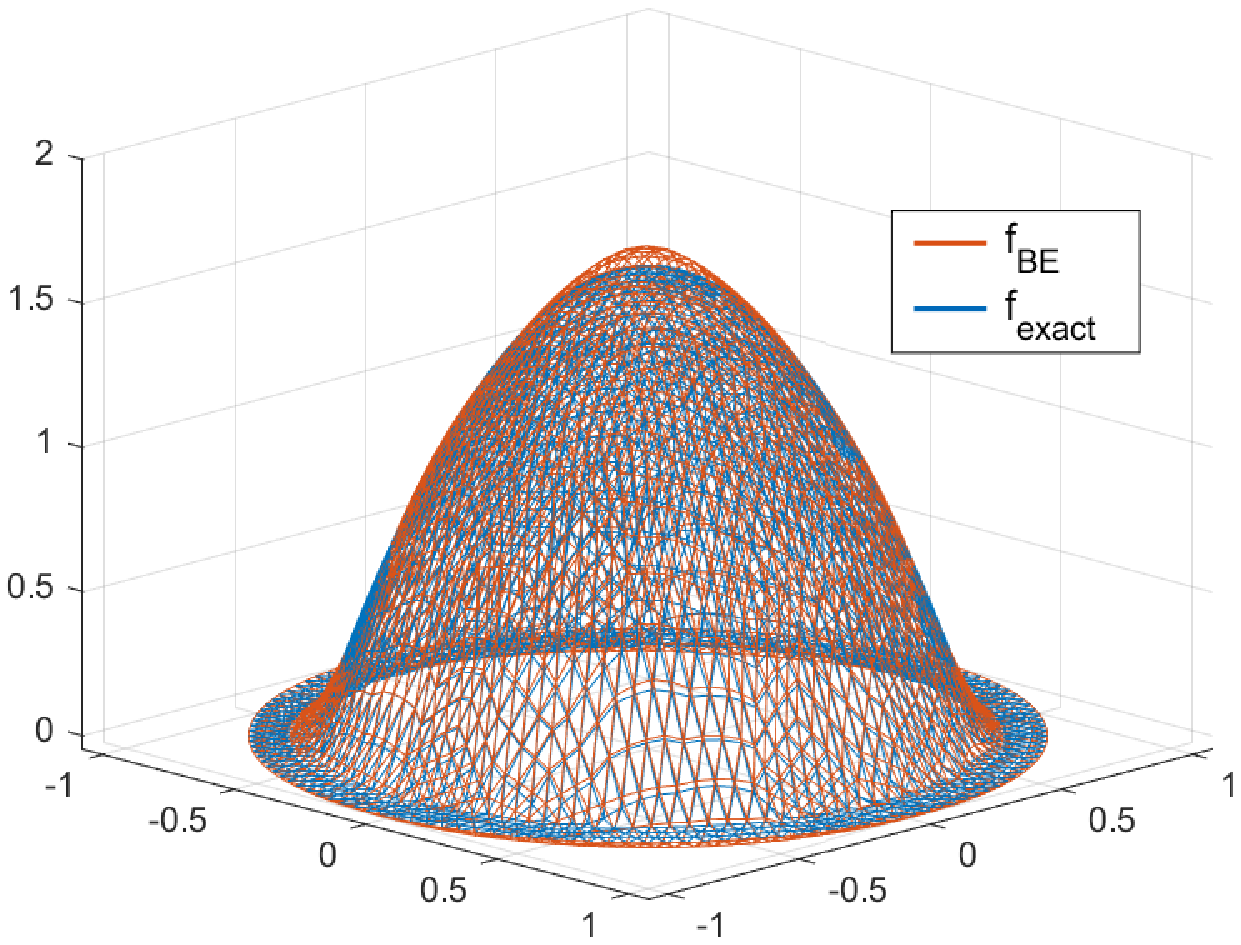}} \\
\end{tabular}
\end{center}
\caption{Example 2. First row: contour plots of the indicators for the DSM. Second row: the histograms of the coefficients for $f_{BE}$. Third row: the reconstructed $f_{BE}$ and exact $f$. Left column: $\Gamma_1$. Middle column: $\Gamma_2$. Right column: $\Gamma_3$.}
\label{fig:example2fig1}
\end{figure}

\vskip 0.2cm
\textbf{Example 3:} Let
\[
f(x)=5\exp(-45x_1^2-30x_2^2)).
\]
In this case, $f(x) \ne 0$ for all $x\in \mathbb{R}^2$. However, $f(x)$ is very close to $0$ when $|x|$ is large and the approximation \eqref{eigenfuncexpansion} for $f(x)$ is still valid approximately for $\hat{B}$ large enough. We consider a rough support of $f(x)$: $B^*=\{x\in \mathbb{R}^2| |f(x)|\leq 10^{-10}\}$. We have $B^* \approx B(0,0.7471)$. The contour plots of the indicator functions by the DSM are shown in the first row of Fig.~\ref{fig:example3fig1} for $\Gamma_1$, $\Gamma_2$ and $\Gamma_3$. The reconstructed domains $\hat{B}$ contains $B^*$ and are close to it for all three apertures. The radii of the reconstructed discs $\hat{B}$'s are $0.8246$, $1.0198$ and $1.0630$, which are listed in Table~\ref{table1}. The histograms of the coefficients are shown in the second row of Fig.~\ref{fig:example3fig1}. The reconstructed $f_{BE}$'s and the exact $f$ are shown in the third row of Fig.~\ref{fig:example3fig1}. The errors are listed in Table~\ref{table2}. Again when the measurement aperture becomes less the errors increase.

\begin{figure}[h!]
\begin{center}
\begin{tabular}{lll}
\resizebox{0.3\textwidth}{!}{\includegraphics{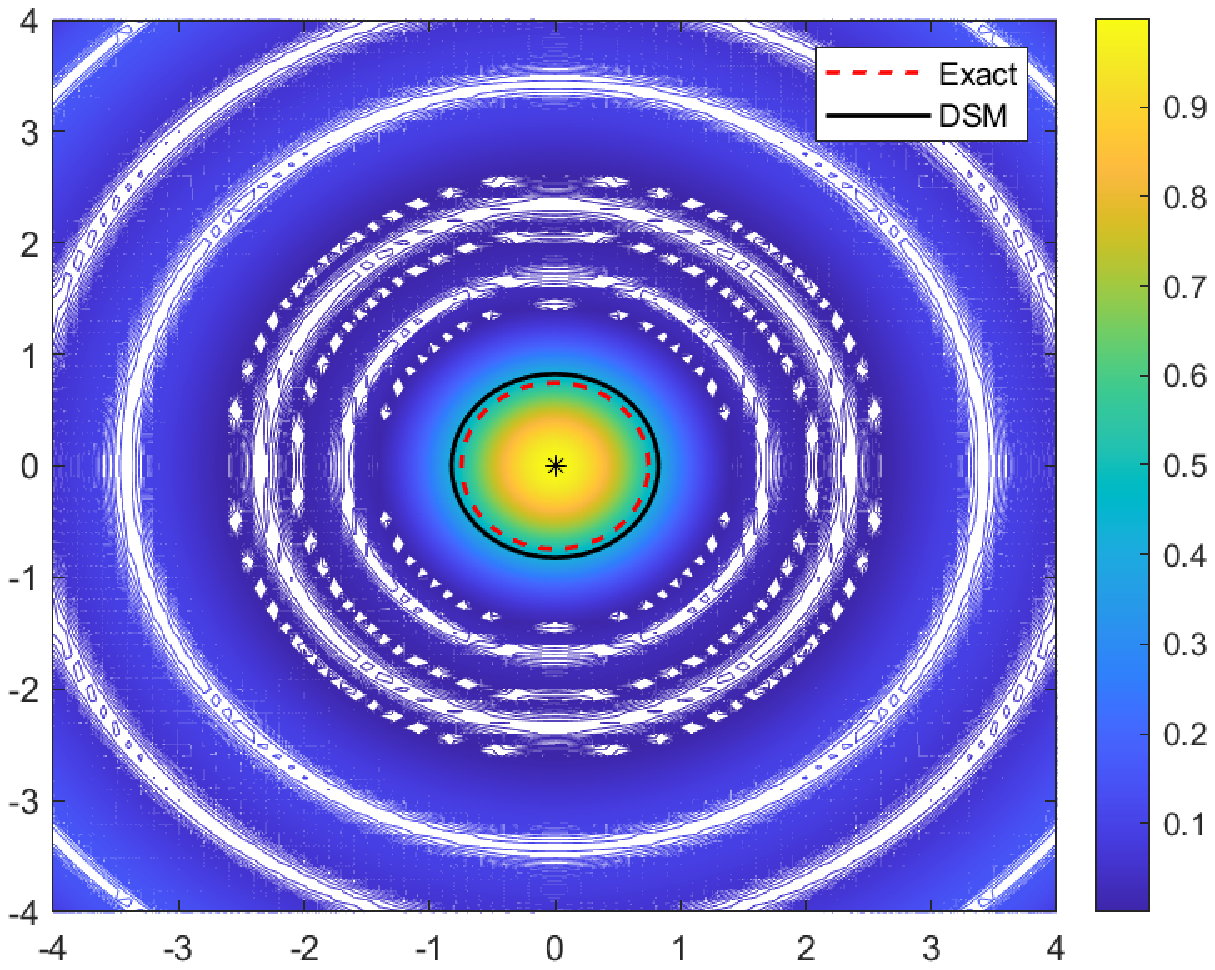}}&
\resizebox{0.3\textwidth}{!}{\includegraphics{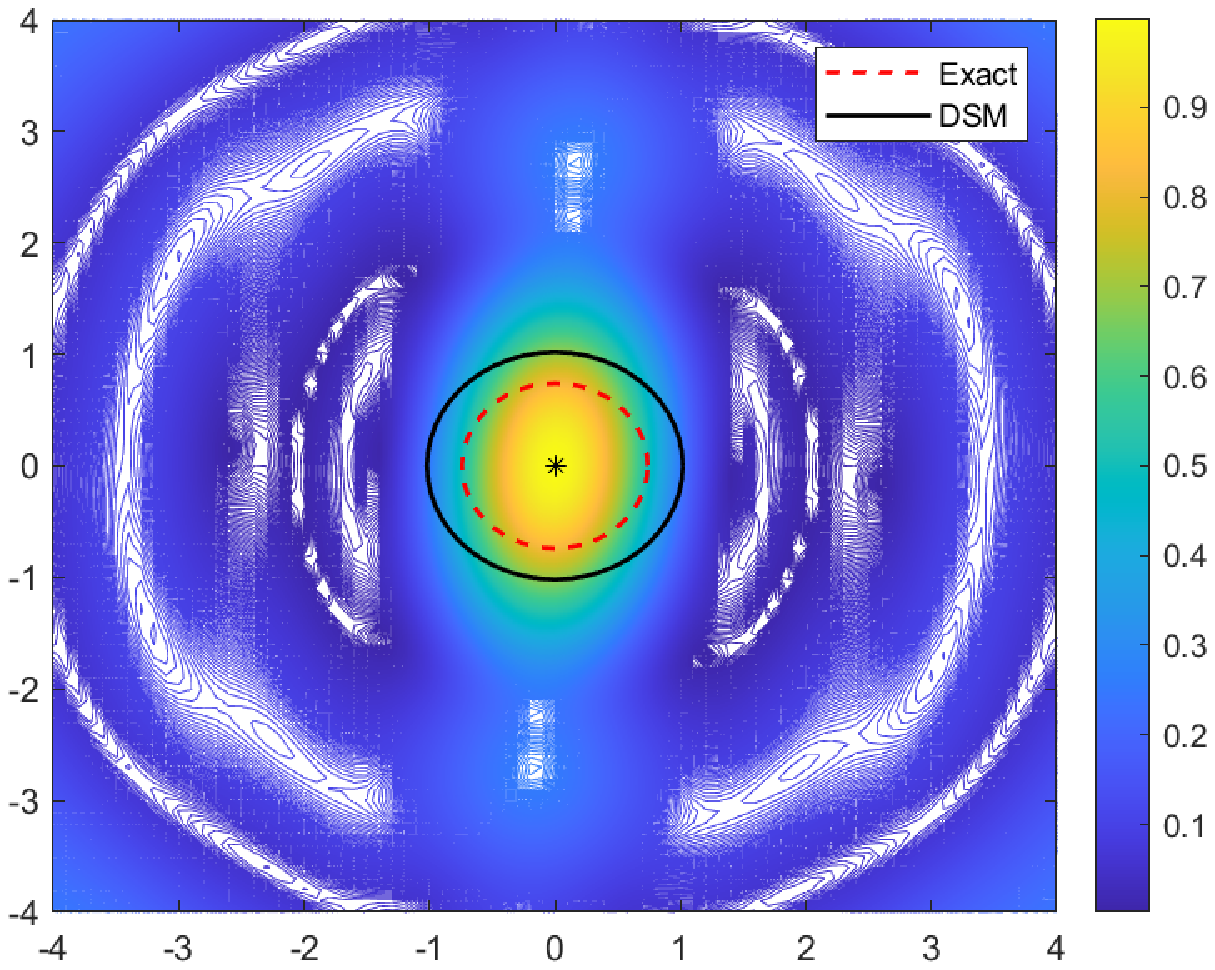}}&
\resizebox{0.3\textwidth}{!}{\includegraphics{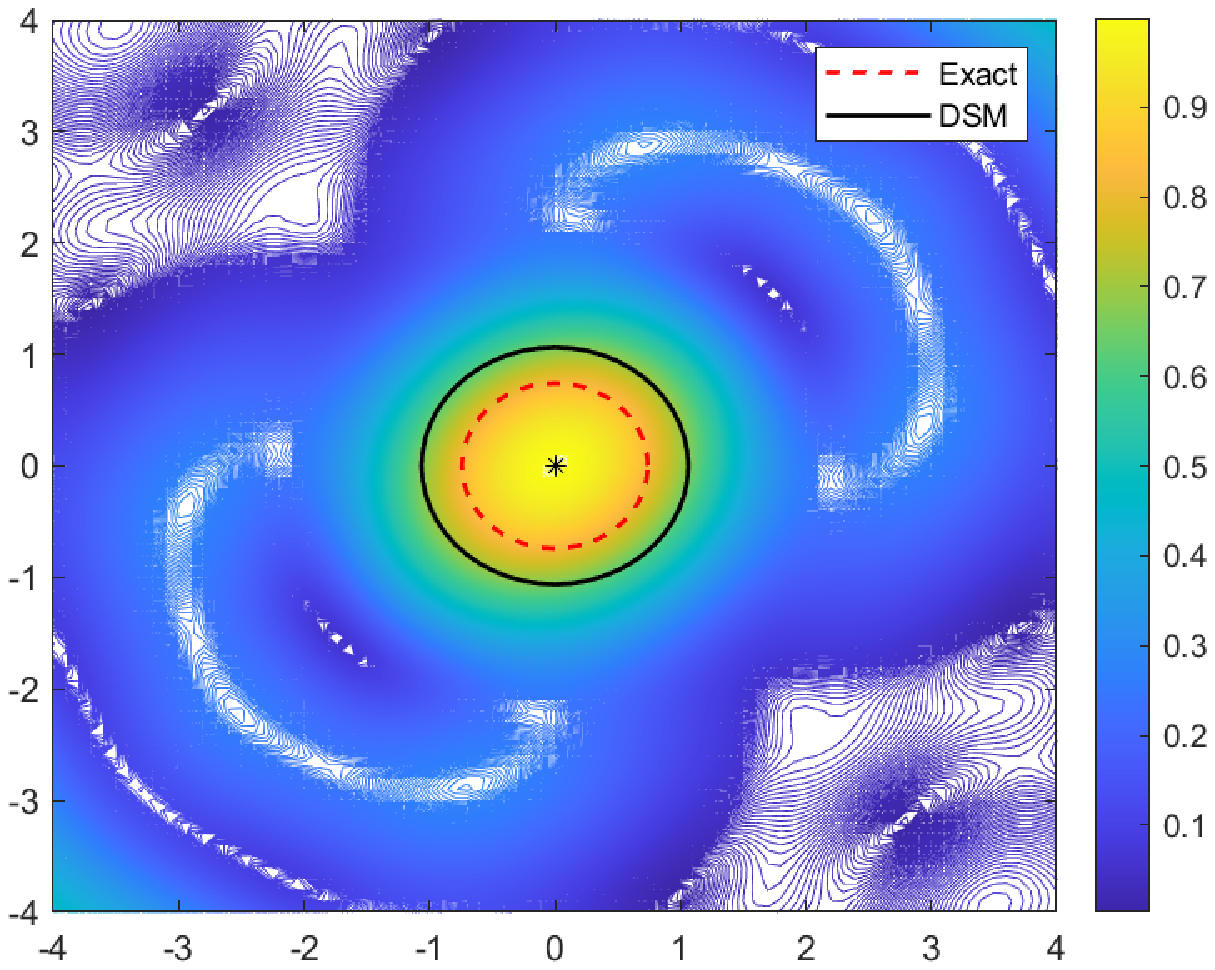}} \\
\resizebox{0.3\textwidth}{!}{\includegraphics{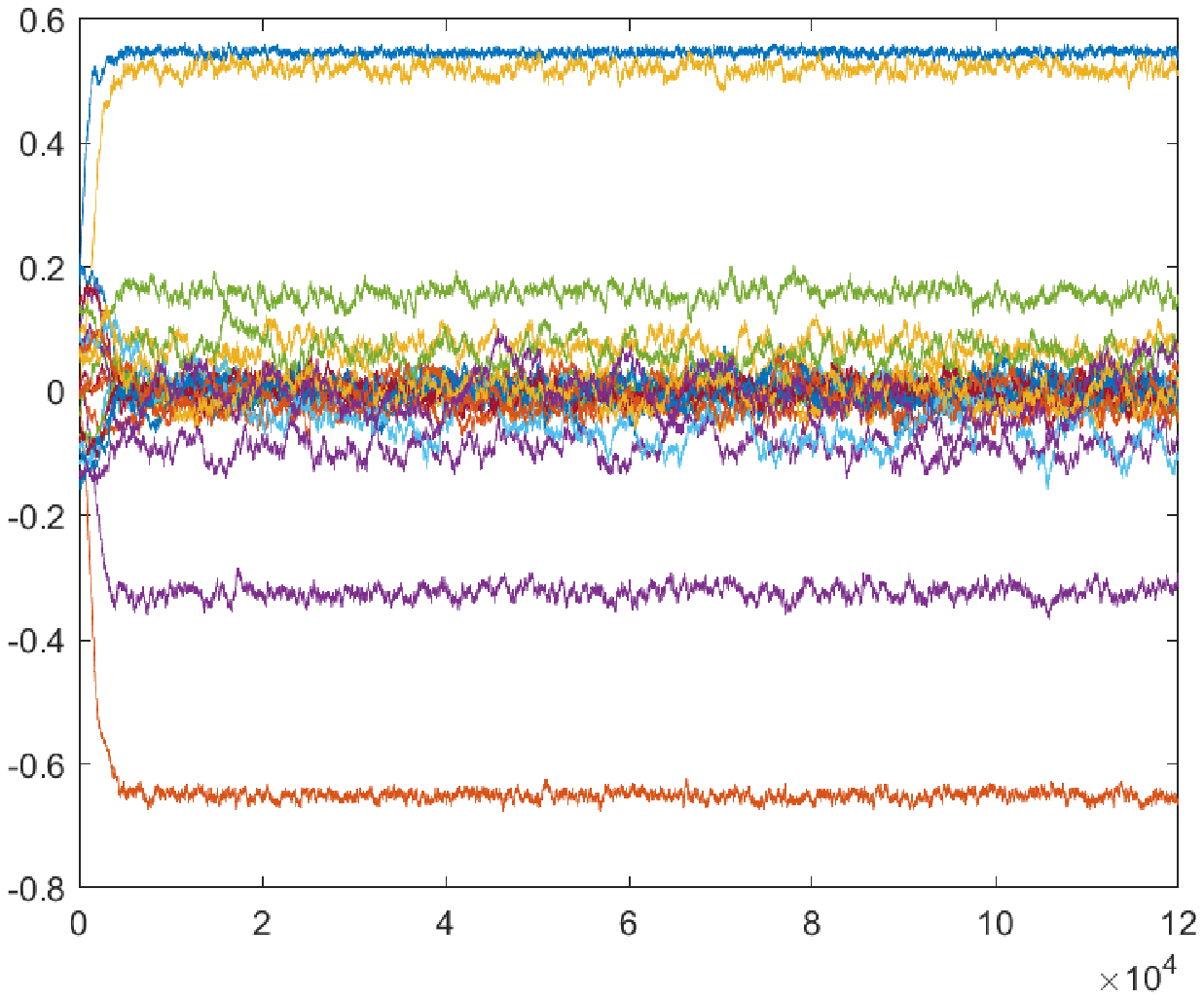}}&
\resizebox{0.3\textwidth}{!}{\includegraphics{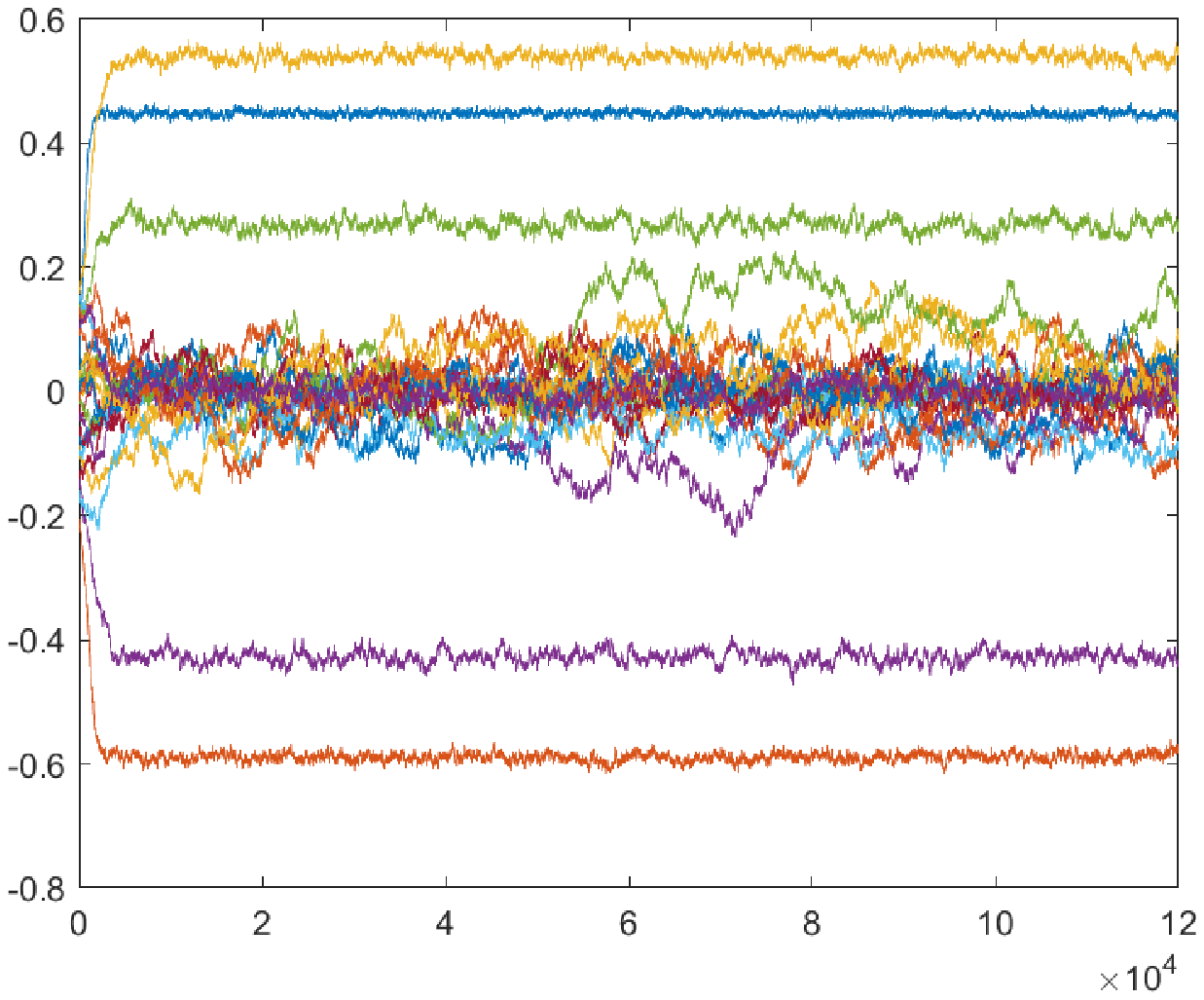}}&
\resizebox{0.3\textwidth}{!}{\includegraphics{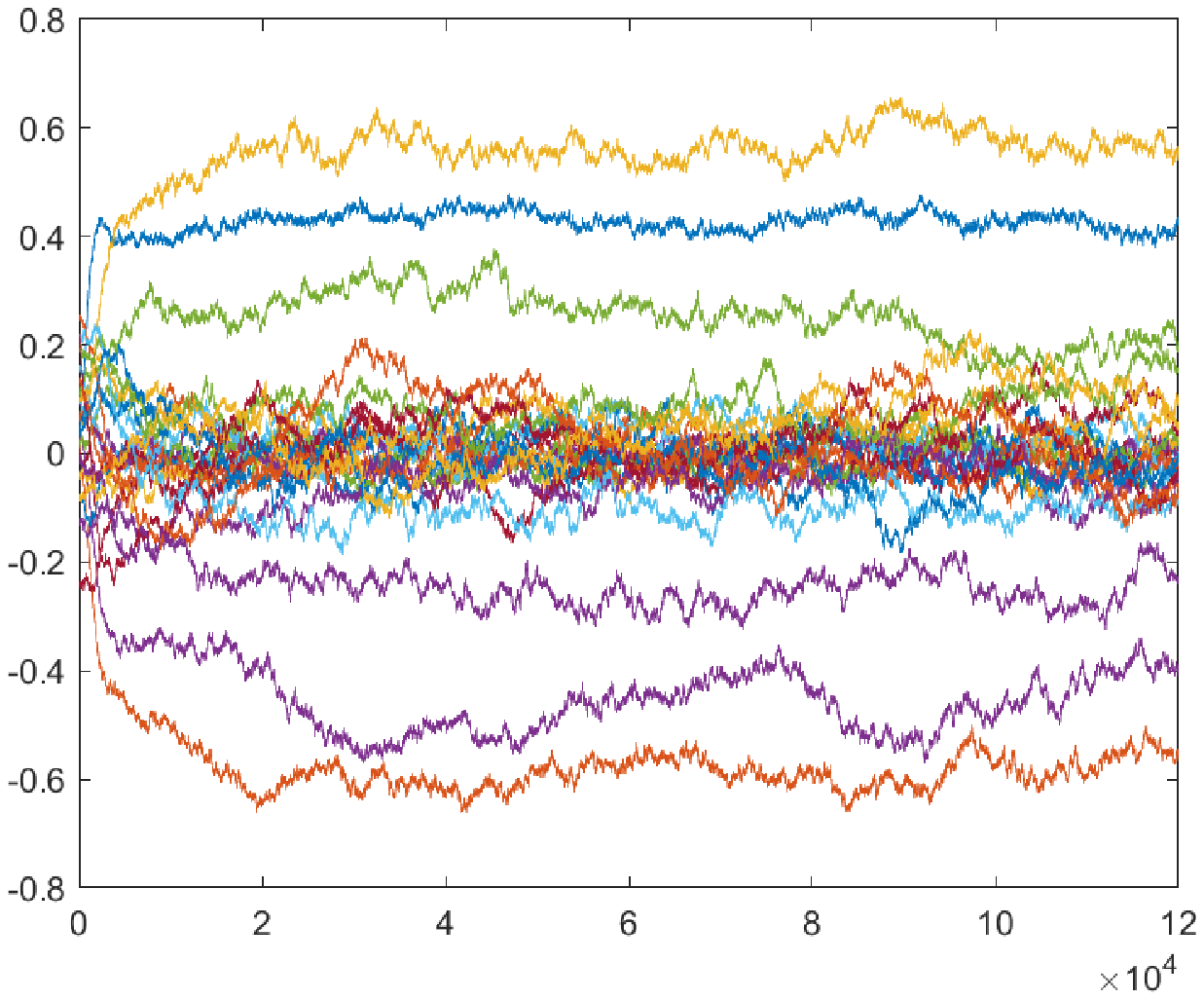}} \\
\resizebox{0.3\textwidth}{!}{\includegraphics{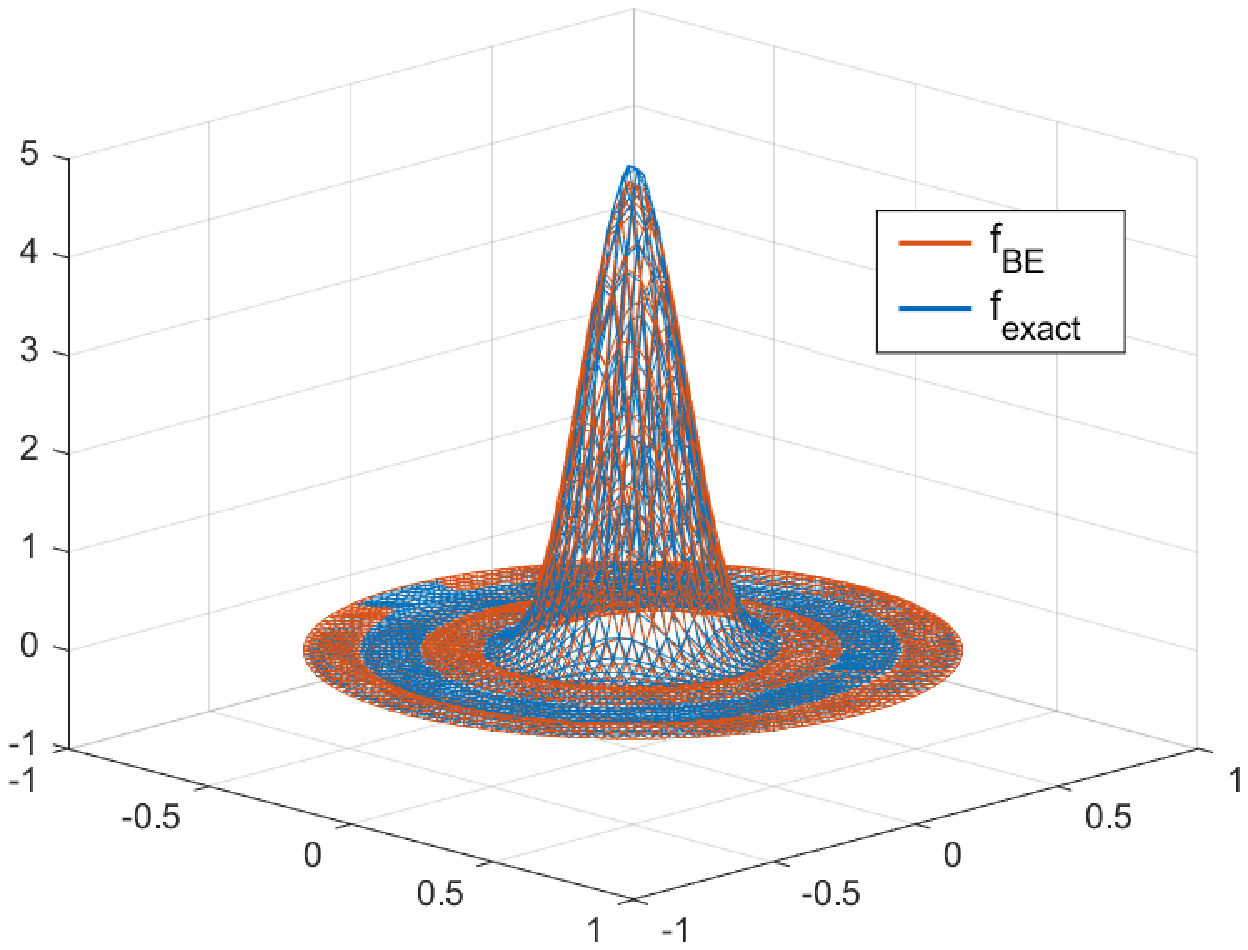}}&
\resizebox{0.3\textwidth}{!}{\includegraphics{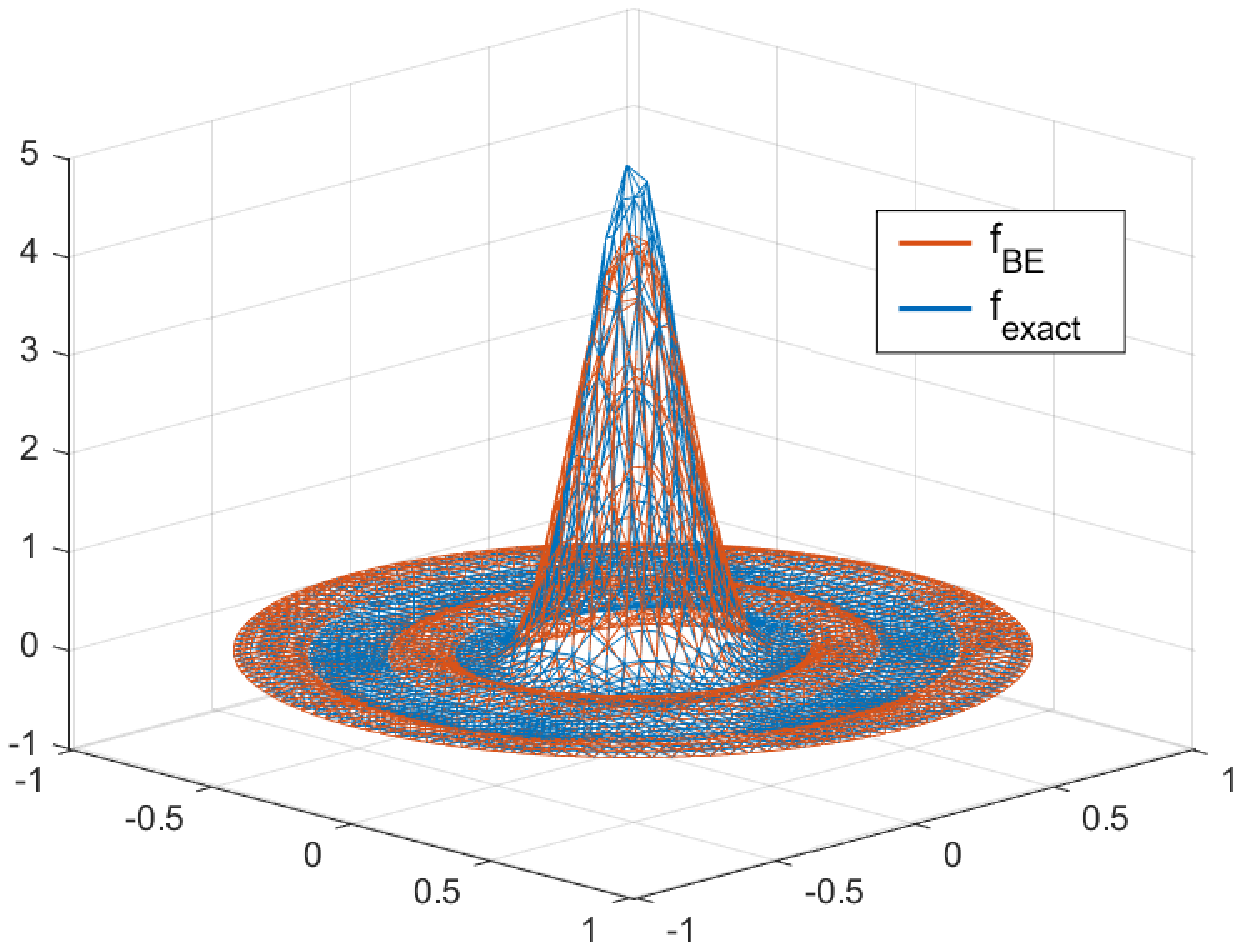}}&
\resizebox{0.3\textwidth}{!}{\includegraphics{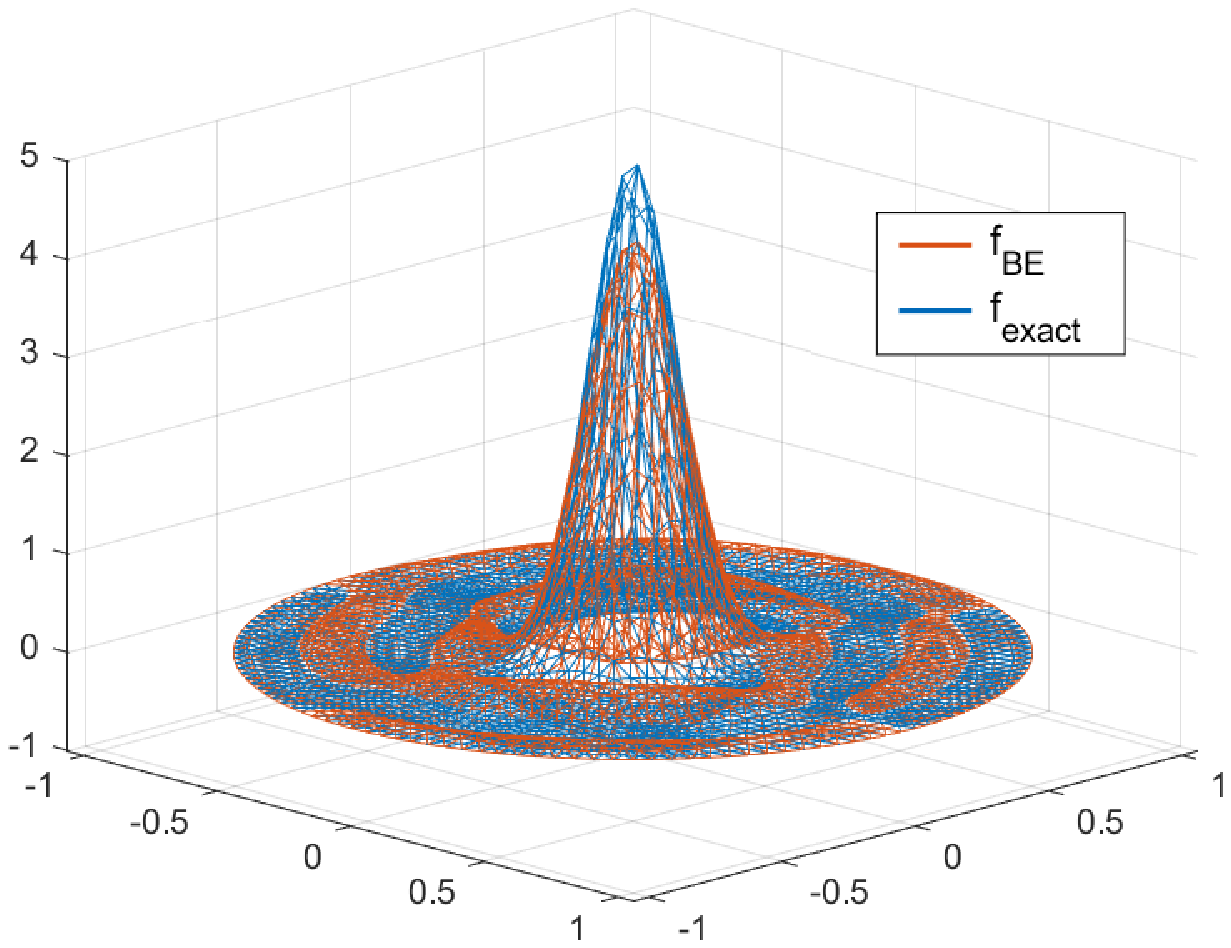}} \\
\end{tabular}
\end{center}
\caption{Example 3. First row: contour plots of the indicators for the DSM. Second row: the histograms of the coefficients for $f_{BE}$. Third row: the reconstructed $f_{BE}$ and exact $f$. Left column: $\Gamma_1$. Middle column: $\Gamma_2$. Right column: $\Gamma_3$.}
\label{fig:example3fig1}
\end{figure}

\vskip 0.2cm
\textbf{Example 4:} Let
\[
f(x)=15x_1x_2(0.81-(x_{1}^2+(x_{2}/1.2)^2))\chi_{\{(x_{1}^2+(x_{2}/1.2)^2)<=0.81\}}.
\]
The compact support of $f(x)$ is an ellipse with minor radius 0.9 and major radius 1.08. The approximate discs by the DSM (first row of Fig.~\ref{fig:example4fig1}) provide reliable estimates for the support $f(x)$, which are given in Table~\ref{table1}. The histograms of the coefficients are shown in the second row of Fig.~\ref{fig:example4fig1}. The reconstructed $f_{BE}$'s and the exact $f$ are shown in the third row of Fig.~\ref{fig:example4fig1}. The errors are listed in Table~\ref{table2}. Again the errors increase as the measurement aperture becomes less.
\begin{figure}[h!]
\begin{center}
\begin{tabular}{lll}
\resizebox{0.3\textwidth}{!}{\includegraphics{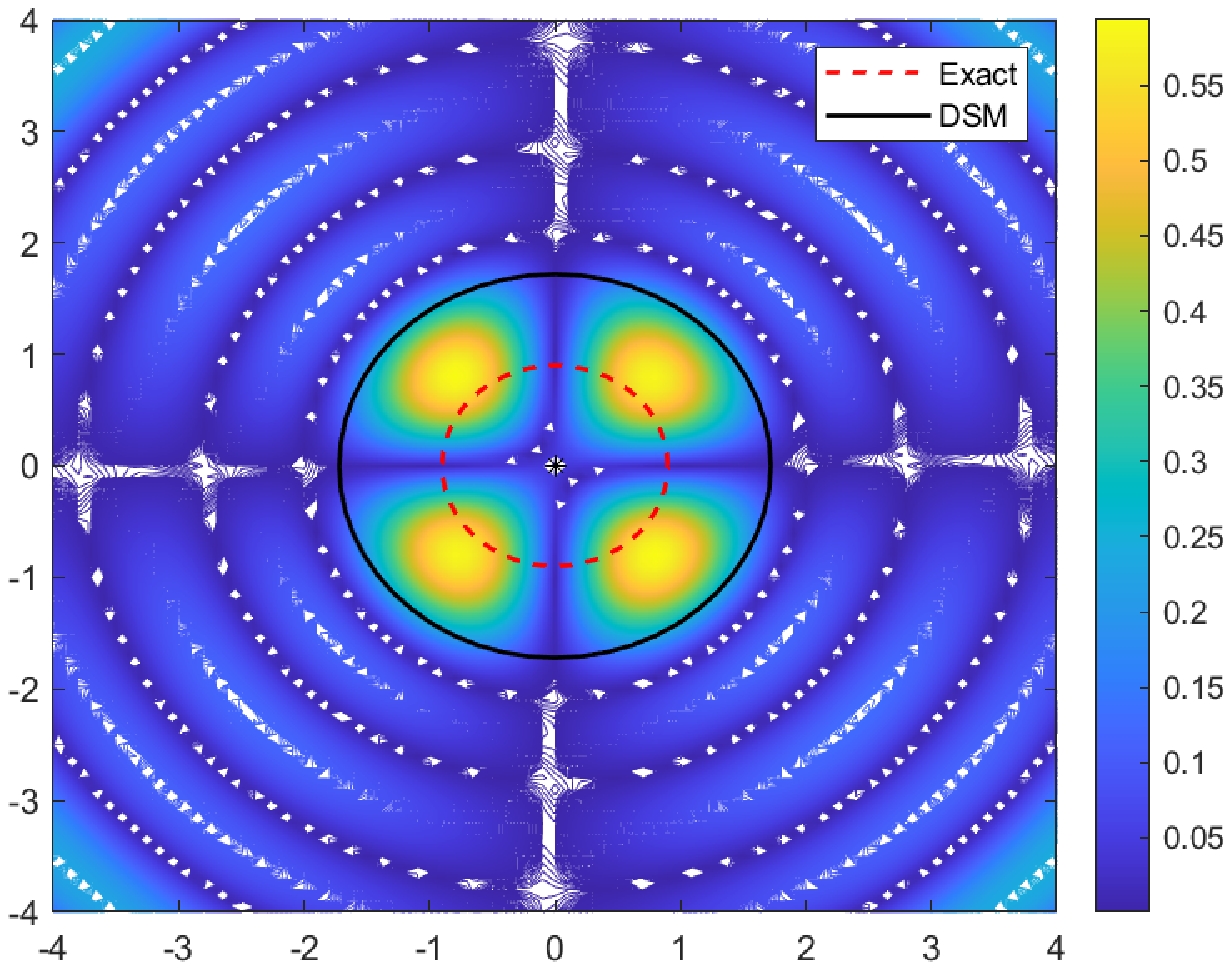}}&
\resizebox{0.3\textwidth}{!}{\includegraphics{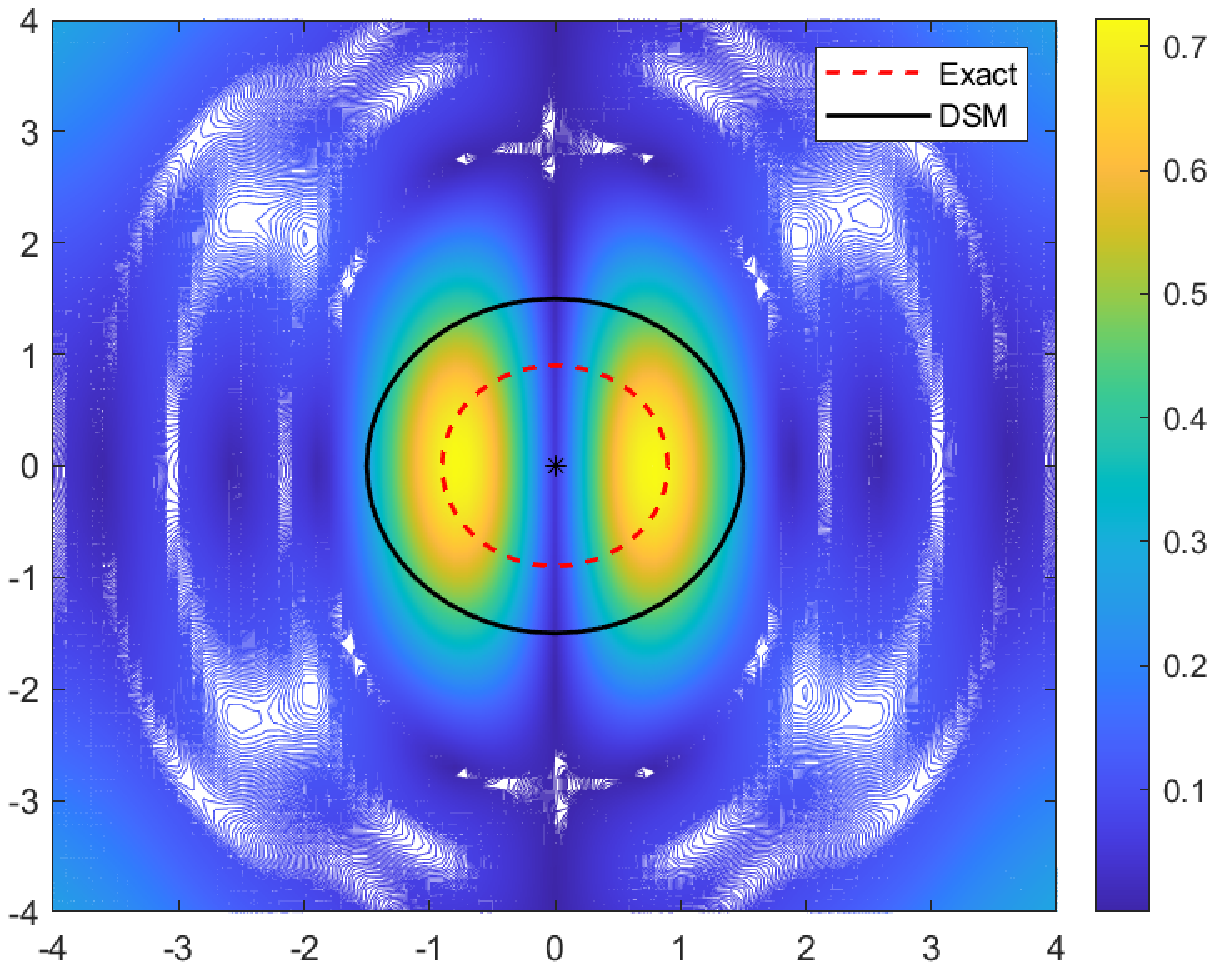}}&
\resizebox{0.3\textwidth}{!}{\includegraphics{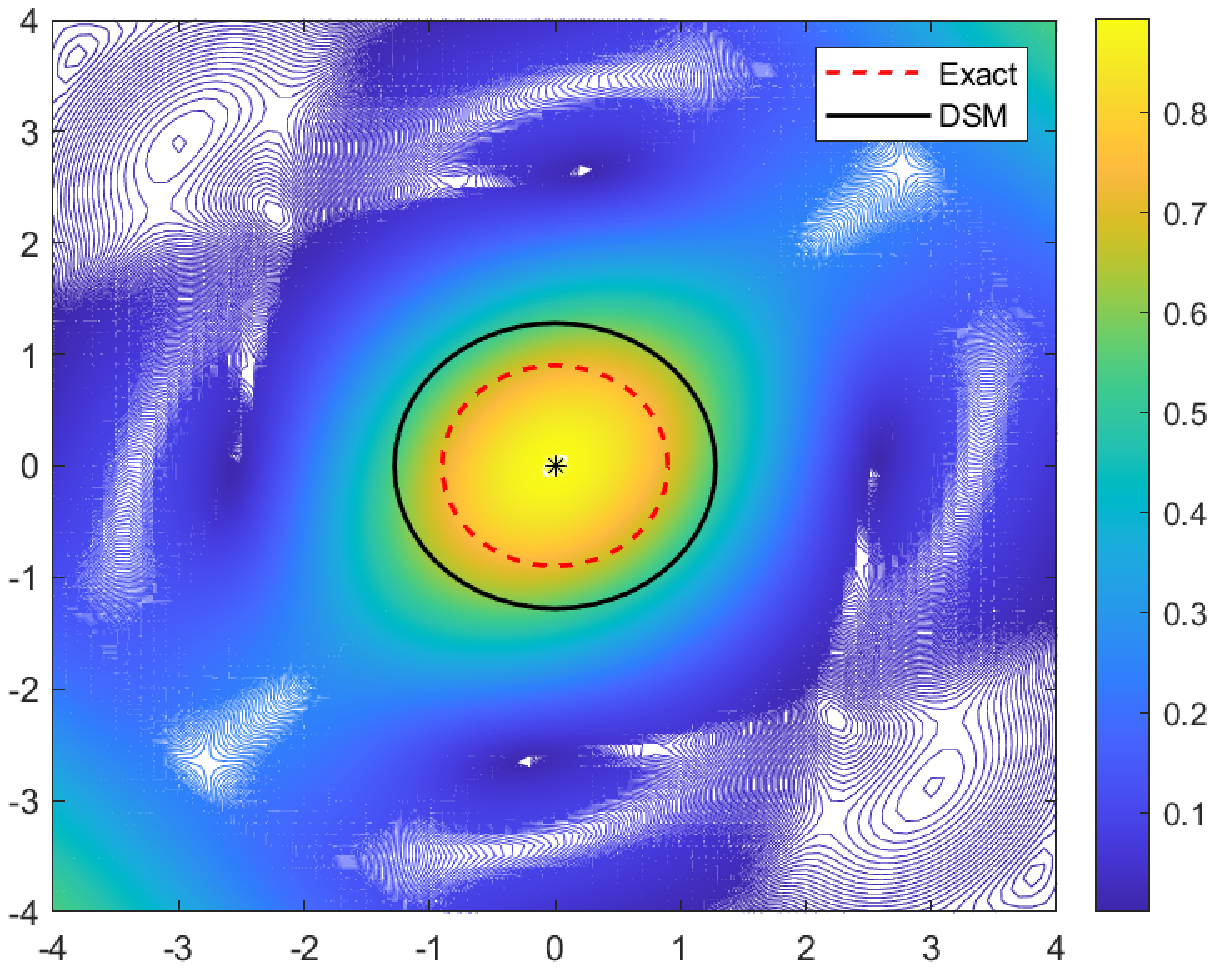}} \\
\resizebox{0.3\textwidth}{!}{\includegraphics{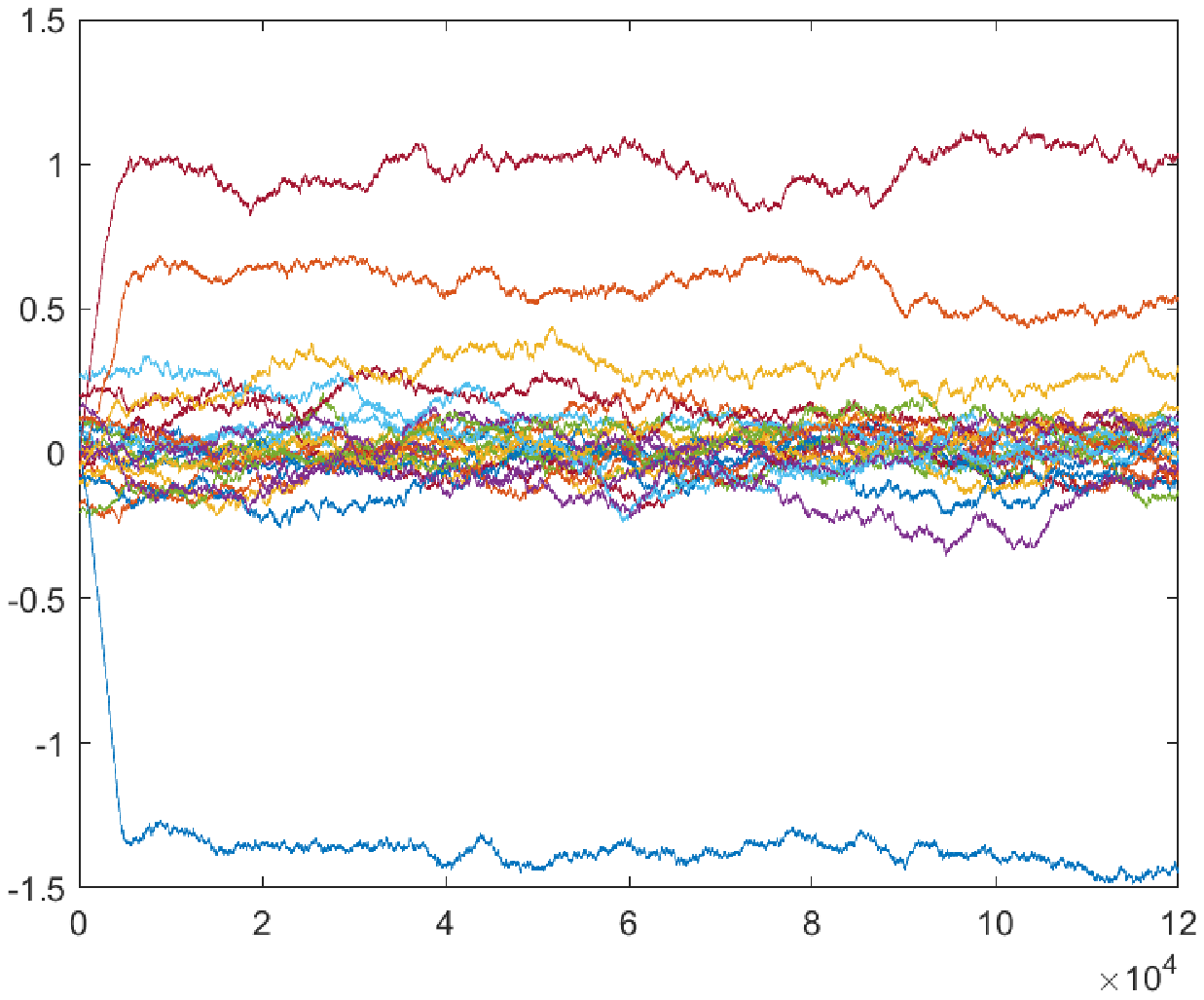}}&
\resizebox{0.3\textwidth}{!}{\includegraphics{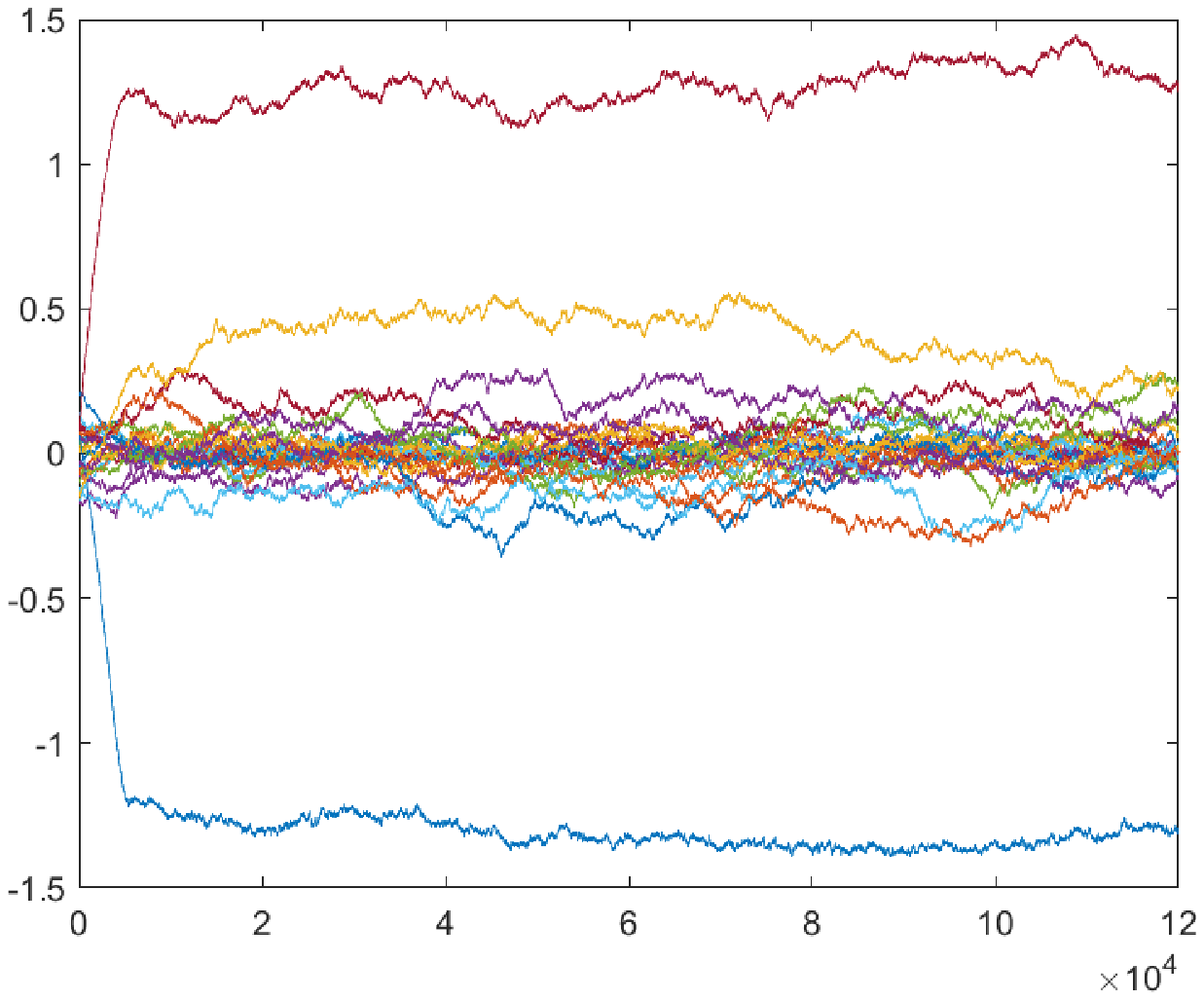}}&
\resizebox{0.3\textwidth}{!}{\includegraphics{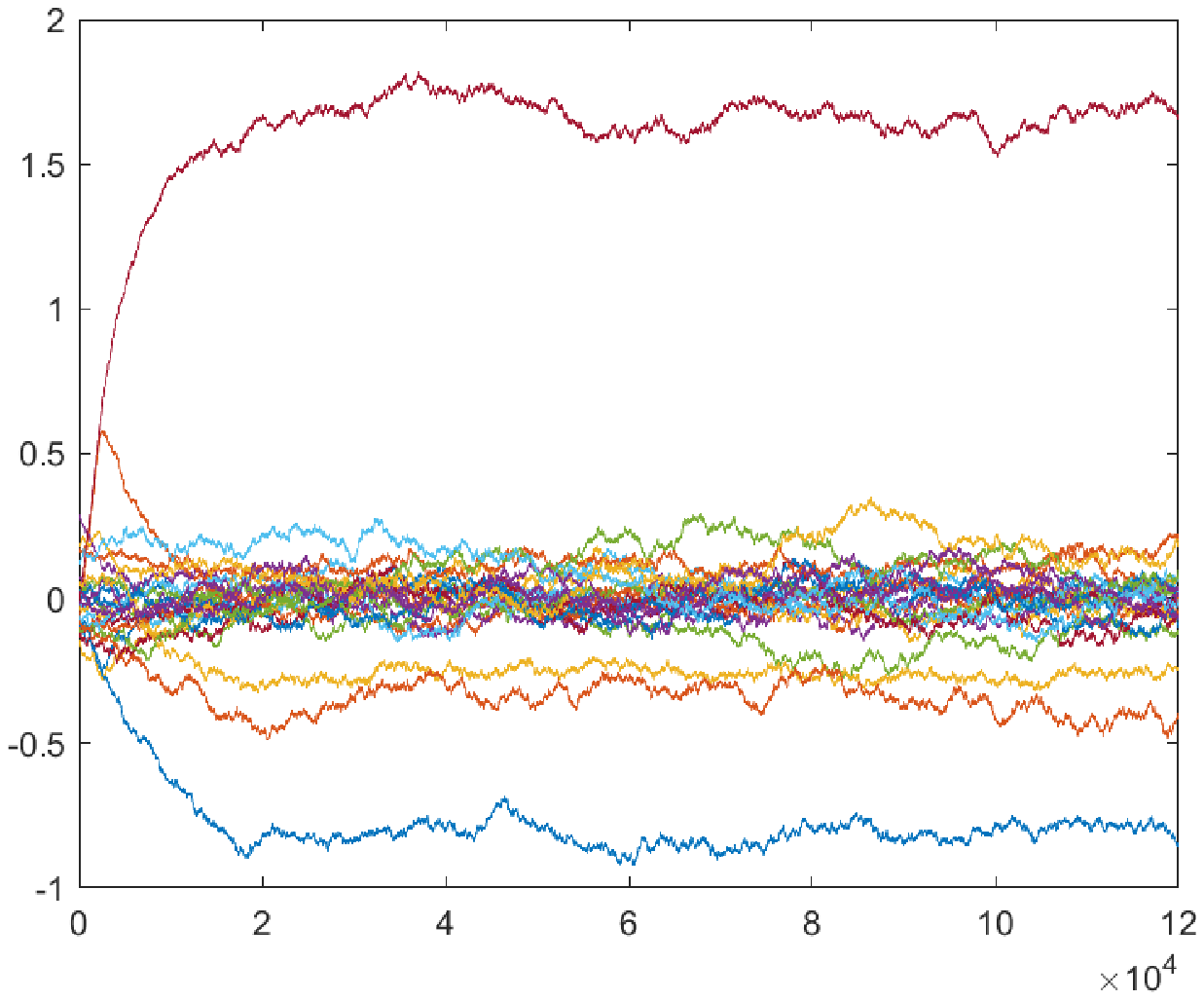}} \\
\resizebox{0.3\textwidth}{!}{\includegraphics{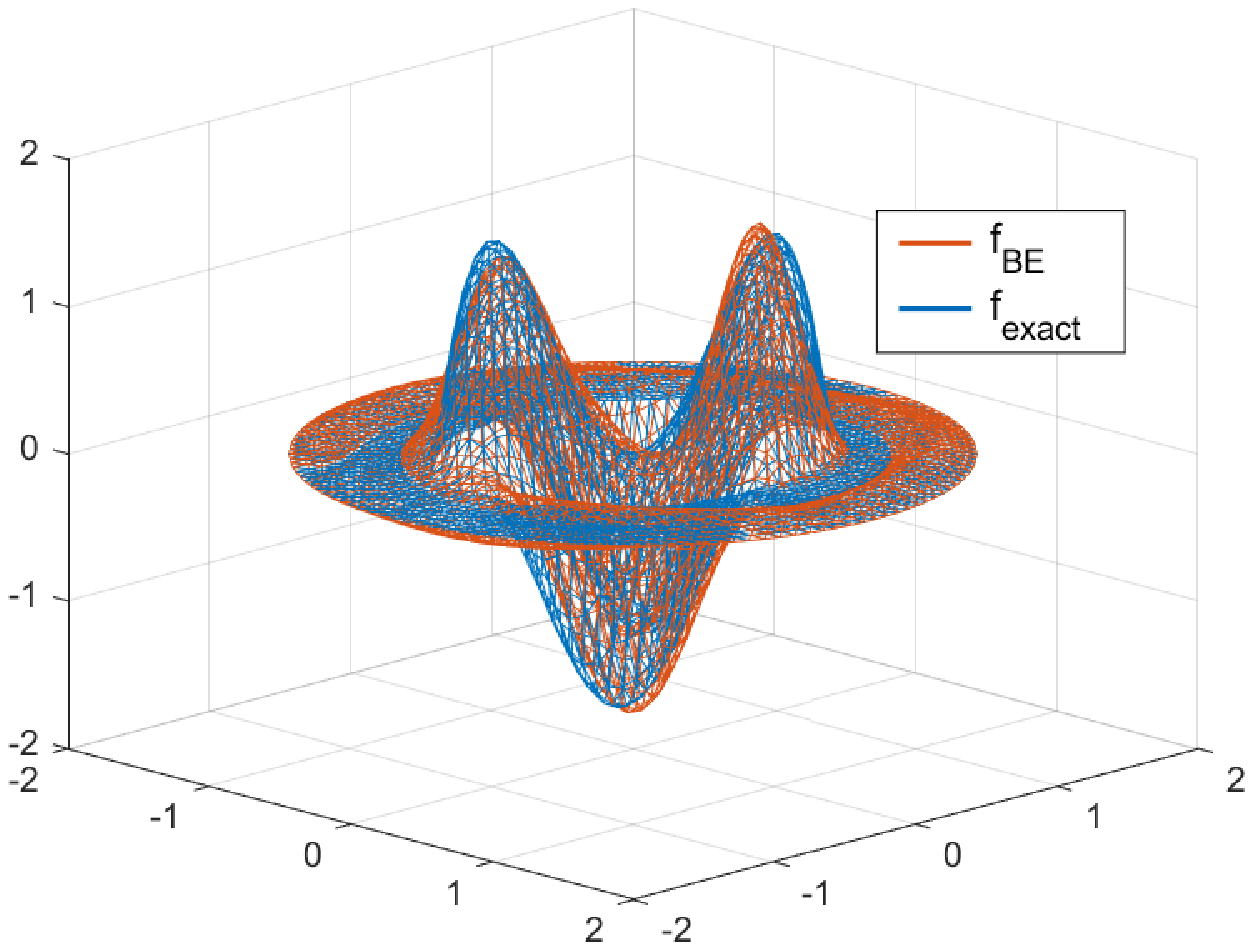}}&
\resizebox{0.3\textwidth}{!}{\includegraphics{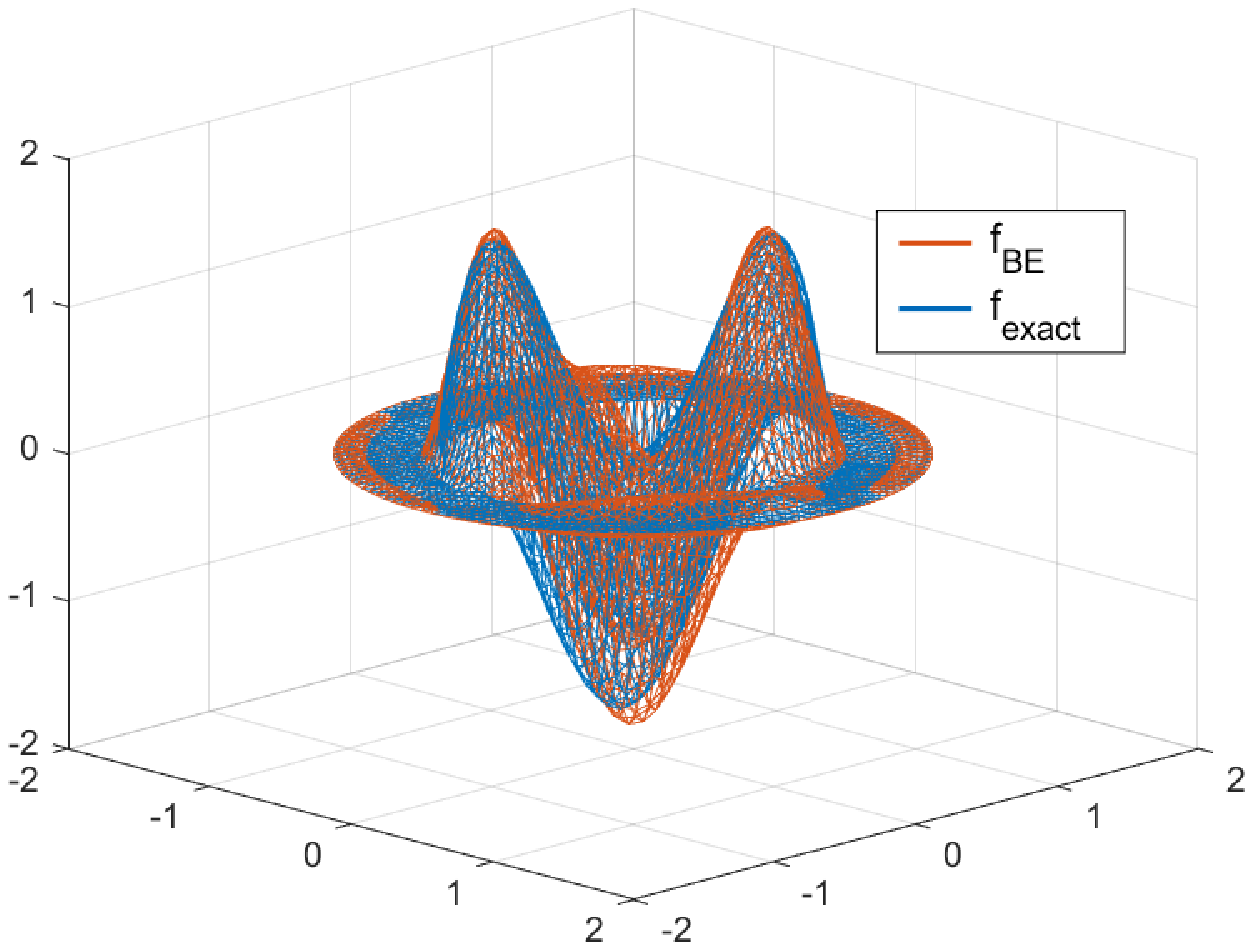}}&
\resizebox{0.3\textwidth}{!}{\includegraphics{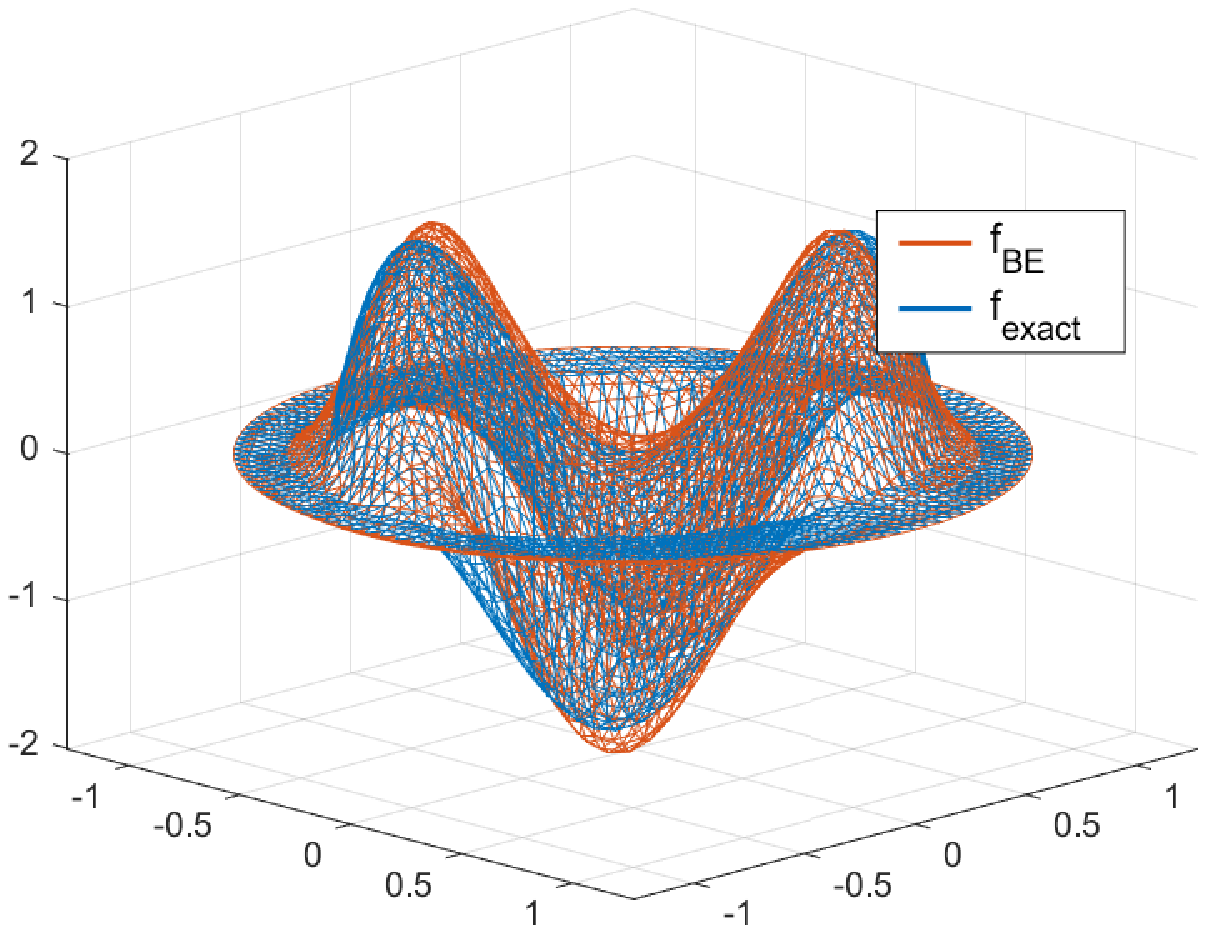}} \\
\end{tabular}
\end{center}
\caption{Example 4. First row: contour plots of the indicators for the DSM. Second row: the histograms of the coefficients for $f_{BE}$. Third row: the reconstructed $f_{BE}$ and exact $f$. Left column: $\Gamma_1$. Middle column: $\Gamma_2$. Right column: $\Gamma_3$.}
\label{fig:example4fig1}
\end{figure}


\vskip 0.2cm
\textbf{Example 5:} The last example is a discontinuous source function. Let
\[
f(x)=\chi_{(x_1^2+x_2^2<=0.81)}.
\]
The contour plots of the indicator functions by the DSM are shown in the first row of Fig.~\ref{fig:example5fig1} for $\Gamma_1$, $\Gamma_2$ and $\Gamma_3$. The radii of the reconstructed discs $\hat{B}$'s are $0.9849$, $1.2166$ and $1.1705$ listed in Table~\ref{table1}. The histograms of the coefficients are shown in the second row of Fig.~\ref{fig:example5fig1}. The reconstructed $f_{BE}$'s and the exact $f$ are shown in the third row of Fig.~\ref{fig:example5fig1}. The errors are listed in Table~\ref{table2}. The main features of the discontinuous source $f(x)$ such as the value and discontinuity are reconstructed well.

\begin{figure}[h!]
\begin{center}
\begin{tabular}{lll}
\resizebox{0.3\textwidth}{!}{\includegraphics{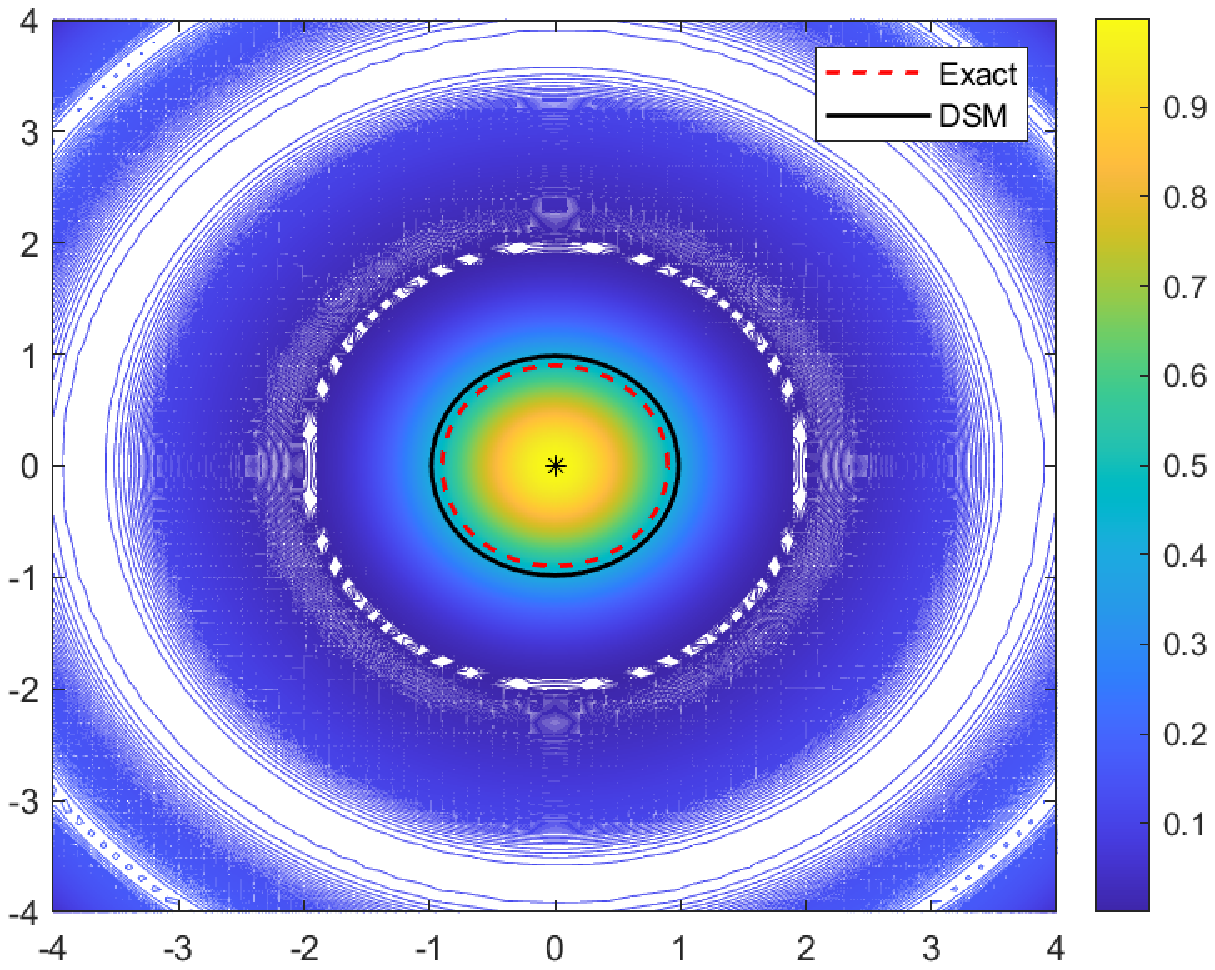}}&
\resizebox{0.3\textwidth}{!}{\includegraphics{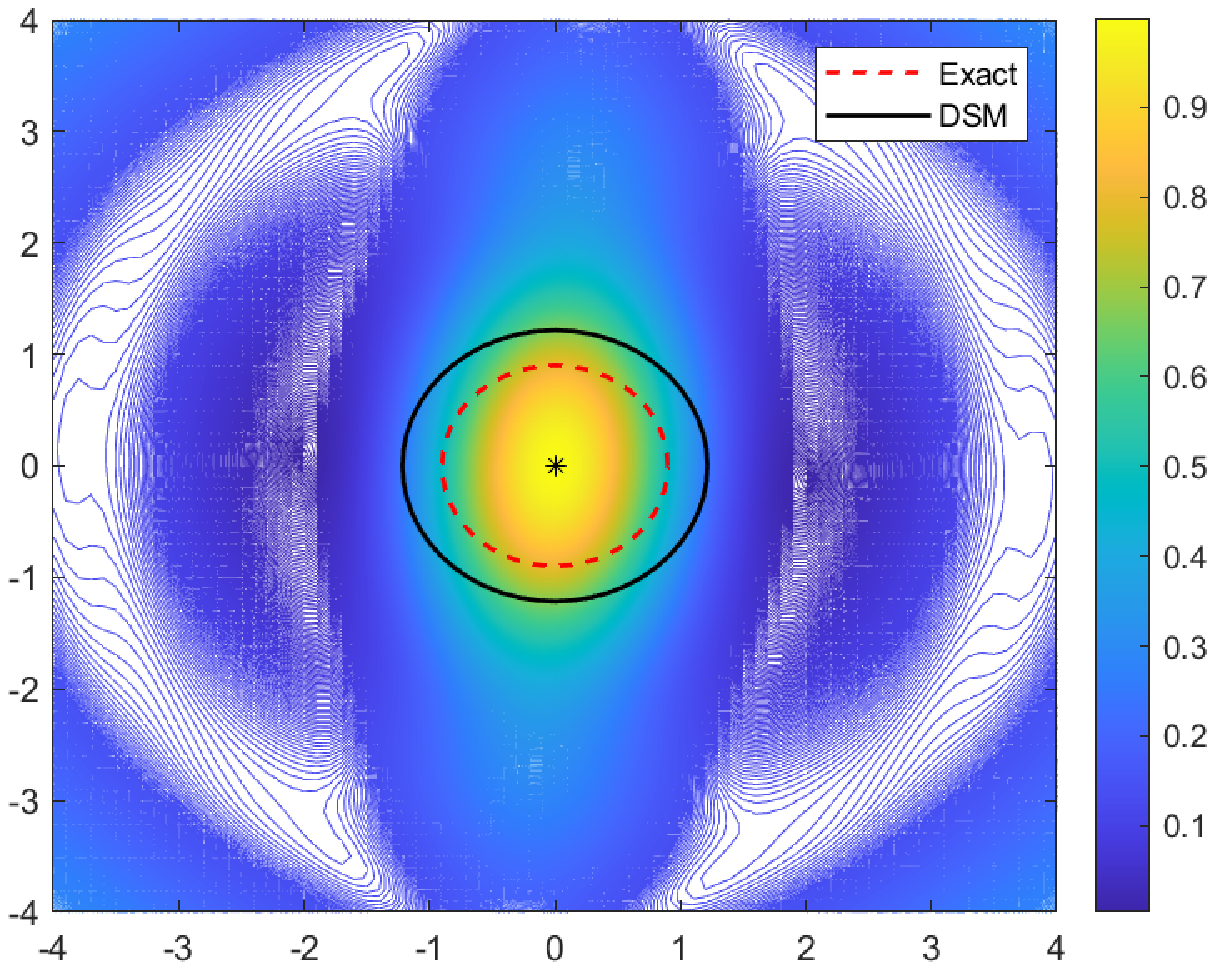}}&
\resizebox{0.3\textwidth}{!}{\includegraphics{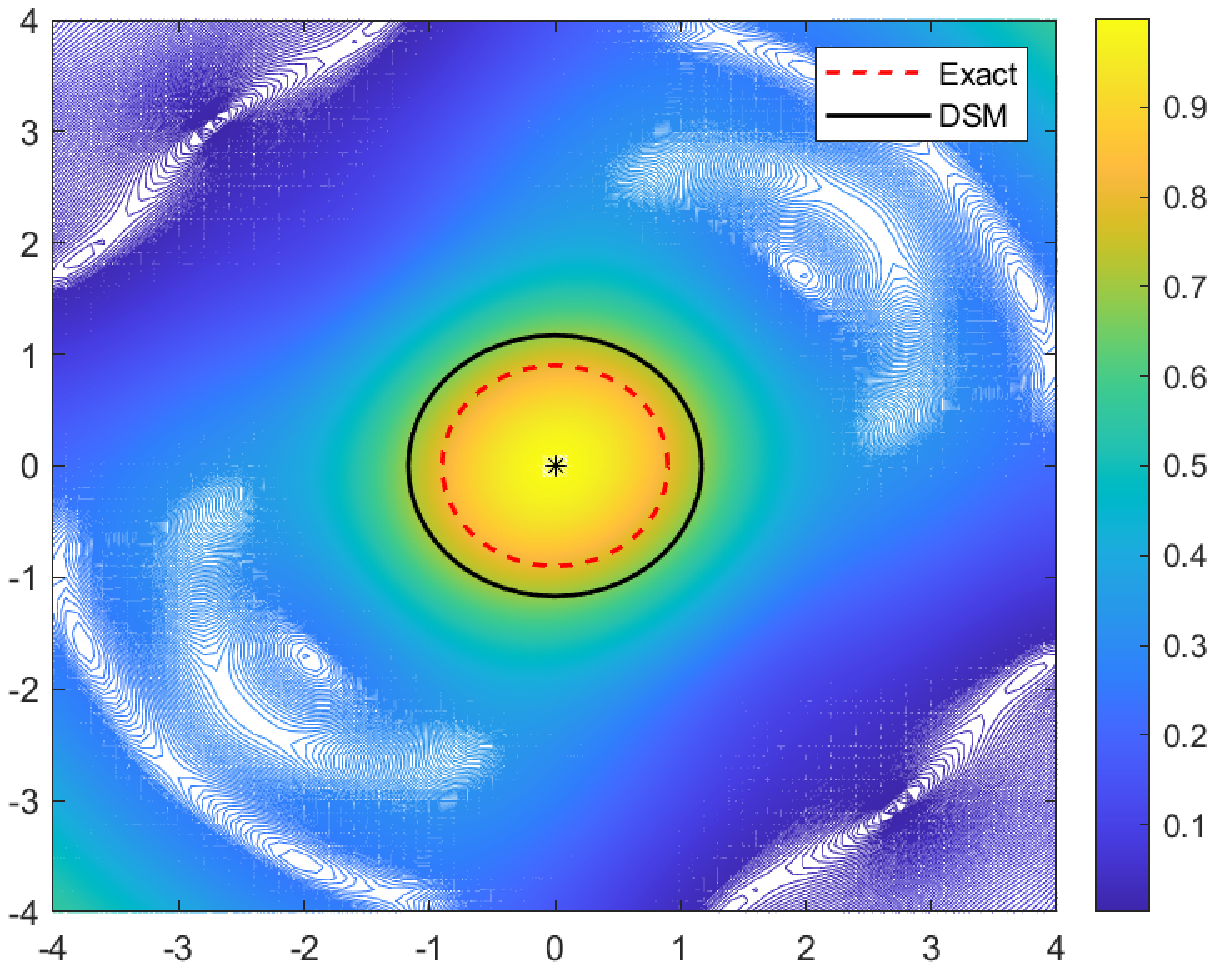}} \\
\resizebox{0.3\textwidth}{!}{\includegraphics{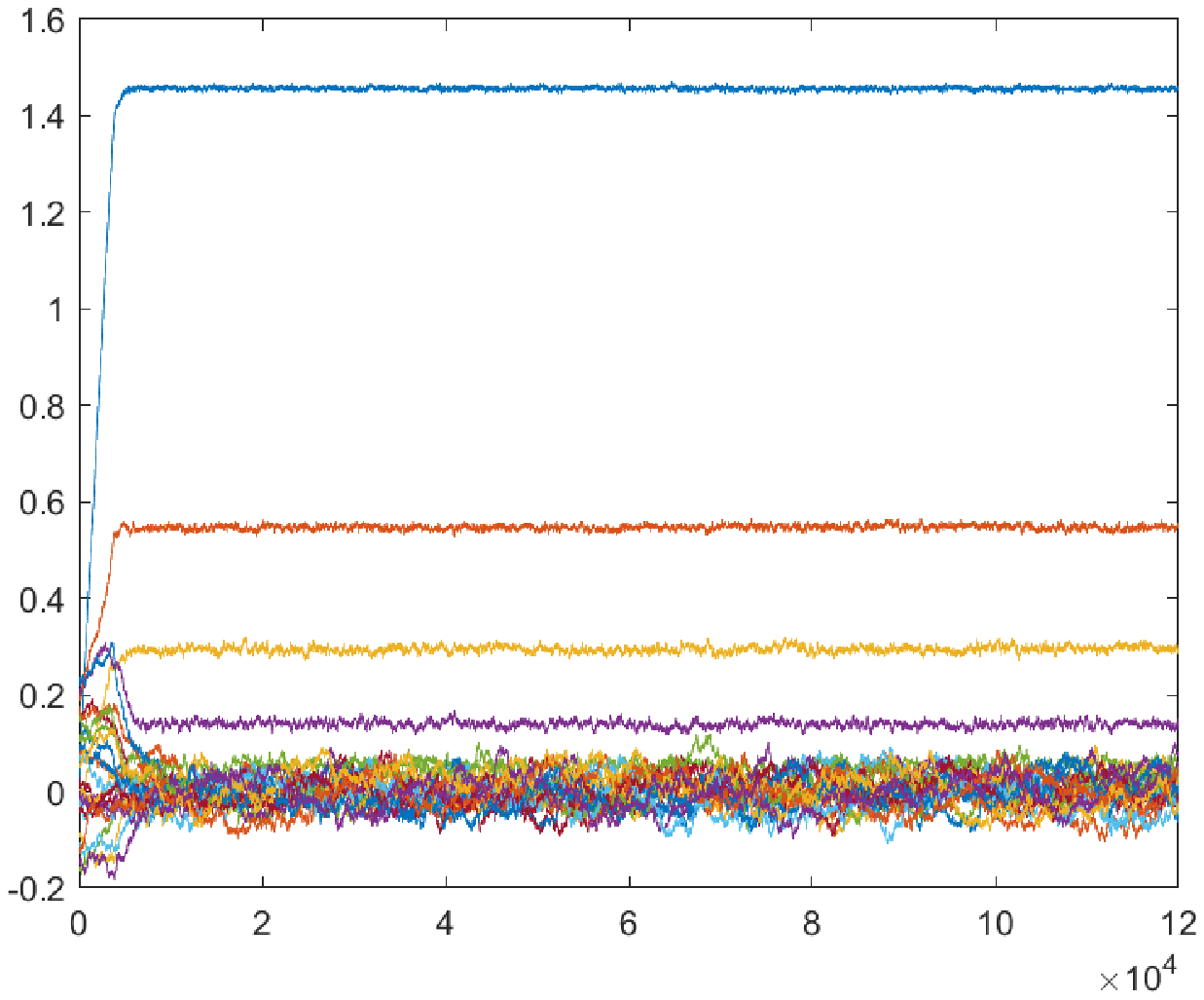}}&
\resizebox{0.3\textwidth}{!}{\includegraphics{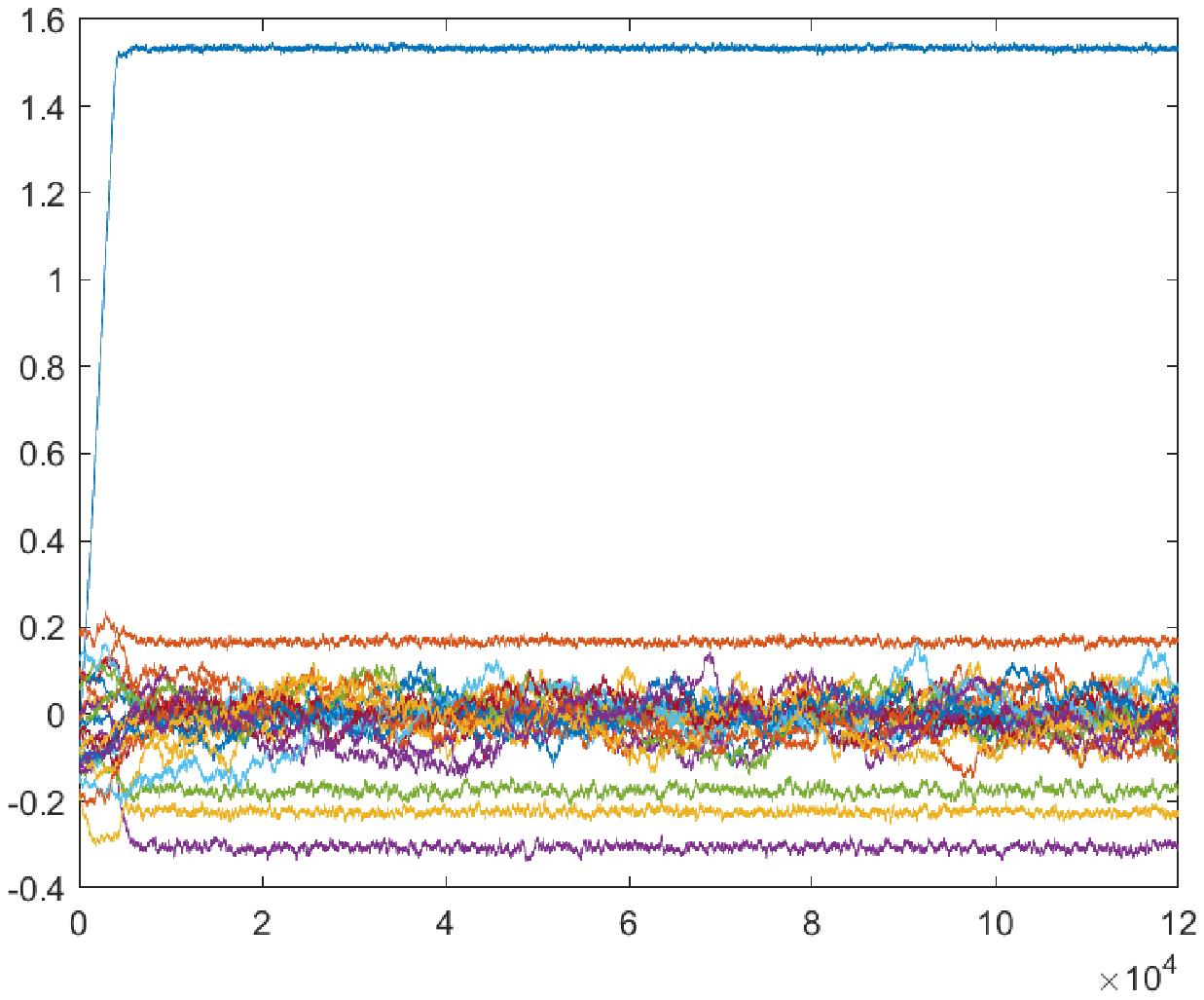}}&
\resizebox{0.3\textwidth}{!}{\includegraphics{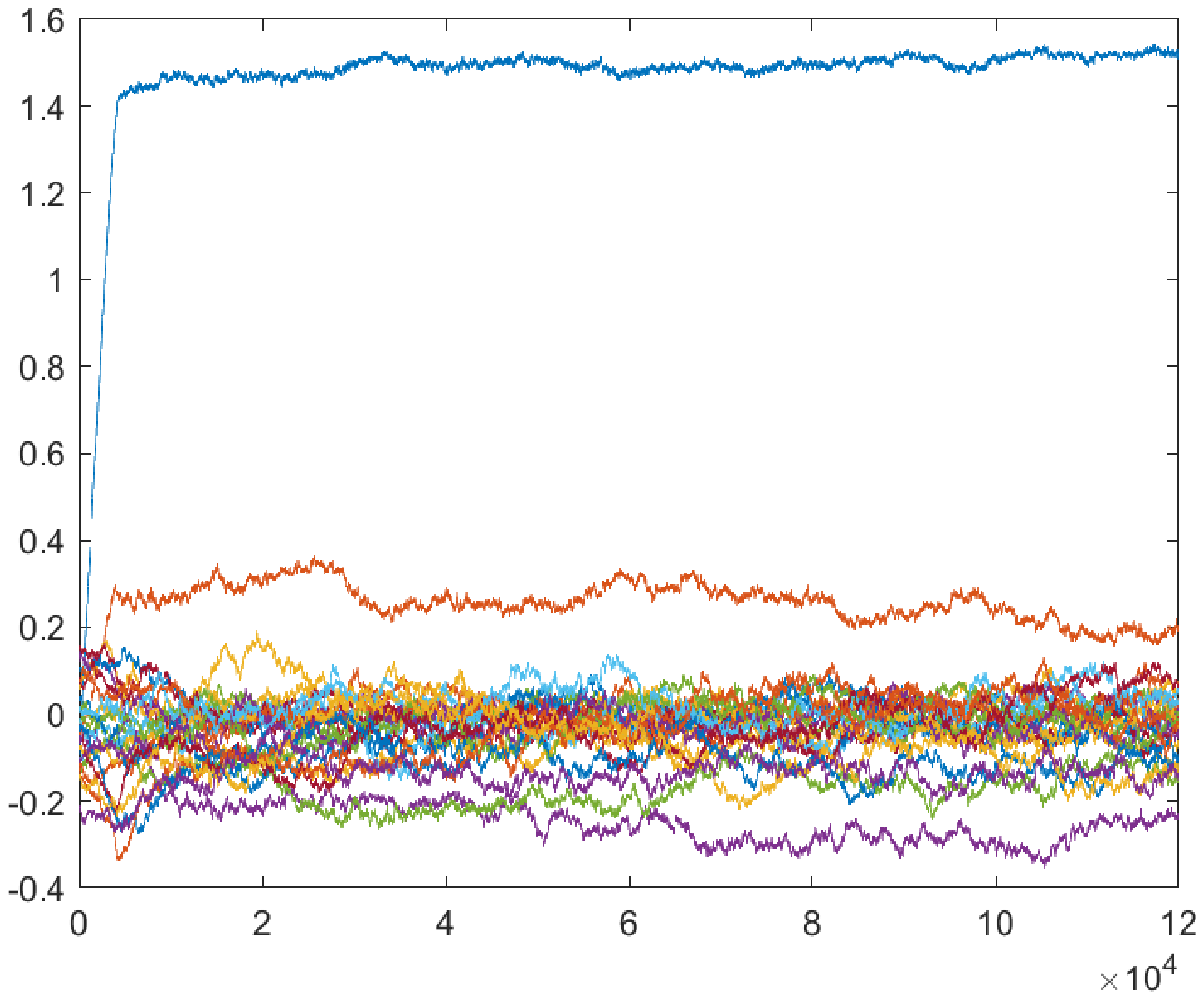}} \\
\resizebox{0.3\textwidth}{!}{\includegraphics{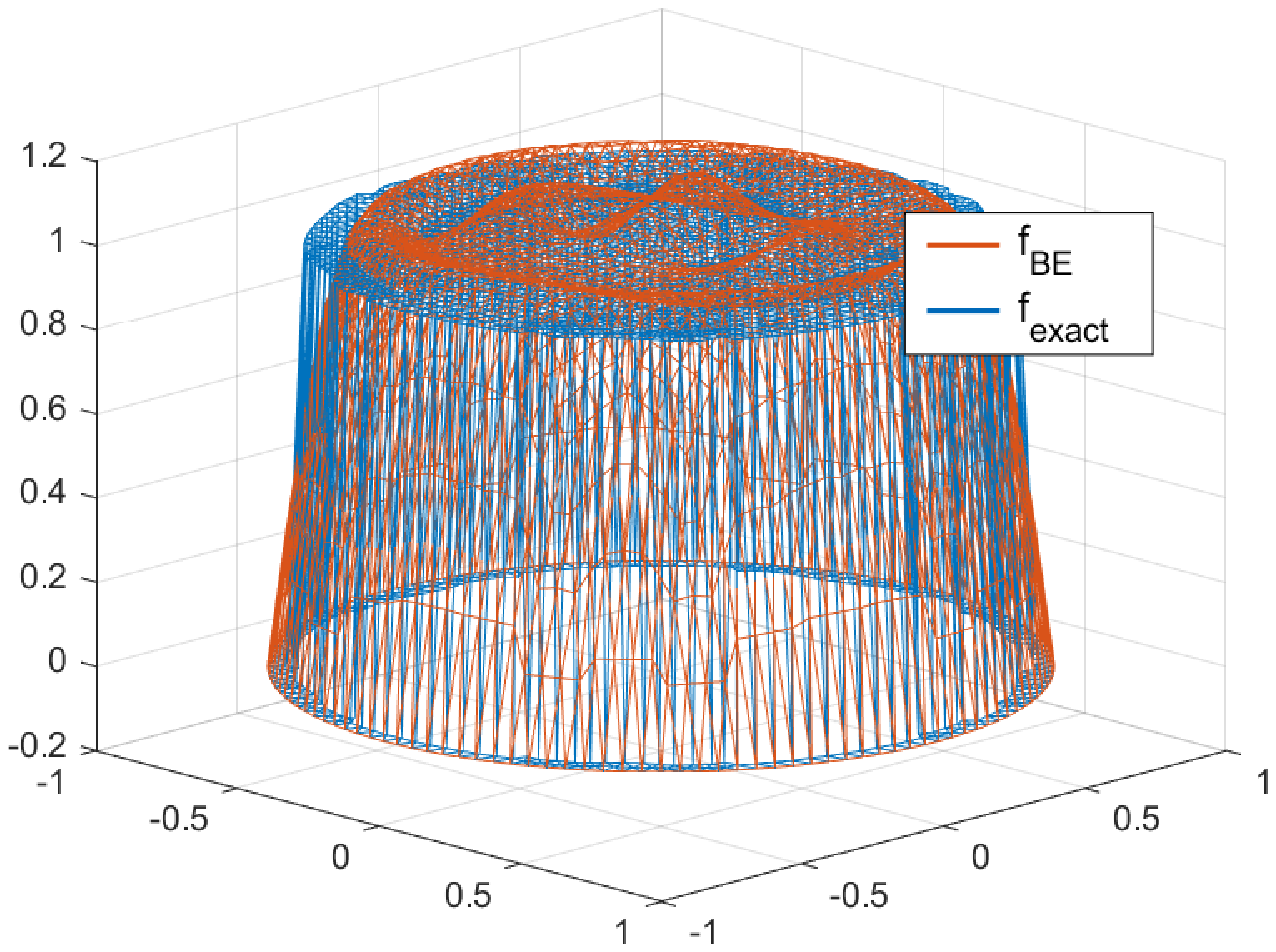}}&
\resizebox{0.3\textwidth}{!}{\includegraphics{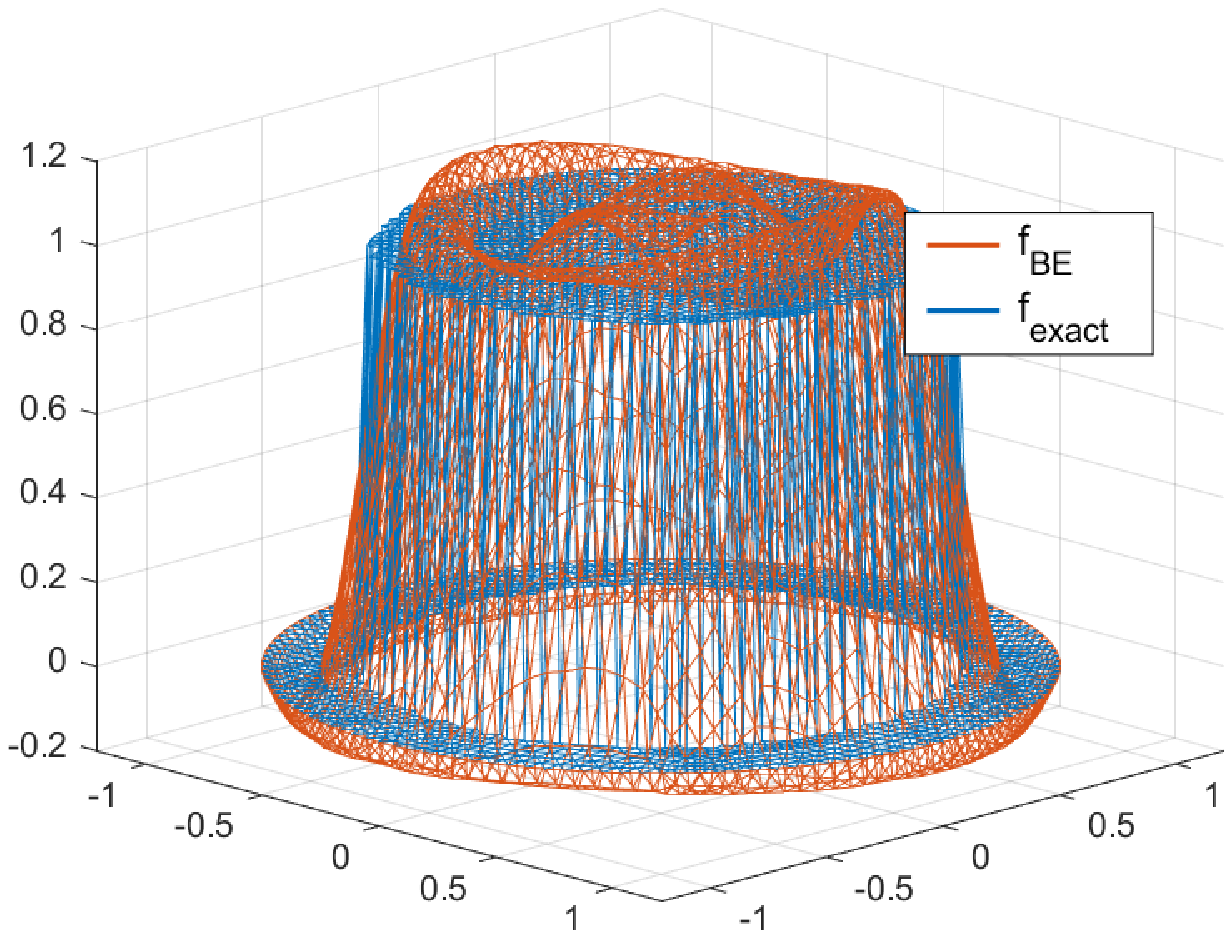}}&
\resizebox{0.3\textwidth}{!}{\includegraphics{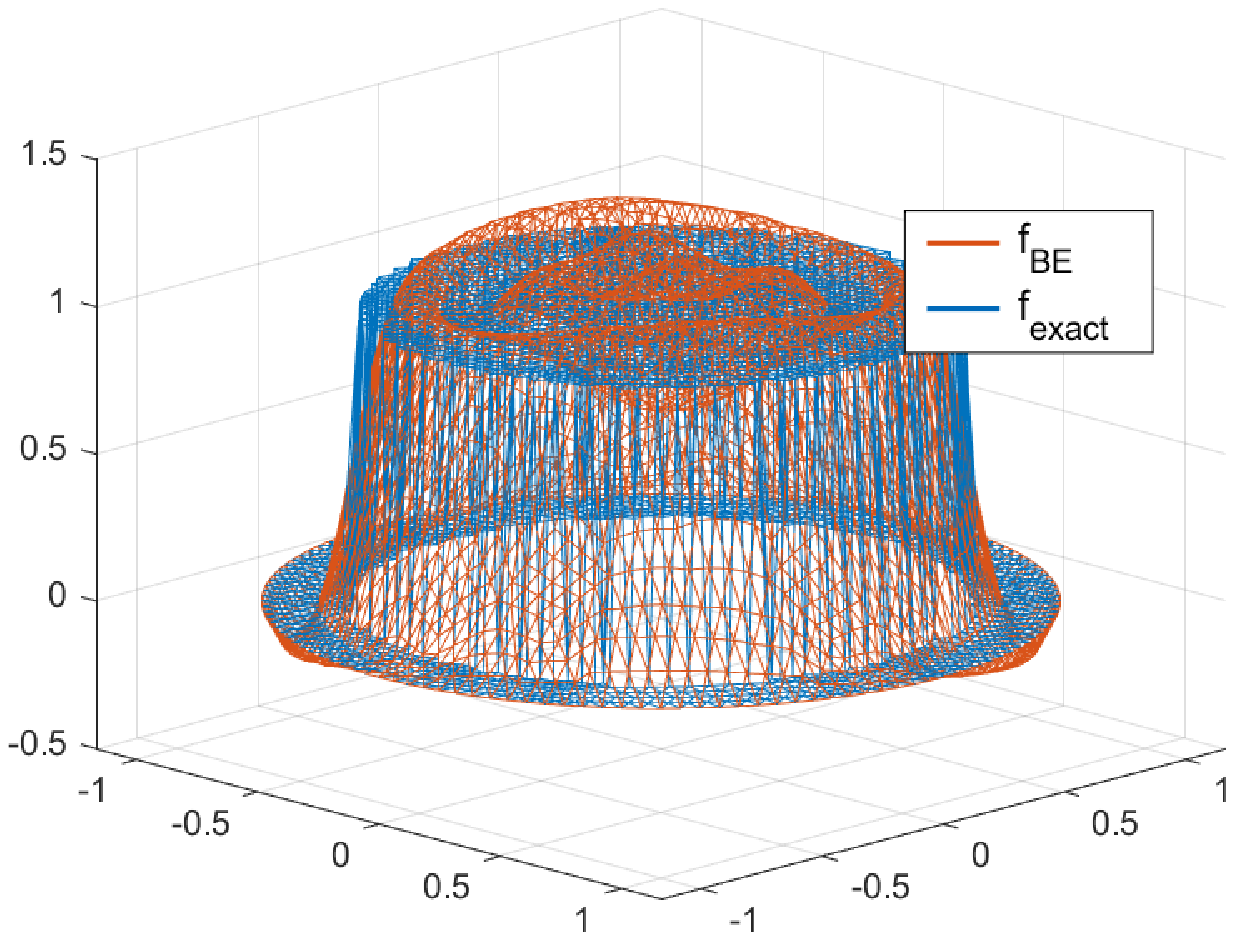}} \\
\end{tabular}
\end{center}
\caption{Example 5. First row: contour plots of the indicators for the DSM. Second row: the histograms of the coefficients for $f_{BE}$. Third row: the reconstructed $f_{BE}$ and exact $f$. Left column: $\Gamma_1$. Middle column: $\Gamma_2$. Right column: $\Gamma_3$.}
\label{fig:example5fig1}
\end{figure}

\section{Conclusions}
In this paper, we combine the DSM and Bayesian approach to reconstruct an extended source using the multiple frequency limited aperture far field data. In the first step, the DSM is used to obtain an approximation of the compact support (a disc) of the source. Using the eigenfunctions of the disc, we expand the source and employ the Bayesian inverse to recover the expansion coefficients.

Numerical examples, including a discontinuous source function, show the effectiveness of the proposed method. It is observed that as the aperture becomes smaller the reconstruction error increases. Nonetheless, the results are satisfactory for limited aperture data.

The cutoff value for the indicator function of the DSM is chosen by trial and error. We are investigating other methods to avoid choosing ad-hoc cutoff values. Algorithms that can improve the acceptance rate of the samplings in the MCMC method are also worth efforts to improve efficiency. Another interesting topic is the case when the source function is also frequency dependent, i.e, $f$ depends on $k$ as well.

\section*{Acknowledgements}
\noindent
The research of ZZ is supported by Hong Kong RGC grant (project 17307921), National Natural Science Foundation of China (project 12171406), and a seed funding from the HKU-TCL Joint Research Center for Artificial Intelligence.





\begin{thebibliography}{99}
    \bibitem{Devaney1982} A. Devaney and G. Sherman, Nonuniqueness in inverse source and scattering problems, IEEE Trans. Antennas Propag., 30 (1982), pp. 1034-1037.
    \bibitem{Bleistein1977} N. Bleistein and J. Cohen, Nonuniqueness in the inverse source problem in acoustics and electromagnetics, J. Math. Phys., 18 (1977), pp. 194-201.
    \bibitem{Griesmaier2017SIAMIS} R. Griesmaier and C. Schmiedecke, A factorization method for multifrequency inverse source problems with sparse far field measurements, SIAM J. Imaging Sci., 10 (2017), pp. 2119-2139.
    \bibitem{Liu2017IP} X. Liu,  A novel sampling method for multiple multiscale targets from scattering amplitudes at a fixed frequency. Inverse Problems 33 (2017), no. 8, 085011.
    \bibitem{Bao2015SIAMNA} G. Bao, S. Lu, W. Rundell, and B. Xu, A recursive algorithm for multifrequency acoustic inverse source problems. SIAM J. Numer. Anal. 53 (2015), no. 3, 1608-1628.
    \bibitem{BuiThanh2014} T. Bui-Thanh and O. Ghattas, An analysis of infinite dimensional Bayesian inverse shape acoustic scattering and its numerical approximation, SIAM/ASA J. Uncertain. Quantif., 2 (2014), pp. 203–222.
    \bibitem{LiYuan2017IPI} P. Li and G. Yuan,  Increasing stability for the inverse source scattering problem with multi-frequencies. Inverse Probl. Imaging 11 (2017), no. 4, 745-759.
    \bibitem{Bao2011CR} G. Bao, J. Lin, and F. Triki, An inverse source problem with multiple frequency data. C. R. Math. Acad. Sci. Paris 349 (2011), no. 15-16, 855-859.
    \bibitem{Davaney2007} A. Devaney, E. Marengo, and M. Li, The inverse source problem in nonhomogeneous background media SIAM J. Appl. Math. 67 (2007) 1353-78.
    \bibitem{Cotter2013} S.L. Cotter, G.O. Roberts, A. Stuart, and D.White, {\em MCMC methods for functions: modifying old algorithms to make them faster.}
	Statist. Sci. 28(2013), no. 3, 424-446.
    \bibitem{Arridge1999} S.R. Arridge,  Optical tomography in medical imaging Inverse Problems 15 (1999) 41-93.
    \bibitem{ELBadia2002} A. El Badia and T. Ha-Duong, On an inverse source problem for the heat equation. Application to a pollution detection problem.  J. Inverse Ill-Posed Probl. 10 (2002), no. 6, 585–599.
    \bibitem{Zhang2015} D. Zhang and Y. Guo, Fourier method for solving the multi-frequency inverse source problem for the Helmholtz equation. Inverse Problems 31 (2015), no. 3, 035007.
    \bibitem{Eller2009} M. Eller and N. Valdivia, Acoustic source identification using multiple frequency information. Inverse Problems 25 (2009), no. 11, 115005.
    \bibitem{Anastasio2007} M. Anastasio, J. Zhang, D. Modgil, and P. La Rivi\`{e}re,  Application of inverse source concepts to photoacoustic tomography. Inverse Problems 23 (2007), no. 6, S21-S35.
    \bibitem{Alzaalig2020IP} A. Alzaalig, G. Hu, X. Liu, and J. Sun, Fast acoustic source imaging using multi-frequency sparse data. Inverse Problems 36(2020), no. 2, 025009.
	\bibitem{colton2013inverse} D. Colton and R. Kress, Inverse acoustic and electromagnetic scattering theory. Third edition. Applied Mathematical Sciences, 93. Springer, New York, 2013.
	\bibitem{Fitzpatrick1991IP} B.G. Fitzpatrick, Bayesian analysis in inverse problems, Inverse Problems, 7 (1991), pp. 675–702.
	\bibitem{Isakov1990} V. Isakov, Inverse source problems. Mathematical Surveys and Monographs, 34. American Mathematical Society, Providence, RI, 1990.
	\bibitem{ito2012direct} K. Ito, B. Jin, and J. Zou, A direct sampling method to an inverse medium scattering problem,  Inverse Problems 28(2012), no. 2, 025003.
	\bibitem{bao2010multi}G. Bao, J. Lin,  F. Triki, A multi-frequency inverse source problem,  J. Differential Equations 249 (2010), no. 12, 3443–3465.
	\bibitem{LiDengSun2020} Z. Li, Z. Deng, and J. Sun, Extended-sampling-Bayesian method for limited aperture inverse scattering problems. SIAM J. Imaging Sci. 13 (2020), no. 1, 422–444.
	\bibitem{li2021quality} Z. Li, Y. Liu, J. Sun, and L. Xu, Quality-Bayesian approach to inverse acoustic source problems with partial data. SIAM J. Sci. Comput. 43 (2021), no. 2, A1062–A1080.
	\bibitem{LiuGuoSun2021} Y. Liu, Y. Guo, and J. Sun, A deterministic-statistical approach to reconstruct moving sources using sparse partial data,  Inverse Problems 37 (2021), no. 6, 065005.
	\bibitem{SunZhou2016} J. Sun and A. Zhou, Finite element methods for eigenvalue problems. Monographs and Research Notes in Mathematics. CRC Press, Boca Raton, FL, 2017.
	\bibitem{li2012direct} K. Ito, B. Jin, and J. Zou, A direct sampling method to an inverse medium scattering problem. Inverse Problems 28 (2012), no. 2, 025003
	\bibitem{kaipio2006statistical} J. Kaipio and E. Somersalo. Applied Mathematical Sciences, 160. Springer-Verlag, New York, 2005.
	\bibitem{Stuart2010} A. Stuart, Inverse problems: a Bayesian perspective. Acta Numer. 19 (2010), 451–559.
	\bibitem{Ito1993} K. Ito, ed. Encyclopedic dictionary of mathematics. Vol. 1. MIT press, 1993.
    \bibitem{Griesmaier2011IP} R. Griesmaier, Multi-frequency orthogonality sampling for inverse obstacle scattering problems. Inverse Problems 27 (2011), no. 8, 085005	
	\bibitem{LiuSun2018} J. Liu and J. Sun, Extended sampling method in inverse scattering. Inverse Problems 34 (2018), no. 8, 085007, 17 pp.
	\bibitem{WangMaZheng2015} Y. Wang, F. Ma, and E. Zheng, Bayesian method for shape reconstruction in the inverse interior scattering problem, Math. Probl. Eng., 2015 (2015), 935294.
	\bibitem{Yang2020IP} Z. Yang, X. Gui, J. Ming, and G. Hu, Bayesian approach to inverse time-harmonic acoustic scattering with phaseless far-field data. Inverse Problems 36 (2020), no. 6, 065012
\end{thebibliography}
\end{document}